\documentclass[runningheads,12pt]{llncs}
\usepackage{geometry}
\usepackage{wrapfig,booktabs}
\usepackage{amssymb}
\usepackage{amsmath}
\usepackage{graphicx}
\usepackage{subfigmat}
\usepackage{url}
\usepackage{amsfonts}
\usepackage{latexsym} 
\usepackage{wasysym}
\usepackage{newlfont}
\usepackage{leqno}
\usepackage{etoolbox}

\usepackage{lineno}
\usepackage[usenames]{color}
\usepackage{fancybox}
\usepackage{sidecap}
\usepackage{xcolor, soul}
\definecolor{airforceblue}{rgb}{0.90, 0.44, 0.86}
\definecolor{ashgrey}{rgb}{0.7, 0.75, 0.71}
\sethlcolor{airforceblue}
\usepackage{multirow}
\usepackage[printonlyused]{acronym}
\newcommand*{\BLACKBOX}{\hfill\ensuremath{\blacksquare}}

\usepackage{stmaryrd}

\usepackage{xifthen}
\newcommand{\sMP}[1][]{\ifthenelse{\isempty{#1}}{^{(i)}}{^{(#1)}}}

\everymath{\displaystyle}
{\begin{Sbox}\begin{minipage}}%
{\end{minipage}\end{Sbox}\fbox{\TheSbox}}

\newtheorem{assume}{Assumption}

\numberwithin{equation}{section}
    \urldef{\mailsd}\path|ioannis.toulopoulos@ricam.oeaw.ac.at|
    \urldef{\mailse}\path|christoph.hofer@jku.at|
    \newcommand{\keywords}[1]{\par\addvspace\baselineskip  
     \noindent\keywordname\enspace\ignorespaces#1}

\begin{document}
\mainmatter
\title{  Discontinuous Galerkin  Isogeometric Analysis for 
segmentations generating  overlapping regions.
          }
\titlerunning{DG-IGA on multipatch representations with  overlapping regions. }
\author{
Christoph Hofer$^1$
\and Ioannis Toulopoulos$^{2}$
}
\authorrunning{C. Hofer,   I. Toulopoulos }

\institute{ $1$ Johannes Kepler University (JKU),\\
				Altenbergerstr. 69, A-4040 Linz, Austria,\\[1ex]
	    $2$ Johann Radon Institute for Computational and Applied Mathematics (RICAM),\\
				Austrian Academy of Sciences\\
				Altenbergerstr. 69, A-4040 Linz, Austria,\\[1ex]
\mailse\\
\mailsd  
 }

\noindent
\maketitle
\begin{abstract}
\textcolor{black}
{
In the Isogeometric Analysis (IGA) framework, the computational domain has  very often a multipatch  
representation.  The multipatch domain can be obtained by a volume segmentation of a boundary represented domain, e.g.,  provided by a Computer Aided Design (CAD) model. 
Typically, small gap and overlapping regions can appear at the patch interfaces of such multipatch representations. 
In the current work we consider multipatch representations having only  small overlapping regions between the patches. 
  We develop a Discontinuous Galerkin (DG)- IGA method which can be immediately   applied to these representations. 
Our method  appropriately connects the fluxes of the one face of the overlapping region with the flux of the opposite  face. 
We provide a theoretical justification of our approach by splitting   the whole  
error into two components: the first is related to the incorrect representation  of the patches  ({consistency error}) and the  
second  to the approximation properties of the  IGA space.
We show bounds for both  components of the error. We verify the theoretical error estimates in a series of 
numerical examples. 
}
\end{abstract}

\section{Introduction}
Isogeometric Analysis (IGA) has been introduced in \cite{LT:HughesCottrellBazilevs:2005a} as a new methodology for solving numerically Partial Differential Equations (PDE). 
The key idea of the IGA concept
is to use  the superior finite dimensional spaces, which are used
in Computer Aided Design (CAD), e.g., B-splines, NURBS, for both the exact representation of the computational domain $\Omega$ and 
discretizing the PDE problem. Since this work, many applications of the IGA methodology to several fields have been discussed 
in several papers, see, e.g.,   the monograph   \cite{LT:Hughes_IGAbook_2009} and the references within, as well as
the survey paper \cite{VeigaBuffaSangalli_2014}. From a computational point of view,  we can say that 
the numerical algorithm for constructing  the B-spline (or NURBS) basis functions is quite simple. This 
 helps to produce  high order approximate solutions. 
From the theoretical point of view, the fundamental approximation properties of the B- spline spaces 
 on a reference domain are discussed in \cite{LT:Shumaker_Bspline_book}. The approximation properties of the 
 mapped B-spline (or NURBS) spaces are discussed in several papers, 
see e.g.,  \cite{LT:Bazilevs_IGA_ERR_ESti2006a}, \cite{TagliabueQuart_2014}, \cite{VeigaBuffaSangalli_2014}, \cite{LT:LangerToulopoulos:2014a}.
 \par
 \textcolor{black}
 {
Let us consider  a complex domain $\Omega$ where its boundary is prescribed by CAD models. 
The CAD models can not be directly used  in IGA in order  to discretise the PDE problems.  We need to create volumetric patch parametrizations from the CAD models.  
  The  boundary represented  domain is first segmented into a collection of suitable blocks and consequently a parametrization procedure is applied to each block. 
  This produces  
 the volumetric multipatch representation  $\cup_{i=1}^N \overline{\Omega_i}$ of $\overline{\Omega}$ suitable for IGA.  
Several  segmentation algorithms  and 
associated parametrization procedures have been discussed in the literature, see, e.g., \cite{Hoschek_Lasser_CAD_book_1993}, \cite{HLT:PauleyNguyenMayerSpehWeegerJuettler:2015a},\cite{HLT:JuettlerKaplNguyenPanPauley:2014a}, \cite{xu_mourain_duvingeau_Galligo_CAD_2013a,xu_mourain_duvingeau_Galligo_JCP_2013b}. 
Furthermore, we refer to  \cite{xu_mourain_bordas_CMAME_2018a}, \cite{Jutler_Falini_Spech_2015} and   \cite{BUCHEGGER_Jutler2017a} for different approaches for constructing   IGA planar parametrizations without 
utilizing segmentation algorithms. 
We mention the segmentation approach presented in \cite{HTL:NguyenPauleyJuettler:2014a}, and  \cite{HLT:JuettlerNguyenPauley:2016a}, from which, we have been motivated to present the current work. 
The main idea   is to split the given boundary represented domain, using a spline curve (or face in 3d case) with 
 the following properties: (i) must have the end points on the boundary and the 
 tangents to be specified, (ii) the curve is reasonably regular and does not intersect the boundary of the domain, (iii) the curve cuts the domain into new   subdomains with good shapes. Consequently,  tensor-product B-spline spaces are fitted
 in the collection of  the subdomains for defining the tensor product B-spline surfaces or volumes \cite{Jutler_Falini_Spech_2015}. 
Note that the previous consideration is also concerns   CAD models that are connected along  a non-matching interface. 
 It is  important  to obtain a  curve that  splits 
$\Omega$ into new simple domains with good shapes being suitable for IGA.
During the computation of the multipatch representation,  errors can occur when defining  the corresponding control points, see \cite{HLT:PauleyNguyenMayerSpehWeegerJuettler:2015a}, \cite{HTL:NguyenPauleyJuettler:2014a} and \cite{Jutler_Falini_Spech_2015}. 
A consequence of this is a non-conforming parametrizations of the 
patches in the sense that the images of the patch interfaces under the parametrizations   are not identical. 
This in turn leads to the existence of gap and/or overlapping regions between the adjoining patches, 
see a schematic illustration in  Fig.~\ref{Fig_ContolNet_Overl}(b).  
}
\par
\textcolor{black}{
This paper considers   the case where there are only overlapping regions  between the patches.
If we apply an  IGA methodology to this multipatch representation, a direct consequence
is that the whole discretization error will include two (main) parts: the first naturally comes  from the approximation properties of the B-spline spaces (for the purposes of this work  we use B-spline spaces)
and, the second comes from the  geometric error. The later 
is  due to the incorrect parametrization of the patch interfaces. 
Furthermore, the geometric error can be characterized as a consistency error, which 
consists of two error components. 
The first error component is related to the approximation of the jumps of the flux of the solution on the non-matching interfaces. The second component is related to the existence of more than one numerical solution in the 
overlapping regions.    
\\
The contribution of this paper is to develop a DG-IGA  method which can be applied on volumetric patch representations with
non-matching interface parametrizations. We present our methodology   for 
discretizing the following elliptic 
Dirichlet boundary value problem 
\begin{align}\label{0}
-\mathrm{div}(\rho \nabla u) = f \; \text{in} \; \Omega
\quad \mbox{and}	\quad 
u  := u_D=0   \; \text{on} \;
\partial \Omega,
\end{align}
where the diffusion coefficient $\rho(x)$ can be discontinuous 
across a smooth internal interface. 
We derive bounds for the two main parts of the whole error.  In our analysis,  
we  derive separate bounds for the two components of the geometric error. 
To the best of our knowledge, we believe this is an new area of analysis to be investigated. Our current work is the first step in the analysis, where we are developing our methodology
for the numerical solution of the simple stationary diffusion problem  (\ref{0}).
Our intention for future works is to extend the current methodology to more complicated time dependent problems, 
where the interface can move with time, cf.  \cite{Bazilevs_FSI_2013}. 
}
\par
{
Due to the non-matching interior patch interfaces, a direct application of the classical DG numerical fluxes
proposed in literature, see e.g.   \cite{LT:LangerToulopoulos:2014a}, \cite{Nguyen_Nitche_3Dcoupli_2014},  is not possible, as these fluxes are only applicable for matching interface parametrizations.  
In our recent papers, \cite{HoferLangerToulopoulos_2016_SISC} and \cite{HoferToulopoulos_IGA_Gaps_2015a}, we  developed
DG-IGA  schemes for  multipatch unions that include  only gap regions.
In particular,  we  considered the  PDE model given  in (\ref{0}) and we 
denoted by $d_g$ the maximum distance between the diametrically opposite points located on the gap boundary.  
We applied Taylor expansions using the diametrically opposite points of the gap, in order to give estimates 
 for the jumps of the normal fluxes  with respect to $d_g$.  
Finally, we used the same Taylor expansions in the DG-IGA scheme 
 for constructing suitable DG numerical fluxes across the gap boundary 
 that help  on the weakly coupling of the local patch-wise discrete problems.
 We developed a discretization error analysis and  showed  a priori estimates in the DG-norm, expressed in terms of the mesh size and the gap width, 
 i.e., $\mathcal{O}(h^r)+\mathcal{O}(d_g)$, where $r$ depends on
 the B-spline degree $p$ and  the regularity of the solution.   
  In  \cite{HoferLangerToulopoulos_2016_SISC} and \cite{HoferToulopoulos_IGA_Gaps_2015a}, we have shown that,
 if $d_g=\mathcal{O}(h^{p+\frac{1}{2}})$, the proposed DG-IGA scheme has optimal approximation properties. 
 }
 \par
 \textcolor{black}{
In this paper, we extend the previous work  to multipatch unions with  overlapping regions. 
 In the analysis presented in  \cite{HoferLangerToulopoulos_2016_SISC} and \cite{HoferToulopoulos_IGA_Gaps_2015a}, the whole geometric error does not include the  component coming from  the co-existence of different IGA solutions in the overlapping regions. Here, the new approach is to introduce local (patch-wise) auxiliary  variational problems, which are 
  compatible with the overlapping nature of the multipatch representation of $\Omega$. 
  We denote the solutions of the new variational problems  by $u^*$.
  These problems are not consistent,  in the sense that   the original  solution $u$ of  (\ref{0})
  does not satisfy them.   
 Following the IGA concept, the B-spline spaces used for the  parametrization of the patches 
 are also used for discretizing the local auxiliary  problems. We denote by  $u_h^*$ the produced IGA solutions. 
 Under some regularity assumptions on $u^*$, we can expect (see  Section 3)  that the IGA solution $u_h^*$  has optimal approximation properties    associated with  $u^*$. However,  we can not directly infer that $u_h^*$  can approximate in an optimal way 
  the  solution $u$ of the original     problem. 
  In our analysis,  we  provide an estimate for  the consistency error $u-u^*$ and consequently 
  using  the triangle inequality $\|u-u_h^*\|_{DG} \leq \|u^*-u_h^*\|_{DG}+\|u-u^*\|_{DG}$, we can derive an estimate
  for the error between the exact solution $u$ and the IGA solution $u_h^*$. The  mesh-dependent norm
  $\|\cdot\|_{DG}$ is  defined in Section 2. 
  We  give  error estimates  for both terms $\|u^*-u_h^*\|_{DG}$ and $\|u-u^*\|_{DG}$
  expressed in terms of the mesh size $h$ and the quantity $d_o$, which
  is introduced in our analysis in order to quantify the width of the overlapping regions.  
  In particular, we  show that under appropriate assumptions on the data 
  and for the  case where $d_o$ is of order $\textstyle{h^{\lambda},\,\lambda \geq p+\frac{1}{2}}$, 
  the proposed DG-IGA scheme 
  has optimal convergence properties. This convergence  result is similar to the result in 
  \cite{HoferLangerToulopoulos_2016_SISC} and \cite{HoferToulopoulos_IGA_Gaps_2015a}. }
   \par
  \textcolor{black}{  
   In a future work,  we apply the same approach to solve problems on multipatch partitions,
   which can include gap and overlapping regions. 
   We present numerical solutions in multipatch unions with more complicated gaps and overlapping regions. 
   We also provide details  related to the implementation of the proposed DG-IGA scheme. In the same work, we also discuss issues related to the 
   construction  of domain decomposition methods on these multipatch representations  and provide several numerical tests 
   for evaluating their performance. The first results in this direction can be found in \cite{Report_Hofer_Langer_ToulopoulosDGIGA_GapOver2016}.
   }
  \par
     We note that IGA multipatch representations with non-matching interfaces meshes, overlapping regions 
   and  trimmed patches have been considered in many publications. For the communication of the 
   discrete patch-wise problems, several
   Nitsche's type coupling  methods involving normal flux terms have been   applied across the interfaces, 
   see e.g., 
   \cite{Ruess_NitcCoplPtac_2015},\cite{Nguyen_Nitche_3Dcoupli_2014},\cite{Apostolatos_Schmidt_Wuchner_Bletzinger_IJNUMEng_2014},\cite{Hughe_Bazilevs_WeaklyBC_2007} 
   and  references therein. We mention also that in \cite{ZHANG_DG_IGA_Overl_2017}, DG-IGA methods 
   have been presented to discretize Laplace problems on multipatch unions with large overlapping regions. 
   The proposed strategy follows the additive Schwartz methodology. 
   To the knowledge of the authors, there are no
   works  that analytically discuss estimates for the error, which is  caused by the incorrect representation of the shape
   of the patches. The purpose of this work is to present  such an error analysis.
\par 
The structure of the paper is as follows: Section 2 presents the PDE model, briefly reviews the B-spline spaces and
describes the case of having non-matching parametrized interfaces with overlapping regions.
Section 3, presents in detail the perturbation problems, the bounds for the consistency error,
the proposed DG-IGA scheme and the error analysis. Section 4, includes several numerical examples that confirm the theoretical estimates. The paper closes with the Conclusions. 
\section{The model problem}
\subsection{Preliminaries}
Let $\Omega$ be a bounded Lipschitz domain in $\mathbb{R}^d,\,d=2,3$, 
and let $\boldsymbol{\alpha}=(\alpha_1,\ldots,\alpha_d)$ be a multi-index of non-negative integers 
$\alpha_1,\ldots,\alpha_d$
with degree $\textstyle{ |\boldsymbol{\alpha}| = \sum_{j=1}^d\alpha_j}$. 
For any  $\boldsymbol{\alpha}$,  we define the differential operator
$D^{\boldsymbol{\alpha}}=D_1^{\alpha_1}\ldots D_d^{\alpha_d}$,
with $D_j = \partial / \partial x_j$, $j=1,\ldots,d$, and $D^{(0,\ldots,0)}\phi=\phi$.
For a  non-negative integer $m$, let $C^{m}(\Omega)$ denote the space
of all functions $\phi:\Omega \rightarrow \mathbb{R}$, whose 
partial derivatives
$D^{\boldsymbol{\alpha}} \phi$ of all orders $|\boldsymbol{\alpha}| \leq m$ are continuous in $\Omega$. 
Let  $\ell$ be a non-negative integer. 
 As usual, 
 $L^2(\Omega)$ denotes the Sobolev space for which 
 $\textstyle{\int_{\Omega}|\phi(x)|^2\,dx < \infty}$, endowed with the norm
 $\textstyle{\|\phi\|_{L^2(\Omega)} = \big(\int_{\Omega}|\phi(x)|^2\,dx\big)^{\frac{1}{2}}}$, 
 and  $L^{\infty}(\Omega)$ denotes the functions that are essentially bounded. Also 
 \begin{equation*}
 	H^{\ell}(\Omega)=\{\phi\in L^2(\Omega): D^{\boldsymbol{\alpha}} \phi \in L^{2}(\Omega),\,\text{for all}\,|\alpha| \leq \ell \},
 \end{equation*}
denote the standard Sobolev spaces endowed with the following  norms 
 \begin{equation*}
 \|\phi\|_{H^{\ell}(\Omega)} = \big(\sum_{0\leq |\boldsymbol{\alpha}| \leq \ell} \|D^{\boldsymbol{\alpha}}\phi\|_{L^2(\Omega)}^2\big)^{\frac{1}{2}}.
 \end{equation*}
 We identify $L^2$ and $H^{0}$ and also define the subspace $H^{1}_0(\Omega)$ and $H^{1}_{\Gamma}(\Omega)$ of $H^{1}(\Omega)$ 
 \begin{align*}
 \hskip -0.2cm {
 	H^{1}_0(\Omega) =\{\phi\in H^{1}(\Omega):\phi=0\,\text{on}\, \partial \Omega \}},\quad
 	H^{1}_{\Gamma}(\Omega)=\{\phi\in H^{1}(\Omega):\phi=0\,\text{on}\,
 	\textcolor{black}{\Gamma\subset \partial \Omega,\, |\Gamma| >0.} \}.
 \end{align*}

We recall H\"older's and Young's inequalities 
\begin{equation}\label{HolderYoung}
  \left|  \int_{\Omega}\phi_1 \phi_2\,dx \right| \leq  \|\phi_1\|_{L^2(\Omega)}\|\phi_2\|_{L^2(\Omega)}
\quad \mbox{and} \quad
\left|  \int_{\Omega}\phi_1 \phi_2\,dx \right|\leq \frac{\epsilon}{2}\|\phi_1\|^2_{L^2(\Omega)}+  
 \frac{1}{2\epsilon}\|\phi_2\|^2_{L^2(\Omega)},
\end{equation}
that hold for all $\phi_1\in L^2(\Omega)$ and $\phi_2\in L^2(\Omega)$ 
and for any fixed $\epsilon \in (0,\infty)$. In addition, we  recall  trace and Poincare's inequalities, \cite{Evans_PDEbook},
\begin{equation}\label{Poincare_trace}
\begin{split}
 \|\phi\|^2_{L^2(\partial \Omega)} \leq& C_{tr} \|\phi\|_{L^2(\Omega)} \|\phi\|_{H^1( \Omega)},\\
 \|\phi\|_{L^2(\Omega)}\leq& \text{meas}_{\mathbb{R}^d}(\Omega) \,\|\nabla \phi\|_{L^2(\Omega)},\quad
 \text{for}\quad \phi \in H^{1}_{\Gamma}(\Omega).
 \end{split}
\end{equation}

\subsection{The elliptic diffusion problem}
The weak formulation of the boundary value problem (\ref{0}) reads as follows:
for given source function $f \in L^2(\Omega)$ 
find a function $u\in H_0^1(\Omega)$ such that 
 the variational identity
\begin{equation}
\label{4a}
a(u,\phi)=l_f(\phi),\;  \forall \phi \in 
H^{1}_{0}(\Omega), 
\end{equation}
is satisfied,
where the bilinear form $a(\cdot,\cdot)$ and the linear form $l_f(\cdot)$
are defined by 
\begin{equation}
\label{4b}
a(u,\phi)=\int_{\Omega}\rho\nabla u \cdot \nabla\phi\,dx
\quad \mbox{and}\quad
l_f(\phi)=\int_{\Omega}f\phi\,dx,
\end{equation}
respectively. 
{\color{black}
The given diffusion coefficient $\rho \in L^{\infty}(\Omega)$ is assumed to be
uniformly positive
and piece-wise (patch-wise, see below) constant. 
These assumptions ensure existence and uniqueness of the solution
due to Lax-Milgram's lemma.
}
For simplicity, we only consider pure
Dirichlet boundary conditions on $\partial \Omega$.
However, the analysis presented in our paper can easily be generalized to 
other constellations of boundary conditions which ensure existence and uniqueness
such as Robin or mixed boundary conditions.
\\
In what follows, positive constants $c$ and $C$ appearing in  inequalities are 
generic constants which do not depend on the mesh-size $h$. In many cases,  
we will indicate on what may the constants depend on.  
Frequently, we will write $a\sim b$ meaning that $c \, a\leq b \leq C \, a$.

\subsection{B-spline spaces}
\label{Bsplinespace}
In this section, we briefly present the B-spline spaces and the form of the B-spline parametrizations
for  the physical  subdomains. 
For a better presentation of the B-spline spaces, we start our discussion for the one-dimensional case. Then
we proceed to higher dimensions. 
We refer to \cite{LT:Hughes_IGAbook_2009}, 
\cite{CarlDeBoor_Splines_2001} and \cite{LT:Shumaker_Bspline_book} 
for a more detailed presentation. 
\par
Consider,   \textcolor{black}{$\mathcal{Z}=\{0=z_1 < z_2 < \ldots <z_{M}=1\}$} to be a partition 
 of $\bar{I}=[0,1]$ with $\bar{I}_j=[z_j,z_{j+1}],\, j=1,\ldots,M-1$ to be the intervals
 of the partition.  
Let the integers $p$ and $n_1$ denote the $p$ spline degree and the number of the B-spline basis functions.
 \textcolor{black}{
Based on $\mathcal{Z}$, we introduce the open knot vector  $\Xi=\{0=\xi_1,  \xi_2, \ldots,  \xi_{n_1+p+1}=1\}$, 
and the associated vector $\mathcal{M}=\{m_{1},\ldots,m_{M}\}$  of knot multiplicities with $m_1=m_M=p+1$, i.e., }
 \begin{align}\label{rep_xi_i}
 \Xi=\{ \underbrace{0=\xi_1,\ldots  ,\xi_{m_1}}_{=z_1}, 
  \underbrace{\xi_{m_1+1}=\ldots=\xi_{m_1+m_2}}_{=z_2},\ldots,
 \underbrace{\xi_{n_1+p+1-m_M},\ldots,\xi_{n_1+p+1}=1}_{=z_M}\}. 
 \end{align}
The B-spline basis functions are defined by the Cox-de Boor formula, see, e.g.,  \cite{LT:Hughes_IGAbook_2009} and \cite{CarlDeBoor_Splines_2001},  
\begin{align}\label{CoxBoor}
B_{i,p}=&\frac{x-\xi_i}{\xi_{i+p}-\xi_i}B_{i,p-1}(x)+
\frac{\xi_{i+p+1}-x}{\xi_{i+p+1}-\xi_{i+1}}B_{i+1,p-1}(x),\\
\nonumber
\text{with}{\ } B_{i,0}(x)=& \begin{cases}
                 	1,\, \text{if}\, \xi_i \leq x\leq \xi_{i+1},\\
                 	0,\, \text{otherwise}
                 \end{cases}
\end{align}
We assume that $m_j\leq p$ for all internal knots, which in turn gives that, at 
$z_j$ the B-spline basis functions have $\mathcal{\kappa}_j=p-m_j$ continuous derivatives. 
\par
Let us now consider the unit cube $\widehat{\Omega}=(0,1)^d\subset \mathbb{R}^d$, which we will refer to as the parametric domain.
Let the integers $p$  and $n_k$  denote  the
 given  B-spline degree  and  the number of basis functions
 of the B-spline space that will be constructed in $x_k$-direction with $k=1,\ldots,d$.  
We introduce the $d-$dimensional  vector of knots 
$\mathbf{\Xi}=(\Xi^1,\ldots,\Xi^{k},\ldots,\Xi^d),$ 
 with the particular components given by 
 $\Xi^{k}=\{0=\xi^{k}_1,  \xi^{k}_2, \ldots, \xi^{k}_{n_k+p+1}=1\}$,\, $k = 1,\ldots,d$, . 
\\
 Given the knot vector $\Xi^{k}$ in every direction $k=1,\ldots,d$, 
 we construct the associated univariate B-spline basis functions, 
 $\hat{\mathbb{B}}_{\Xi^{k},p}=\{\hat{B}_{1,k}(\hat{x}_k),\ldots,\hat{B}_{n_{k},k}(\hat{x}_k)\}$, 
 see, e.g., \cite{CarlDeBoor_Splines_2001} for more details.
Accordingly,  the  B-spline basis functions of $\hat{\mathbb{B}}_{\mathbf{\Xi},k}$ are defined by the tensor-product of the univariate B-spline basis functions, that is 
\begin{align}\label{0.00b1}
\hat{\mathbb{B}}_{\mathbf{\Xi},p}=\otimes_{k=1}^{d}\hat{\mathbb{B}}_{\Xi^{k},p}
=\text{span}\{\hat{B}_{j}(\hat{x})\}_{{j}=1}^{n=n_1\cdot \ldots\cdot n_k\cdot \ldots\cdot n_d},
\end{align}
where each  $\hat{B}_{j}(\hat{x})$ has the form
\begin{align}\label{0.00b2}
\hat{B}_{j}(\hat{x})=&\hat{B}_{j_1}(\hat{x}_1)\cdot\ldots
           \cdot \hat{B}_{j_k}(\hat{x}_k)\cdot\ldots
           \cdot\hat{B}_{j_d}(\hat{x}_d), \,
           \text{with}\,\hat{B}_{j_k}(\hat{x}_k) \in \hat{\mathbb{B}}_{\Xi^{k},k}.
\end{align}
 In the IGA framework, the computational domain  $\Omega$ is described as the image of $\widehat{\Omega}$ 
 under a B-spline, NURBS, etc.,  parametrization mapping of the form
\begin{align}\label{0.0c}
 \mathbf{\Phi}: \widehat{\Omega} \rightarrow \Omega, \quad
 x=\mathbf{\Phi}(\hat{x}) = \sum_{{j=1}}^n \mathbf{C}_{j} \hat{B}_{j}(\hat{x})\in \Omega,
 \end{align}
\label{0.0c2}
 where $\mathbf{C}_{j},\, j=1,\ldots,n$ are the control points and   $\hat{x} =\mathbf{\Phi}^{-1}(x)$, 
 see Fig. \ref{Fig_ContolNet_Overl}(a). 
 Following the IGA methodology, \cite{LT:HughesCottrellBazilevs:2005a}, \cite{LT:Hughes_IGAbook_2009},
 the B-spline spaces for discretizing the PDE problem are defined by using the mapping given in (\ref{0.0c}), i. e., 
we define the  B-spline space in $\Omega$ by 
\begin{align}\label{0.0d2}
	\mathbb{B}_{\mathbf{\Xi},p}:=\textcolor{black}{\text{span} }\{B_{{j}}|_{\Omega}: B_{j}({x})=
  \hat{B}_{j}\circ\mathbf{\Phi}^{-1}({x}),{\ }\text{for}{\ }
  \hat{B}_{j}\in \hat{\mathbb{B}}_{\mathbf{\Xi},p} \}. 
\end{align}
\subsubsection{ Multipatch representations and B-spline spaces} 
 \textcolor{black}{
Our contribution here aims at developing a DG-IGA method appropriate for discretizing PDE models on 
non-conforming multipatch partitions  of the domain $\Omega$. 
Let us suppose that the domain $\Omega$ is described as a union of N-subdomains  
\begin{align}\label{DecmbOmega}
\overline{\Omega}=\cup_{i=1}^N \overline{\Omega}_i, \quad \text{with}{\ } {\Omega}_i\cap{\Omega}_j=\emptyset,\,
\text{for}\, i\neq j,
\end{align}
with interior  interfaces 
 $F_{ij}=\partial \Omega_i\cap \partial \Omega_j$, for $1 \leq i\neq j \leq N$.  
We further suppose that every subdomain $\Omega_i$ has its own parametrization $\mathbf{\Phi}_i$,
which is defined by the corresponding B-spline space $\hat{\mathbb{B}}_{\mathbf{\Xi}_i,p}$ and the corresponding 
control points $\mathbf{C}_{j}^{(i)}$,  see (\ref{0.0c}). Here  $\mathbf{\Xi}_i$ denotes the knot-vector related to $\Omega_i$.   An illustration for $N=2$ is given 
in  Fig. \ref{Fig_ContolNet_Overl}(a). The subdomains $\Omega_i$ are referred to as patches. In an analogous way as in (\ref{0.0d2}), we define the physical patch-wise B-spline spaces $\mathbb{B}_{\mathbf{\Xi}_i,p}$ for $i=1,\ldots,N$. 
 \textcolor{black}{We  define the global discontinuous B-spline space $V_{\mathbb{B}}$ with components on every 
 $\mathbb{B}_{\mathbf{\Xi}_i,p}$
\begin{align}\label{0.0d1}
V_{\mathbb{B}}:=\{\phi_h\in L^2(\Omega): \phi_h|_{\Omega_i}\in \mathbb{B}_{\mathbf{\Xi}_i,p}\}.
\end{align}
}
\begin{assume}\label{smooth_Phi_i}
	Assume that every $\mathbf{\Phi}_i,\, i=1,...,N$ is sufficiently smooth  and 
	there exist constants $0 < c < C$  such that $c \leq |\det J_{\mathbf{\Phi}_i} | \leq C$, where $J_{\mathbf{\Phi}_i}$ is the
        Jacobian matrix of $\mathbf{\Phi}_i$. 
\end{assume}
 The components   of $\mathbf{\Xi}_i$  form 
a mesh  $T^{(i)}_{h_i,\widehat{\Omega}}=\{\hat{E}_m\}_{m=1}^{M_i}$ in $\widehat{\Omega}$,
where $\hat{E}_m$ are the micro-elements and $h_i$ is the mesh size, which is
defined as follows. Given an element  $\hat{E}_m\in T^{(i)}_{h_i,\widehat{\Omega}} $, 
we set $h_{\hat{E}_m}=\text{diameter}(\hat{E}_m)$
and the mesh size $h_i$ is defined to be  $h_i= \max\{h_{\hat{E}_m}\}$.
We set
$\textstyle{h=\max_{i=1,\ldots,N}\{h_i\}}$.
For every $\Omega_i$, we construct a mesh $T^{(i)}_{h_i,\Omega_i} =\{E_{m}\}_{m=1}^{M_i}$, 
whose vertices are the images of the vertices
 of the corresponding parametric mesh $T^{(i)}_{h_i,\widehat{\Omega}}$
 under $\mathbf{\Phi}_i$. 
 \begin{assume}\label{Assumption2}
	The  meshes $T^{(i)}_{h_i,\widehat{\Omega}}$ are quasi-uniform, i.e.,
	there exist a constant $\theta \geq 1$ such that 
	$\theta^{-1} \leq {h_{\hat{E}_m}}/{h_{\hat{E}_{m+1}}} \leq \theta$.
	Also, we assume that $h_i \sim h_j$ for $1 \leq i\neq j \leq N$.
\end{assume}
 }

\subsection{Multipatch description of the computational domain}
\begin{figure}
 \begin{subfigmatrix}{2}
 \subfigure[]{\includegraphics[width=5.650cm, height=6.475cm]{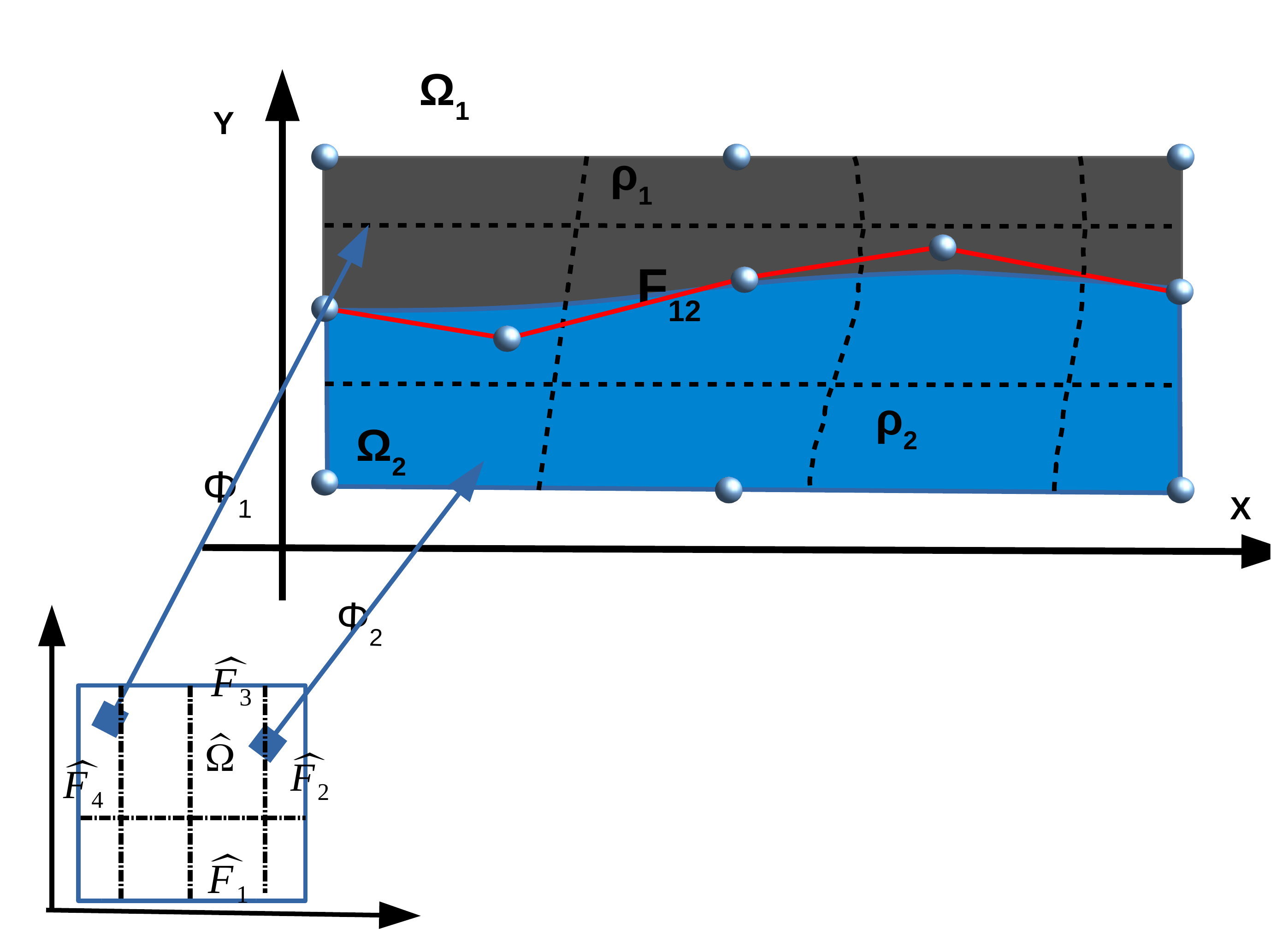}}
 \subfigure[]{\includegraphics[width=.65\textwidth,trim={0 9cm 0 0}]{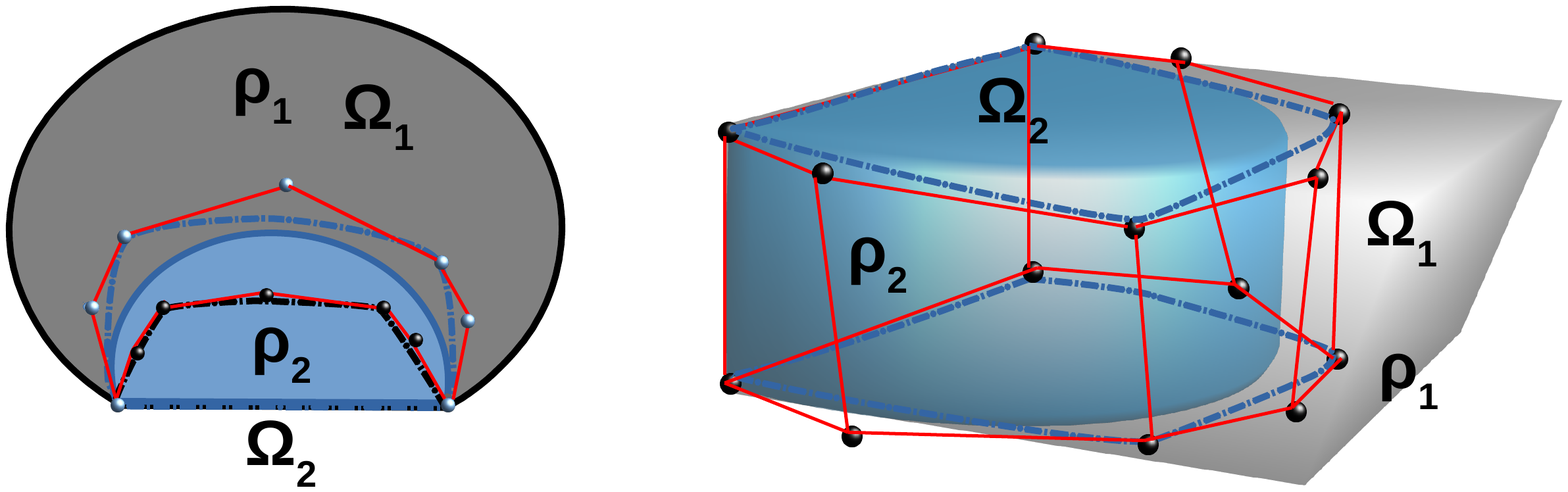}}
\end{subfigmatrix}
 \caption{ (a) A conforming multipatch partition of $\Omega$, (b) 
         the inaccurate control points and the non-conforming multipatch partition of $\Omega$.  }
 \label{Fig_ContolNet_Overl}
\end{figure}
\textcolor{black}{
In many practical applications, the parametrization of a boundary represented domain $\Omega$ by a single B-spline  (NURBS) patch may not be posible.
In order to discretize a PDE problem following the IGA framework in this situation, 
 we  represent the domain $\Omega$  as a  multipatch.  
 Following the methodology presented in 
 \cite{HTL:NguyenPauleyJuettler:2014a,HLT:JuettlerKaplNguyenPanPauley:2014a}, the initial domain $\Omega$ is firstly segmented 
 into a collection of simple subdomains, e.g.,  topological hexahedra. 
Consequently, a suitable parametrization mapping is constructed  
for each subdomain for obtaining  the multipatch representation  of $\Omega$. 
The final parametrization mappings  of the adjoining  patches must provide identical images for the common interfaces. 
 In particular, for a DG-IGA discretization of the  model (\ref{0}), it would be preferable to produce a 
multipatch partition  of  $\Omega$  compatible with the variations of the coefficient $\rho$, i.e., 
the patches to be coincided with the parts of $\Omega$ where the coefficient $\rho$ is constant. 
}
\textcolor{black}{
For example, let us consider  Fig. \ref{Fig_ContolNet_Overl}(a). In this case 
the domain  $\Omega$ is  described as a union of 
two  non overlapping patches, see (\ref{DecmbOmega}),   i.e., 
\begin{alignat}{2}\label{2.7_1}
 \overline{\Omega} = \overline{\Omega_1} \cup \overline{\Omega_2}, &\quad \overline{\Omega_1} \cap \overline{\Omega_2} = \emptyset, &
 {\ }\text{with} {\ } F_{12}=\partial \Omega_1 \cap \partial \Omega_2,
\end{alignat}
where the interface $F_{12}$ coincides with  the  physical interface. 
We use the notation  
$\cal{T}_{H}(\Omega):=\{\Omega_1, \Omega_2\}$ for the  union (\ref{2.7_1}).  
For each $\Omega_i,\,i=1,2$, there exists a  matching parametrization mapping such that $\mathbf{\Phi}_i:\widehat{\Omega}\rightarrow \Omega_i$ 
with $\Omega_i=\mathbf{\Phi}_i(\widehat{\Omega})$.
The control points, which are related to the patch interface $F_{12}$, are  appropriately 
matched in order for the  parametrizations $\mathbf{\Phi}_1$ and $\mathbf{\Phi}_2$ of the neighboring patches to give the same image for the 
parametrized interface $F_{12}$.   
Based on $\cal{T}_{H}(\Omega)$, we can  independently discretize the 
problem on the different patches $\Omega_i,\, i=1,2$, using  interface conditions across $F_{12}$
for coupling the local problems. 
Typically, the interface conditions across  $F_{12}$ concern continuity requirements of the solution $u$ of (\ref{0}), i.e.,
\begin{align}\label{Interf_Cond}
	\llbracket u \rrbracket :=u_1-u_2=0\text{ on }F_{12},\quad\text{and}\quad 
        \llbracket \rho \nabla u\rrbracket\cdot n_{F_{12}} :=(\rho_1\nabla u_1-\rho_2\nabla u_2)\cdot n_{F_{12}} = 0\text{ on }F_{i12},
\end{align}
 where $n_{F_{12}}$ is the unit normal  vector on $F_{12}$ with direction towards $\Omega_2$,
and $\rho_i,\,u_i,\,i=1,2$ denote the restrictions of $\rho$ and $u$ to $\Omega_i$ correspondingly. The conditions (\ref{Interf_Cond}) can be ensured by  considering appropriate regularity assumptions on the solution $u$. 
We note that these type of multipatch representations have been considered in \cite{LT:LangerToulopoulos:2014a} and DG-IGA methods have been proposed for discretizing the problem (\ref{0}). 
\\
Anyway,  for simplicity we develop  our analysis  based on Fig. \ref{Fig_ContolNet_Overl}. 
We introduce the appropriate spaces.  
Let  $\ell\geq 2$ be  an integer, we define the broken Sobolev space
\begin{align}\label{01c_0}
	H^{\ell}(\cal{T}_{H}(\Omega))=\{u\in L^{2}(\Omega): u_i=u|_{\Omega_i}\in H^{\ell}(\Omega_i),\, \text{for}\,i=1,2\}.
\end{align}
\begin{assume}\label{Assumption1}
 We assume that the  solution $u$ of (\ref{4a}) belongs to  
 \textcolor{black}{
 $V= H^{1}_0(\Omega)\cap H^{2}(\Omega) \cap H^{\ell}(\cal{T}_H(\Omega))$ with $\ell \geq  2$. 
}
\end{assume}
}
\begin{figure}
 \begin{subfigmatrix}{2}
   \subfigure[]{\includegraphics[width=.35\textwidth,trim={0 5cm 0 0}]{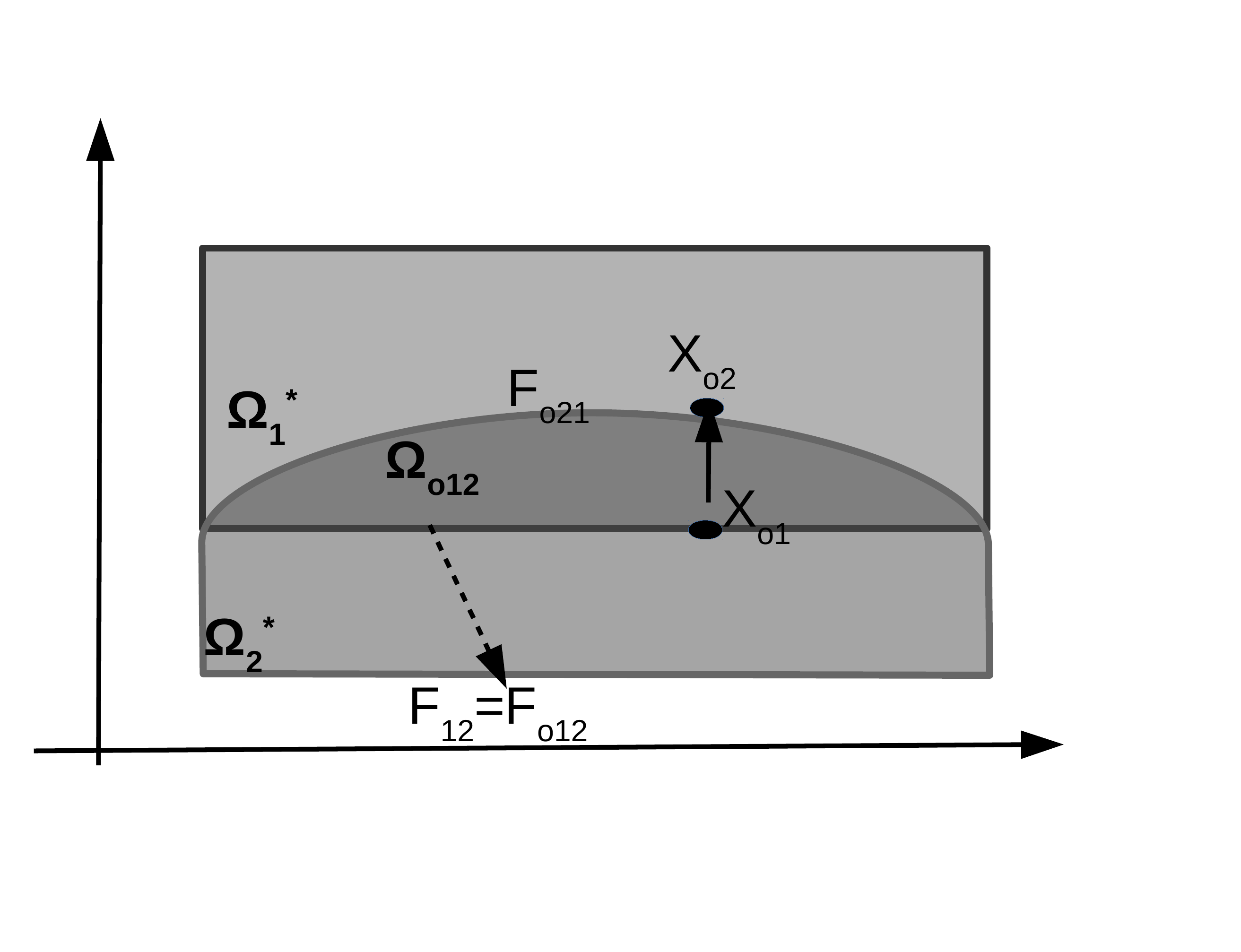}}
  \subfigure[]{\includegraphics[width=.5\textwidth,trim={0 8cm 0 0}]{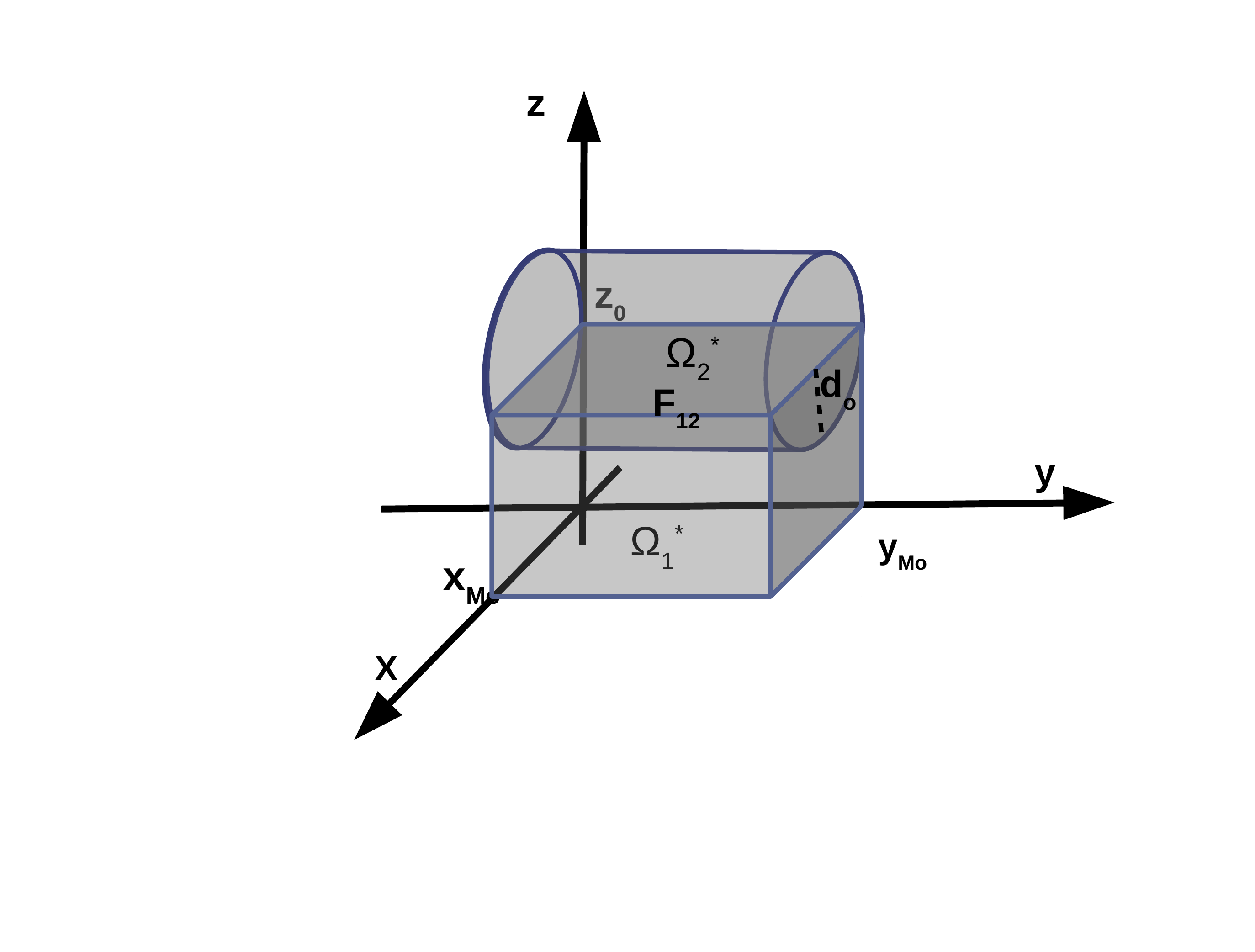}}
  \subfigure[]{\includegraphics[width=.4\textwidth,trim={0 5cm 0 0}]{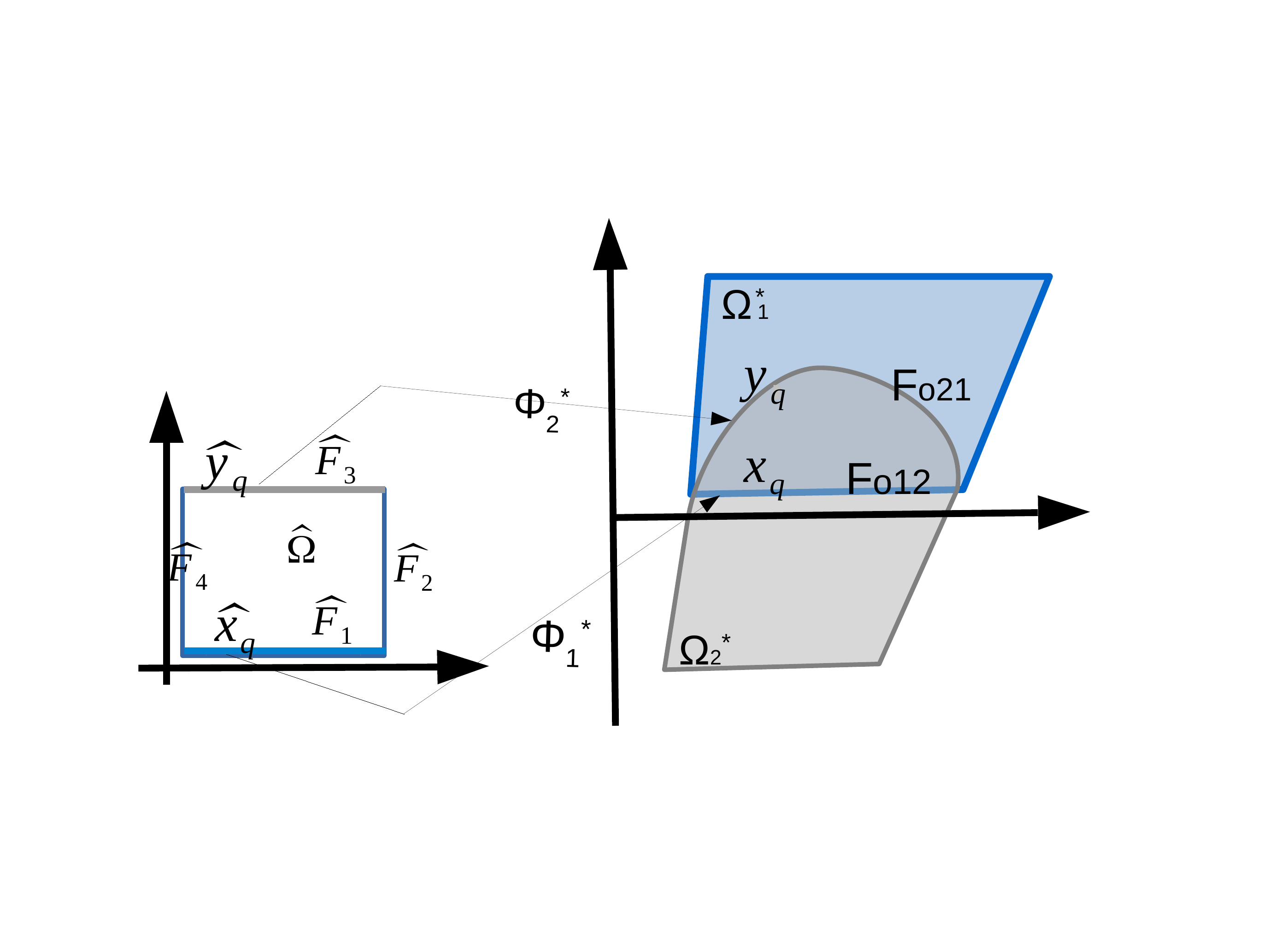}}
  \subfigure[]{\includegraphics[width=.35\textwidth,trim={0 7cm 0 0}]{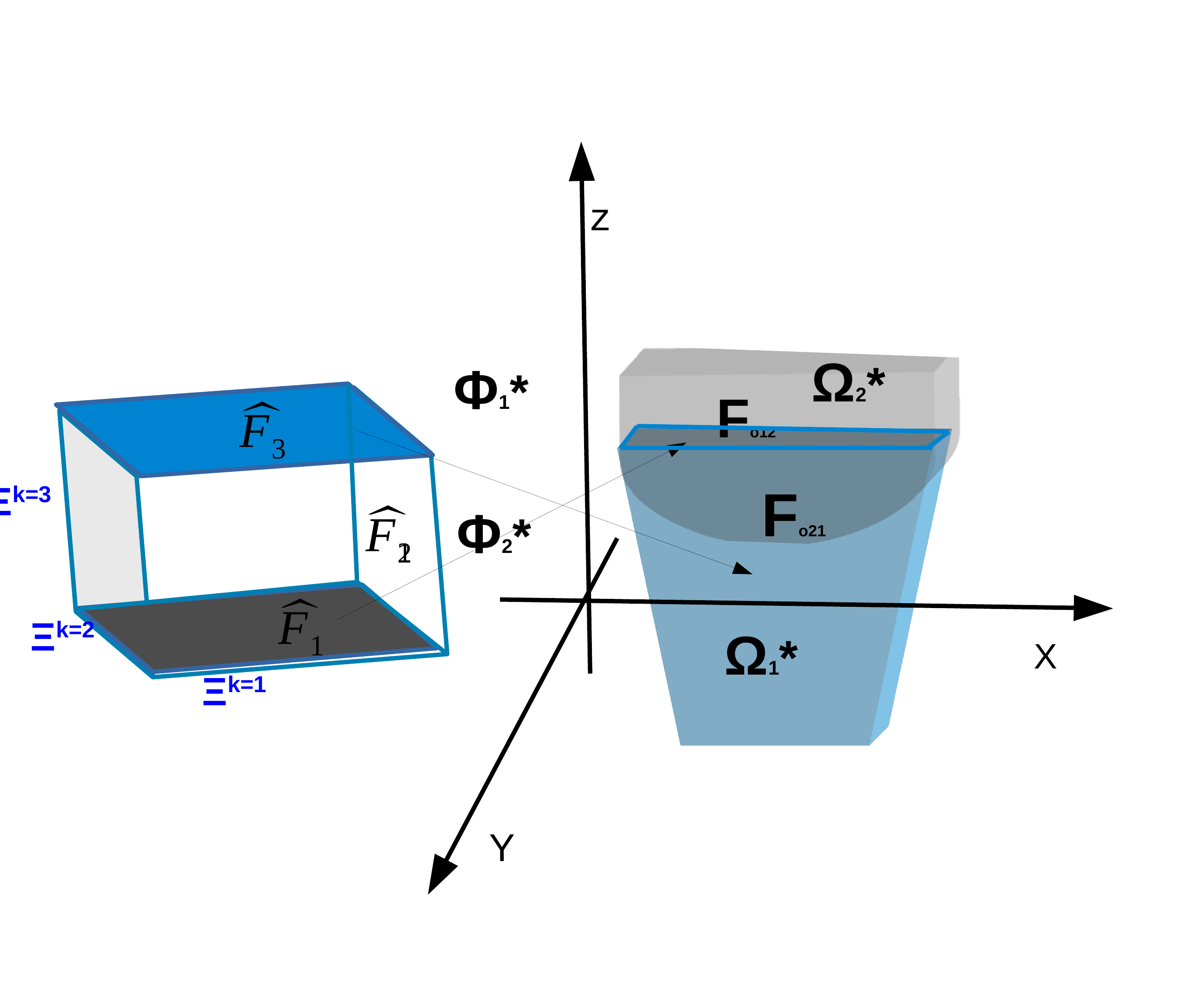}}
\end{subfigmatrix}
 \caption{ (a)  Illustration of a patch partition with the overlapping region $\Omega_{o21}$  in 2d 
                 and the diametrically opposite points on $\partial \Omega_{o21}$, 
           (b)  overlapping patches in 3d, 
           (c) the images of the faces of $\partial \widehat{\Omega}$ under the mappings $\mathbf{\Phi}_i^*,\,i=1,2$ in 2d,
                      (d) the images of the faces of $\partial \widehat{\Omega}$ in 3d.
         }
 \label{Fig_Sub_doms_Ovelap}
\end{figure}
\subsection{Problem statement}
\subsubsection{Non-matching parametrized interfaces}
\textcolor{black}{
Typically, the segmentation procedure will generate multipatch  representations that have
possibly non-matching interface parametrizations, \cite{HLT:PauleyNguyenMayerSpehWeegerJuettler:2015a}.
The result  is the existence of 
gap and overlapping regions in the multipatch representation of the domain $\Omega$. 
In  \cite{HoferLangerToulopoulos_2016_SISC} and \cite{HoferToulopoulos_IGA_Gaps_2015a}, we  developed
DG-IGA  schemes for  multipatch unions that only include gap regions.
In this work, we focus on multipatch representations with small overlapping  regions, 
see  Fig. \ref{Fig_ContolNet_Overl}(b) and Figs. \ref{Fig_Sub_doms_Ovelap}(a), (b).
Due to the  non-matching parametrization of the interior patch interfaces,
 a direct application of interface conditions, as those given in  
 (\ref{Interf_Cond}) for deriving DG-IGA methods,  is not possible. 
 The purpose of this paper is to investigate the construction of auxiliary 
 interface conditions on the boundary of the  overlapping 
regions; which can be used for constructing DG-IGA schemes.  
We present  a discretization error analysis separating the whole 
discretization error into two parts: the first naturally comes  from the approximation 
properties of the B-spline spaces
and the second, is the geometric error coming from the incorrect parametrization of the patches. The geometric error is considered as 
a consistency
error and it is further separated into two  components. The first error component is related to
the approximation of the flux terms across  the non-matching interfaces and
the second component is related to the existence of more than one numerical solution in the
overlapping regions.
\begin{remark}\label{JoinContrlp}
Alternatively, one can perform additional post-processing steps after the segmentation procedure to obtain
matching interfaces. However, this procedure may increase the number of patches and the number of control points.  
Moreover,  the newly obtained patch interfaces may not coincide  with the original interface of the PDE problem,  
and thus  the  geometrical consistency  error will still exist.   
\end{remark}
}
\subsubsection{The overlapping  regions}
\label{overlapping_region} 
\textcolor{black}{
As we mentioned above,  for the sake of  simplicity, we restrict our investigation  to the case where the multipatch representation of
$\Omega$ has  two overlapping patches, see 
Fig. \ref{Fig_Sub_doms_Ovelap}.  
Let  suppose that 
\begin{alignat}{2}\label{2.7_patches}
 \overline{\Omega} = \overline{\Omega_1^*} \cup \overline{\Omega_2^*}, 
\end{alignat}
where each patch has its own parametrization  $\mathbf{\Phi}_1^*:\widehat{\Omega}\rightarrow \Omega_1^*$ and 
 $\mathbf{\Phi}_2^*:\widehat{\Omega}\rightarrow \Omega_2^*$, as it is shown in  Figs. \ref{Fig_Sub_doms_Ovelap}(c),(d).
 We denote the overlapping region by $\Omega_{o21}$, i.e.,
 $\Omega_{o21}=\Omega_1^* \cap \Omega_2^*\subset \Omega$.  
 We  denote the interior boundary faces of the  overlapping region by  $F_{o12}=\partial \Omega^*_{1}\cap \Omega^*_{2}$ 
 and $F_{o21}=\partial \Omega^*_{2}\cap \Omega^*_{1}$,  
  which implies that $\partial \Omega_{o21}=F_{o12} \cup F_{o21}$. 
  Finally,  let
${n}_{F_{oij}}$  denote the unit exterior normal vector to $F_{oij}$, for $1\leq i\neq j\leq 2$.
For  functions  $u_i^*$ defined in $\Omega^*_{i},\,i=1,2$ we identify  their pair $(u_1^*,u_2^*)$ by $u^*$, which
is equal to $u_i^*$ on $\Omega_i^*$. 
Next, we introduce an assumption related to the form of the faces $F_{o21}$ and $F_{o12}$. This assumption 
will help us to simplify the analysis, to explain in a better way  our ideas, 
and to keep the notation to a minimum, e.g., the form of Jacobians, the form of face integrals etc. 
In Section \ref{Section_numerics},  we give details of implementing the proposed method to more complicated overlapping regions.  
\begin{assume}\label{Assu_Omega_*Shape}
	(a) $\overline{\Omega_1}:= \overline{\Omega^*_{1}}$. The face $F_{o12}$    is a an elementary face in the plane, 
	   and  it coincides with the physical interface,                i.e., $F_{o12}=F_{12}$, see (\ref{2.7_1}).\\
        (b) The face $F_{o21}$  can be described  as the set of points $(x,y,z)$ 
    satisfying 
         \begin{equation}\label{5a0_o}
0\leq x  \leq x_{M_o},{\ } 0  \leq y \leq y_{M_o},{\ } z=\zeta_0(x,y),
\end{equation}
           where $x_{Mo}$ and $y_{M_o}$  are  real numbers, and   $\zeta_0(x,y)$ 
             is a  given  smooth  function, see Fig.~\ref{Fig_Sub_doms_Ovelap}.     
\end{assume}
We note  that   we will discretize the PDE problem using the B-spline spaces defined in  
 $\Omega_1^*$ and $\Omega_2^*$. We will couple the resulting discrete problems  in $\Omega_1^*$ and 
in  $\Omega_2^*$ following discontinuous Galerkin techniques, this means by introducing  appropriate
numerical fluxes on $F_{o12}$ and on $F_{o21}$. In order to construct these fluxes, 
we need to assign the points located on $F_{o12}$ to the diametrically opposite points located on $F_{o21}$. 
 Based on Assumption \ref{Assu_Omega_*Shape}, 
we can construct   a parametrization for the face $F_{o21}$, i.e., a  mapping 
$\mathbf{\Phi}_{o12}:F_{o12}\rightarrow F_{o21}$, of the form
\begin{align}\label{parmatric_lro}
 x_{o1}\in F_{o12} \rightarrow \mathbf{\Phi}_{o12}(x_{o1}):=x_{o2}\in F_{o21},\quad \text{with }\quad
          \mathbf{\Phi}_{o12}(x_{o1})=x_{o1}+ \zeta_o(x_{o1}){n}_{F_{o12}},
\end{align}
where   ${n}_{F_{o12}}$ is the unit normal vector on ${F_{o12}}$ and $\zeta_o$ has the same form as in (\ref{5a0_o}),
and it  is a B-spline function with the same degree as the mapping $\mathbf{\Phi}_{2}^*$. More precisely, the face $F_{o21}$ 
is the image of a face of $\partial \hat{\Omega}$ under  the mapping $\mathbf{\Phi}_{2}^*$. 
For the schematic illustration in Figs. \ref{Fig_Sub_doms_Ovelap}(c),(d), we have $F_{o21}=\mathbf{\Phi}_{2}^*(\hat{F}_3)$. 
  Utilizing the mapping $\mathbf{\Phi}_{o12}$ given in (\ref{parmatric_lro}), we   consider 
each point
$x_{o2}\in F_{o21}$ as an image  of a point $x_{o1}\in F_{o12}$ under the  $\mathbf{\Phi}_{o12}$, 
see Figs.~\ref{Fig_Sub_doms_Ovelap}(a),(c).
Finally, we  introduce a parameter $d_o$, which  quantifies the width of the overlapping region $\Omega_{o21}$, i.e., 
\begin{align}\label{5b_0_o}
d_o=&\max_{x_{o1}\in F_{o12}} |{x_{o1}}- \mathbf{\Phi}_{o12}(x_{o1})|.
\end{align}
In the present work, we are interested in overlapping regions with small size, and in particular
for regions where their width $d_o$  decreases polynomially  in $h$, i.e.,
\begin{align}\label{5b_1_o}
d_o \leq & C\, h^\lambda,\quad \text{with $C>0$ and  some}\quad  \lambda \geq  1.
\end{align}
Based on this, we  assume that  ${n}_{F_{o12}} \approx -{n}_{F_{o21}}$, and 
define the mapping $\mathbf{\Phi}_{o21}:F_{o21}\rightarrow F_{o12}$  as
 \begin{align}\label{parmatric_rlo}
        \mathbf{\Phi}_{o21}(x_{o2})= x_{o1},{\ }\text{with} \quad
        \mathbf{\Phi}_{o12}(x_{o1})=x_{o2},
 \end{align}
 where   $\mathbf{\Phi}_{o21}$ is  the inverse of $\mathbf{\Phi}_{o12}$. 
\begin{remark}\label{JoinContrlp}
Our methodology can also be applied to the case where the interior faces of $\partial\Omega_{o21}$ do not touch the boundary $\partial \Omega$.
\end{remark}
\begin{remark}\label{InverMApnonMatchInterf}
As we previously said, the face 
$F_{o12}$ is the image of a face of $\partial\widehat{\Omega}$ under  the mapping $\mathbf{\Phi}_1^*$, for example 
in Fig. \ref{Fig_Sub_doms_Ovelap}(c) we have  $F_{o12}=\mathbf{\Phi}_1^*(\hat{F}_{1})$. 
On the other hand, the face $F_{o21}$ is an interior curve for  $\Omega_1^*$, see Figs. \ref{Fig_Sub_doms_Ovelap}(a),(c). 
Thus, one could try to see $F_{o21}$ as an image of a curve $\hat{F}_{o21}\subset\widehat{\Omega}$ under  the mapping $\mathbf{\Phi}_1^*$, i.e., $F_{o21}=\mathbf{\Phi}_1^*(\hat{F}_{o21})$. In that way, it would be advantageous to  have a parametric description of $\partial \Omega_{o21}$ using  the mapping $\mathbf{\Phi}_1^*$,  which in turn  would help  
 to link the diametrically opposite points $x_{o1}$ and $x_{o2}$, see (\ref{parmatric_lro}). 
This approach requires the computation of the inverse $\big(\mathbf{\Phi}_1^*\big)^{-1}$, which in general is very costly and demands the 
use of a Newton approach for solving  many nonlinear systems. We are thus led to see the faces of $\partial \Omega_{o21}$ 
as images of  both mappings $\mathbf{\Phi}_1^*$ and $\mathbf{\Phi}_2^*$. 
We note also that the mappings $\mathbf{\Phi}_{o12}$ and $\mathbf{\Phi}_{o21}$ are introduced and used 
only for  deriving the discretization error analysis. They are not used in the computation of
the entries of the system matrix of the discrete DG-IGA scheme, see also discussion in Subsection 
\ref{Implementation}. 
\end{remark}
\begin{remark}\label{normalerror}
In  Section \ref{Section_numerics}, we present examples where the 
normal vector $n_{F_{o12}}$ is not constant across the face $F_{o12}$. 
\end{remark}
}
\section{The patch-wise  problems and the  fluxes}
\label{section_3}
\textcolor{ black}{
We  compute a numerical solution in each $\Omega_i^*,\, i=1,2$ 
  using the corresponding diffusion coefficient $\rho_i$ and the corresponding  B-spline spaces defined in  $\Omega_i^*$, 
  lets say   $\mathbb{B}^{*}_{\mathbf{\Xi}_i,p}$. Therefore on $\Omega_{o21}$ we will have the coexistence of  two different numerical solutions and this makes the computation  of the  bounds for the error   
  $\|u-u^*_h\|_{DG}$ more complicated. The norm $\|\cdot\|_{DG}$ is defined in (\ref{3.5.2a}).  
  The idea in our approach  is to
 introduce local (patch-wise)   problems $ a_i^*(u_i^*,\phi_i^*)=l_{i,f}(\phi_i^*) $ in every $\Omega_i^*$, 
 with appropriate  bilinear forms $a_i^*(\cdot,\cdot)$.  Using the triangle inequality, we split the error as  $\|u-u^*_h\|_{DG} \leq \|u-u^*\|_{DG} + \|u^*-u^*_h\|_{DG}$.   Then we estimate every term separately. 
}

 \subsection{The patch-wise variational problems}
 Denote  $\cal{T}_H^*(\Omega):=\{\Omega_1^*,\Omega_2^*\}$,  let $\ell \geq 1$ be an integer and let the B-spline spaces $\mathbb{B}^{*}_{\mathbf{\Xi}_i,p}$ defined in $\Omega_i^*, \,i=1,2$. 
 Accordingly to the spaces  (\ref{01c_0}) and (\ref{0.0d1}), we introduce the spaces 
 \begin{align}\label{spaces_Omega*}
 \begin{split}
 	H^{\ell}(\cal{T}^*_{H}(\Omega)):=&\{u^*=(u^*_1,u^*_2):u_i^*\in H^{\ell}(\Omega^*_i),\, 
 	                  u^*_i|_{\partial \Omega_i^*\cap \partial \Omega}=0,\text{for}\,i=1,2\},\\
 	H^{\ell}_0(\cal{T}^*_{H}(\Omega)):=&\{u^*=(u^*_1,u^*_2):u_i^*\in H^{\ell}_0(\Omega^*_i),\, \text{for}\,i=1,2\}, \\
 	V^*_{\mathbb{B}}:=& \{\phi^*_h=(\phi_{1,h},\phi_{2,h}): \phi_{i,h}\in  \mathbb{B}^{*}_{\mathbf{\Xi}_i,p}, \text{for}\,i=1,2  \}.
 \end{split}
 \end{align}
In order to proceed, we first define    
the {DG}-norm $\|.\|_{DG}$ associated with  $\mathcal{T}^*_H(\Omega)$. 
For all $v\in V_{h}^{*}:=H^{\ell}(\cal{T}^*_{H}(\Omega))+
V^*_{\mathbb{B}}$,
\begin{align}\label{3.5.2a}
\hskip -0.0cm{
	\|v\|^2_{DG} =  \sum_{i=1}^2\Big(\rho_i\|\nabla v_i\|^2_{L^2(\Omega^*_i)} + 
                      \frac{\rho_i}{h}\|v_i\|^2_{L^2(\partial \Omega^*_i\cap \partial\Omega)}
                       +  \sum_{F_{oij}\subset \partial \Omega_i^*}   \frac{\{\rho\}}{h}\|v_i  \|^2_{L^2(F_{oij})}\Big), {\ }\text{for}{\ }1\leq i\neq j\leq 2, }
\end{align} 
where  $F_{oij}$  are the interior faces related to  overlapping regions,
see Fig.~\ref{Fig_Sub_doms_Ovelap}(a), and $\textstyle{\{\rho\}=\frac{1}{2}(\rho_i+\rho_j)}$. 
\par
 We recall Assumption \ref{Assu_Omega_*Shape}. 
On each $\Omega_i^*,\,i=1,2$, we consider  the auxiliary problems:
 \begin{multicols}{2}
\noindent
\begin{subequations}\label{ArtProb_Omega*}
    \begin{equation} 
       (P1)\, \begin{cases}
            -\mathrm{div}(\rho_1 \nabla u_1^*) & = f,\, \text{in}\, \Omega_1^*, \\
                                            u_1^*    = u_D & \text{on}\,   \partial \Omega_1^*\cap \partial \Omega,\\ 
                                            u_1^*    = u  & \text{on}\,   F_{o12},\\ 
              \end{cases}
        \label{eq:ArtProb_Omega*_1}
    \end{equation}
    \begin{equation}
        (P2)\, \begin{cases}
            -\mathrm{div}(\rho_2 \nabla u_2^*) & = f,\, \text{in}\, \Omega_2^*, \\
                                            u_2^*    = u_D & \text{on}\,    \partial \Omega_2^*\cap \partial \Omega,\\ 
                                            u_2^*    = u  & \text{on}\,  F_{o21},\\ 
              \end{cases}
        \label{eq:ArtProb_Omega*_2}
    \end{equation}
    \end{subequations}
\end{multicols}
 and furthermore,  we  consider the corresponding  variational problems,  
\begin{subequations}\label{Artf_Vart_Problm}
\begin{align}\label{Artf_Vart_Problm_0}
\nonumber
 \text{(P1V)\,  find $u_1^*\in H^1_{ \partial \Omega_1^*\cap \partial \Omega}(\Omega_1^*)$ such that}{\ }& \\
  \,u_1^*=u\,\text{on}\,{F_{o12}},{\ } 
 \text{and}{\ }	& \quad a_1^*(u_1^*,\phi_1) = l^*_{1, f}(\phi_1), {\ }\text{for }{\ }\phi_1\in H^{1}_0(\Omega_1^*),
 \end{align}
 where
 \begin{align}\label{Artf_Vart_Problm_01}
 a_{1}^{*}(u_{1}^{*},\phi_1) =&  \int_{\Omega^*_1}\rho_1\nabla u_{1}^{*}\cdot\nabla \phi_1\,dx, 
 {\ }\text{and}{\ }  l^{*}_{1,f}(\phi_1)        = \int_{\Omega^*_1} f\phi_1\,dx,
\end{align}
 \begin{align}\label{Artf_Vart_Problm_1}
 \nonumber
 \text{(P2V)\,  find $u_2^*\in H^1_{\partial \Omega_2^*\cap \partial \Omega}(\Omega_2^*)$ such that}{\ }& \\
 u_2^*=u\,\text{on}\,F_{o21},{\ }
\text{and}{\ }& \quad  	a_2^*(u_2^*,\phi_2) = l^*_{2, f}(\phi_2), {\ }\text{for }{\ }\phi_2\in H^{1}_0(\Omega_2^*),
 \end{align}
 where
 \begin{align}\label{Artf_Vart_Problm_2}
 a_{2}^{*}(u_{2}^{*},\phi_2) =&  \int_{\Omega^*_2}\rho_2\nabla u_{2}^{*}\cdot\nabla \phi_2\,dx, 
 {\ }\text{and}{\ }    l^{*}_{2,f}(\phi_2)        = \int_{\Omega^*_2} f\phi_2\,dx,
\end{align}
\end{subequations} 
 \begin{remark}\label{NonConsistenceArt}
 By Assumption  \ref{Assu_Omega_*Shape} and the definition of problem (P1), we can imply that  
 the solution $u$ of (\ref{4a}) satisfies the problem (P1V). The definition of (P2) and the fact that $\rho_2\neq \rho_1$ on $\Omega_{o21}$ imply  that $u$ of (\ref{4a})   does not satisfy  the problem   (P2V). 
 \end{remark}
 According to Assumption \ref{Assumption1}, we make the following  assumption.
 \begin{assume}\label{Appr_in_omega_o}
The solutions $u_i^*,\,i=1,2$  in  (\ref{Artf_Vart_Problm})  belong to  $H^{\ell}(\cal{T}^*_{H}(\Omega))$ 
with $\ell\geq 2$.
\end{assume}
In Appendix, see Subsection \ref{Apend_u_minus_u*}, we give an estimate for the distance of the solutions $u$ and $u^*$. 
\subsection{The non-consistent   {terms}.}
We multiply the problem (\ref{eq:ArtProb_Omega*_2}) by $\phi_{2,h} \in \mathbb{B}^{*}_{\mathbf{\Xi}_2,p}$, 
integrate over $\Omega_2^*$ and apply integration by parts, then after few calculations 
we find that 
\begin{align}\label{0_PetrVF}
\begin{split}
\int_{\Omega^*_{2}}&\rho_2\nabla u_2^*\cdot \nabla \phi_{2,h}\,dx 
                    -   \int_{\partial \Omega^*_2}\rho_2\nabla u_2^*\cdot n_{\partial \Omega^*_2} \phi_{2,h}\,d\sigma \\
 =&\int_{\Omega_{o21}}\rho_1\nabla u_2^*\cdot \nabla \phi_{2,h}\,dx 
                   + \int_{\Omega_{2}}\rho_2\nabla u_2^*\cdot \nabla \phi_{2,h}\,dx  
  - \int_{\partial \Omega_2^*\cap\partial \Omega}\rho_2\nabla u_2^*\cdot n_{\partial \Omega_2} \phi_{2,h}\,d\sigma 
  \\
  \hskip -0.73cm{
 -}
 &\int_{F_{o21}} {\rho_1}\nabla u_2^*\cdot n_{F_{o21}}\phi_{2,h}\,d\sigma
    - \int_{F_{o12}} {\rho_1}\nabla u_2^*\cdot n_{F_{o12}}\phi_{2,h}\,d\sigma  
    - \int_{F_{o12}} {\rho_2}\nabla u_2^*\cdot (-n_{F_{o12}})\phi_{2,h}\,d\sigma \\  
 +\int_{\Omega_{o21}}(\rho_2-\rho_1)&\nabla u_2^*\cdot \nabla \phi_{2,h}\,dx - 
 \int_{F_{o12}} (\rho_2-\rho_1)\nabla u_2^*\cdot n_{F_{o12}}\phi_{2,h}\,d\sigma 
  - \int_{F_{o21}} (\rho_2-\rho_1)\nabla u_2^*\cdot n_{F_{o21}}\phi_{2,h}\,d\sigma   \\
= l^{*}_{2,f}(\phi_{2,h}),
\end{split}
  \end{align}
  and in a similar way,   multiplying 
  the problem (\ref{eq:ArtProb_Omega*_1}) by $\phi_{1,h} \in \mathbb{B}^{*}_{\mathbf{\Xi}_1,p}$, 
  we  have
  \begin{alignat}{2}
  \int_{\Omega^*_{1}}\rho_1\nabla u_1^*\cdot \nabla \phi_h\,dx 
 - \int_{F_{o12}} {\rho_1}\nabla u_1^*\cdot n_{F_{o12}}\phi_h\,d\sigma - 
 \int_{\partial \Omega^*_1\cap\partial \Omega}\rho_1\nabla u_1^*\cdot n_{\partial \Omega_1} \phi_{1,h}\,d\sigma=
 l^{*}_{1,f}(\phi_{1,h}).
  \end{alignat}
  We define the forms
  \begin{subequations}\label{PetrVF_wh}
\begin{align}
\hskip -0.8cm{
a_{2,h}^{*}(u_2^*,\phi_{2,h}):=\int_{\Omega^*_{2}} }  &\rho_2\nabla u_2^*\cdot \nabla \phi_{2,h}\,dx 
 - \int_{F_{o21}} {\rho_2}\nabla u_2^*\cdot n_{F_{o21}}\phi_{2,h}\,d\sigma - 
 \int_{\partial \Omega^*_2\cap\partial \Omega}\rho_2\nabla u_2^*\cdot n_{\partial \Omega_2} \phi_{2,h}\,d\sigma, \\
 \hskip -0.8cm{
a_{1,h}^{*}(u_1^*,\phi_{1,h}):=\int_{\Omega^*_{1}} }  &\rho_1\nabla u_1^*\cdot \nabla \phi_{1,h}\,dx 
 - \int_{F_{o12}} {\rho_1}\nabla u_1^*\cdot n_{F_{o12}}\phi_{1,h}\,d\sigma - 
 \int_{\partial \Omega^*_1\cap\partial \Omega}\rho_1\nabla u_1^*\cdot n_{\partial \Omega_1} \phi_{1,h}\,d\sigma  
  \end{align}
\end{subequations} 
  and also
   \begin{subequations}\label{resd_form}
  \begin{align}
 \hskip -0.85cm { a_{o,2}(u_2^*,\phi_{2,h}) =}&
  \int_{\Omega_{o21}}\rho_1\nabla u_2^*\cdot \nabla \phi_{2,h}\,dx 
                   + \int_{\Omega_{2}}\rho_2\nabla u_2^*\cdot \nabla \phi_{2,h}\,dx  - 
   \int_{\partial \Omega_2^* \cap \partial \Omega}\rho_2\nabla u_2^*\cdot n_{\partial \Omega_2^*} \phi_{2,h}\,d\sigma                  
  \\
  \nonumber
 - &\int_{F_{o21}} {\rho_1}\nabla u_2^*\cdot n_{F_{o21}}\phi_{2,h}\,d\sigma
    - \int_{F_{o12}} {\rho_1}\nabla u_2^*\cdot n_{F_{o12}}\phi_{2,h}\,d\sigma  
    - \int_{F_{o12}} {\rho_2}\nabla u_2^*\cdot (-n_{F_{o12}})\phi_{2,h}\,d\sigma \\  
  a_{res}(u_2^*,\phi_{2,h}) = &
  \int_{\Omega_{o21}}(\rho_2-\rho_1)\nabla u_2^*\cdot \nabla \phi_{2,h}\,dx - 
  \int_{F_{o12}} (\rho_2-\rho_1)\nabla u_2^*\cdot n_{F_{o12}}\,\phi_{2,h}\,d\sigma \\
  \nonumber
  -& \int_{F_{o21}} (\rho_2-\rho_1)\nabla u_2^*\cdot n_{F_{o21}}\phi_{2,h}\,d\sigma.
 \end{align}
 \end{subequations} 
By (\ref{0_PetrVF}), (\ref{PetrVF_wh}) and  (\ref{resd_form}), we get  that
\begin{alignat}{2}\label{PetrVF}
a_{2,h}^{*}(u_2^*,\phi_h) = a_{o,2}(u_2^*,\phi_{2,h})& +a_{res}(u_2^*,\phi_{2,h}) 
= l^{*}_{2,f}(\phi_{2,h}),
\end{alignat}
Also  for the solution $u$ of (\ref{4a}) we have that 
\begin{align}\label{1_VaraForm_u_Omega*}
\int_{\Omega_{o21}}\rho_1\nabla u\cdot \nabla \phi_{2,h}\,dx 
                   + \int_{\Omega_{2}}\rho_2\nabla u\cdot \nabla \phi_{2,h}\,dx  
 - \int_{F_{o21}} {\rho_1}\nabla u\cdot n_{F_{o21}}\phi_{2,h}\,d\sigma \\
 \nonumber
-    \int_{\partial \Omega_2^* \cap \partial \Omega}\rho_2\nabla u\cdot n_{\partial \Omega_2^*} \phi_{1,h}\,d\sigma
 =
 l^{*}_{2,f}(\phi_{2,h}).
  \end{align}
From the conditions (\ref{Interf_Cond}),  the 
forms defined in (\ref{PetrVF_wh}), (\ref{resd_form}) and the relations  (\ref{PetrVF}) and (\ref{1_VaraForm_u_Omega*}), we   derive that 
 \begin{subequations}\label{A_VaraForm_u_Omega*}
\begin{align}\label{2_VaraForm_u_Omega*}
	&a_{o,2}(u_2^*,\phi_{2,h}) +a_{res}(u_2^*,\phi_{2,h}) = a_{o,2}(u,\phi_{2,h})=l^{*}_{2,f}(\phi_{2,h})\\
	\intertext{and}
	\label{3_VaraForm_u_Omega*}
& a_{2,h}^{*}(u,\phi_{2,h}) - a_{res}(u,\phi_{2,h})  = l^{*}_{2,f}(\phi_{2,h}).
\end{align}
 \end{subequations}
 By a simple application of divergence theorem, we get
\begin{alignat}{2}\label{resd_form_1}
\begin{split}
  a_{res}(u,\phi_{2,h}) = &
  \int_{\Omega_{o21}}-\mathrm{div}\big((\rho_2-\rho_1)\nabla u\big)\phi_{2,h}\,dx   = 
  \int_{\Omega_{o21}}\frac{(\rho_2-\rho_1)}{\rho_1}f\,\phi_{2,h}\,dx.
  \end{split}
 \end{alignat}
Finally, by   (\ref{3_VaraForm_u_Omega*}) and (\ref{resd_form_1}), we deduce that
\begin{align}\label{PetrVF_2}
a_{2,h}^{*}(u,\phi_{2,h})
+ 
\int_{\Omega_{o21}}\frac{(\rho_1-\rho_2)}{\rho_1}f\,\phi_{2,h}\,dx=l^{*}_{2,f}(\phi_{2,h}).
\end{align}
\begin{proposition}\label{Propos_PoincareOmega_o}
 Let $\phi_{2,h} \in \mathbb{B}^{*}_{\mathbf{\Xi}_2,p}$. There is a $c>0$ dependent on $\rho$ but
 independent of $u$ and $\Omega_{o21}$ such that
 \begin{align}\label{PetrVF_3}
  \|\phi_{2,h}\|^2_{L^2(\Omega_{o21})} \leq c d_o\,h\big(\int_{\Omega_2^*}|\nabla\phi_{2,h}|^2\,dx+
  \frac{\{\rho\}}{h}\int_{F_{o21}} \phi^2_{2,h}\,d\sigma.\big)
 \end{align}
\end{proposition}
\begin{proof}
 Let $\mathbf{v}=(0,y\phi_{2,h}^2)$. The divergence theorem for $\mathbf{v}$ on $\Omega_{o21}$ 
  yields, 
 \begin{align}\label{PetrVF_4}
 \int_{\Omega_{o21}}\phi_{2,h}^2\,dx+
 \int_{\Omega_{o21}}2y\phi_{2,h}\,\partial_y\phi_{2,h}\,dx=\int_{F_{o21}} y\phi^2_{2,h}\,d\sigma.
 \end{align}
 Using that $y\leq d_o$ and applying (\ref{HolderYoung}) in (\ref{PetrVF_4}) we obtain 
 \begin{align}\label{PetrVF_5}
  \|\phi_{2,h}\|^2_{L^2(\Omega_{o21})} \leq \big(\epsilon^2\int_{\Omega_{o21}}\phi_{2,h}^2\,dx+
  \frac{4}{\epsilon^2}\int_{\Omega_{o21}}d_o^2\,|\nabla\phi_{2,h}|^2\,dx+
  d_o h \frac{1}{h}\int_{F_{o21}} \phi^2_{2,h}\,d\sigma\big)
 \end{align}
Gathering similar terms and choosing $\epsilon$ appropriately small,  we get
\begin{align}\label{PetrVF_6}
 c_{1,\epsilon} \|\phi_{2,h}\|^2_{L^2(\Omega_{o21})} \leq c_{2,\epsilon}c_\rho\,d_oh\big(
  \int_{\Omega_2^*}\rho_2|\nabla\phi_{2,h}|^2\,dx+
   \frac{\{\rho\}}{h}\int_{F_{o21}} \phi^2_{2,h}\,d\sigma\big),
 \end{align}
where we  used  that  $d_o^2\leq d_oh$. Rearranging appropriately the constants in (\ref{PetrVF_6}) yields
(\ref{PetrVF_3}).
 \BLACKBOX
\end{proof}

\begin{corollary}\label{fphi_hOmega_o}
	Let $f\in L^{\infty}(\Omega)$, $\phi_{2,h}\in \mathbb{B}^{*}_{\mathbf{\Xi}_2,p}$ and let $u_2^*$ and $u$ be the solutions of  (\ref{Artf_Vart_Problm_2}) and  (\ref{4a}) respectively. 
	There are  constants $c_1,\,c_{\rho}>0$ dependent on $F_{o21}$ but independent
	of $h$ such that
	\begin{subequations}\label{fphi_hOmega_o1}
	\begin{align}\label{fphi_hOmega_o1_A}
		\int_{\Omega_{o21}}f\phi_{2,h}\,dx & \leq c_1\,d_o\, \|f\|_{L^\infty(\Omega_{o21})}\|\phi_{2,h}\|_{DG}, \\
		\label{fphi_hOmega_o1_B}
		\textcolor{black}{
		|a_{res}(u,\phi_{2,h})| }       & 
		\textcolor{black}{
		\leq c_{\rho}\,d_o\, \|f\|_{L^\infty(\Omega_{o21})}\|\phi_{2,h}\|_{DG},
		}
	\end{align}
	\end{subequations}
\end{corollary}
\begin{proof}
	It follows by from  the Cauchy-Schwartz inequality that
	\begin{align}\label{fphi_hOmega_o2}
	\int_{\Omega_{o21}}f\phi_{2,h}\,dx \leq \|f\|_{L^2(\Omega_{o21})}\|\phi_{2,h}\|_{L^2(\Omega_{o21})}\leq 
	c_{F_o21}d_o^{\frac{1}{2}}\|f\|_{L^\infty(\Omega_{o21})}\|\phi_{2,h}\|_{L^2(\Omega_{o21})}.
	\end{align}
   Using (\ref{PetrVF_3}) in (\ref{fphi_hOmega_o2}), the required assertion follows easily. \\
   Inequality (\ref{fphi_hOmega_o1_B}) follows immediately from (\ref{resd_form_1}) and (\ref{fphi_hOmega_o1_A}).
\BLACKBOX
\end{proof}

\subsection{The discrete problem}
In this section, we use the bilinear forms given in (\ref{PetrVF_wh}) to 
define the patch-wise discrete problems. 
Using the  conditions on $F_{o21}$ and $F_{o12}$, which are given in (P1) and (P2), the 
Assumption \ref{Appr_in_omega_o}, we imply  the following  interface conditions 
 \begin{align}\label{NormalJumpOverl_F*}
  \rho_1\nabla u_{2}^{*}\cdot n_{F_{o12}} = \rho_2\nabla u^*_{2}\cdot n_{F_{o12}}{\ }\text{on $F_{o12}$, \quad and} \quad
   u_{1}^{*}-u^*_{2}=0 \text{on}\,F_{o21}.
 \end{align}
Next, using Taylor expansions, we  appropriately modify the  flux terms $\int_{F_{o21}} {\rho_2}\nabla u_2^*\cdot n_{F_{o21}}\phi_{2,h}\,d\sigma$ and 
 $\int_{F_{o12}} {\rho_1}\nabla u_1^*\cdot n_{F_{o12}}\phi_{1,h}\,d\sigma$ appearing in (\ref{PetrVF_wh}). 
\subsubsection{Taylor expansions  }
Let $x,y\in \overline{\Omega}^*_2$ and let  $f\in C^{m\geq 2}(\overline{\Omega}^*_2)$.
We recall 
Taylor's formula with integral remainder
\begin{subequations}\label{7_b}
\begin{align}\label{7_b_1}
f(y)=&f(x) +\nabla f(x)\cdot(y-x) + R^2f(y +s(x-y)),\\
 f(x)=&f(y) -\nabla f(y)\cdot (y-x)  + R^2 f(x +s(y-x)),
\end{align}
\end{subequations}
where $R^2 f(y +s(x-y))$ and $R^2 f(x +s(y-x))$ are the second order remainder terms defined by
\begin{subequations}\label{7_b_2}
\begin{align}
    R^2 f(y +s(x-y))=   \sum_{|\alpha|=2}(y-x)^\alpha\frac{2}{\alpha !}
        \int_{0}^1sD^\alpha  f(y +s(x-y))\,d s, \\
    R^2 f(x +s(y-x))=   \sum_{|\alpha|=2}(x-y)^\alpha\frac{2}{\alpha !}
        \int_{0}^1sD^\alpha  f(x +s(y-x))\,d s.   
\end{align}
\end{subequations}
By (\ref{7_b}) it follows that  
\begin{subequations}\label{7_1c}
\begin{align}
\label{7_1c_a}
	\nabla f(y) \cdot (y-x)  = &\nabla f(x) \cdot (y-x)  +\big(  R^2 f(x +s(y-x))+ R^2f(y +s(x-y)) \big ), \\
\label{7_1c_b}
	  -f(x) = &-f(y)+\nabla f(x)\cdot (y-x)  + R^2f(y +s(x-y)).
\end{align}
\end{subequations}
\subsubsection{Modifications of the   fluxes on $\partial{\Omega_{o21}}$ }
To illustrate  the use of (\ref{7_b}) to (\ref{7_1c}) in our analysis, we consider the simple case 
of Fig. \ref{Fig_Sub_doms_Ovelap}(a). 
 Let the points  $x_{o1}\in F_{o12}$ and $x_{o2}\in F_{o21}$ be such that 
 $ x_{o2}=\mathbf{\Phi}_{o12}(x_{o1}) $ as in Fig. \ref{Fig_Sub_doms_Ovelap}(a). 
 These points  play the role of the points $x$ and $y$ in (\ref{7_b}). Then
  for a smooth function $f$ we have
 \begin{alignat}{2}\label{7_1cA}
\begin{split}
	f(x_{o1}) =& f(x_{o2})   +\nabla f(x_{o2})\cdot (x_{o1}-x_{o2})  + R^2 f(x_{o1} +s(x_{o2}-x_{o1})) \\
	         =&  f(\mathbf{\Phi}_{o12}(x_{o1})) 
	             +\nabla f(\mathbf{\Phi}_{o12}(x_{o1})) \cdot (x_{o1}-x_{o2})  + R^2 f(x_{o1} +s(x_{o2}-x_{o1})).
\end{split}
\end{alignat}
 Now denoting $r_{o12}=x_{o1}-x_{o2}$ and using the assumption that $r_{o12}=-r_{o21}$, 
 see Section \ref{overlapping_region},  we  obtain that
 $\textstyle{n_{F_{o12}}=\frac{r_{o12}}{|r_{o12}|}=-n_{F_{o21}} }$.
For keeping notation simple,  we  denote the Taylor's residuals  as 
$R^2u^*_{x_{o1}}:=R^2 u^*(x_{o1} +s(x_{o2}-x_{o1}))$ and 
$R^2u^*_{x_{o2}}:=R^2 u^*(x_{o2} +s(x_{o1}-x_{o2}))$.
Using   (\ref{7_1cA}) and interface conditions (\ref{NormalJumpOverl_F*}),  
we  modify the  fluxes in (\ref{PetrVF_wh}) as follows

\begin{subequations}\label{VF_o1_1}
\begin{alignat}{2}
\begin{split}
\int_{F_{o21}}& \rho_2\nabla u_2^*(x_{o2})\cdot n_{F_{o21}}  \phi_{2,h}\,d\sigma - 
\frac{\{\rho\}}{h}\int_{F_{o21}} (u_2^*(x_{o2})-u_1^*(x_{o2}))\phi_{2,h}\,d\sigma \\
                          =&
  \int_{F_{o21}} \frac{1}{2}\rho_2\nabla u_2^*(x_{o2}) \phi_{2,h}  \,d\sigma 
 -\frac{\{\rho\}}{h}\int_{F_{o21}}(u_2^*(x_{o2})-u_{1}^*(\mathbf{\Phi}_{o21}(x_{o2}))\phi_{2,h} \,d\sigma \\
 +&
 \int_{F_{o21}}\frac{\{\rho\}}{h}\big(|r_{o12}|\nabla u_{2}^*(x_{o2})\cdot n_{F_{o21}} +
 R^2u^*_{x_{o2}} \big ) \phi_{2,h}  \,d\sigma,
 \end{split}
 \end{alignat}
 where   $\{\rho\}=\frac{1}{2}(\rho_1+\rho_2)$. Similarly,  we have
\begin{alignat}{2}
\begin{split}
\int_{F_{o12}} \rho_1\nabla & u_1^*(x_{o1})\cdot n_{F_{o12}}  \phi_{1,h}\,d\sigma =
\int_{F_{o12}}   \rho_1\nabla u_1^*(x_{o1})\cdot n_{F_{o12}} \phi_{1,h}\,d\sigma -  
 \frac{\{\rho\}}{h}\int_{F_{o12}} (u_1^*(x_{o1})-u_1^*(x_{o1}))\phi_{1,h}\,d\sigma \\
 =
 \int_{F_{o12}} &  \rho_1\nabla u_1^*(x_{o1})\cdot n_{F_{o12}} \phi_{1,h}\,d\sigma -  
 \frac{\{\rho\}}{h}\int_{F_{o12}} (u_1^*(x_{o1})-u_2^*(\mathbf{\Phi}_{o12}(x_{o1})))\phi_{1,h}\,d\sigma \\
+
 \int_{F_{o12}}&\frac{\{\rho\}}{h}\big(|r_{o21}|\nabla u_{1}^*(x_{o1})\cdot n_{F_{o12}} +
 R^2u^*_{x_{o1}} \big ) \phi_{1,h}  \,d\sigma .
  \end{split}
\end{alignat}
\end{subequations}

\subsubsection{The global modified form}
We consider the global bilinear form $a^*(\cdot,\cdot): V_h^*\times V_{\mathbb{B}}^* \rightarrow \mathbb{R}$, which is 
formed by the  contributions of $a_{i,h}^{*}(\cdot,\cdot),\,i=1,2$
given in (\ref{PetrVF_wh}) and  the flux forms  given in (\ref{VF_o1_1}), that is 
\begin{multline}\label{VF_o2}
a^{*}(u^*,\phi_{h})=a_{2,h}^{*}(u_2^*,\phi_{2,h})+a_{1,h}^{*}(u_1^*,\phi_{1,h})=
\int_{\Omega_1^*}\rho_1\nabla u_1^*\cdot \nabla \phi_{1,h}\,dx  
+\int_{\Omega_2^*}\rho_2\nabla u_2^*\cdot \nabla \phi_{2,h}\,dx \\
-   \int_{\partial \Omega_1^* \cap \partial \Omega}\rho_1\nabla u_1^*\cdot n_{\partial \Omega_1^*} \phi_{1,h}\,d\sigma  - 
   \int_{\partial \Omega_2^* \cap \partial \Omega}\rho_2\nabla u_2^*\cdot n_{\partial \Omega_2^*} \phi_{2,h}\,d\sigma \\
   + \frac{\rho_1}{h}\int_{\partial \Omega_1^* \cap \partial \Omega}(u_1^*-u_D)\phi_{1,h}\,d\sigma 
   + \frac{\rho_2}{h}\int_{\partial \Omega_2^* \cap \partial \Omega}(u_2^*-u_D)\phi_{2,h}\,d\sigma\\
   -\int_{F_{o12}}  \Big( \rho_1\nabla u_1^*(x_{o1})\cdot n_{F_{o12}}       
  +\frac{\{\rho\}}{h}(u_1^*(x_{o1})-u_2^*(\mathbf{\Phi}_{o12}(x_{o1}))\Big)\phi_{1,h}\,d\sigma  \\ 
   -\int_{F_{o21}} \Big(\rho_2\nabla u_2^*(x_{o2})\cdot n_{F_{o21}}  
 +\frac{\{\rho\}}{h}(u_2^*(x_{o2})-u_{1}^*(\mathbf{\Phi}_{o21}(x_{o2}))\Big) \phi_{2,h}  \,d\sigma \\  
  -\int_{F_{o21}}      \frac{\{\rho\}}{h}\big(|r_{o12}|\nabla u_{2}^*(x_{o2})\cdot n_{F_{o21}} +
 R^2u^*_{x_{o2}} \big )\phi_{2,h}\,d\sigma \\
 -\int_{F_{o12}}    \frac{\{\rho\}}{h}\big(|r_{o21}|\nabla u_{1}^*(x_{o1})\cdot n_{F_{o12}} +
         R^2u^*_{x_{o1}} \big ) \phi_{1,h}  \,d\sigma.
 \end{multline}
\begin{remark}\label{modif_for_u_and_u^*}
	Note that  the exact 
	solution $u$ has similar  regularity properties to the solution $u^*$, 
	see Assumption \ref{Assumption1}, and thus  we can derive for $u$ 
	an analogous formulation as this  in (\ref{VF_o2}).   
\end{remark}
\subsubsection{The DG-IGA scheme.}
In view of (\ref{VF_o2}),  we    define the  forms
$ A_{\Omega_i^*}(\cdot,\cdot):V_h^*\times V_{\mathbb{B}}^* \rightarrow \mathbb{R}$, 
$R_{\Omega_{o21}}(\cdot,\cdot):V_h^*\times V_{\mathbb{B}}^*  \rightarrow \mathbb{R}$,
and the linear functional
$l_{f, \Omega_i^*}:V_{\mathbb{B}}^* \rightarrow \mathbb{R}$ by
\vskip -0.4cm
 \begin{subequations}\label{7_d3_b}
\begin{align}\label{Var_Form_Gap_Overl}
A_{\Omega_i^*}(u^*,\phi_h) =& \sum_{i=1}^2\Big(\int_{\Omega^*_i}\rho_i\nabla u_{i}^{*}\cdot \nabla \phi_{i,h},dx - 
   \int_{\partial \Omega^*_i \cap \partial \Omega}\rho_i\nabla u_{i}^{*}\cdot n_{\partial \Omega^*_i} \phi_{i,h}\,d\sigma \\
  \nonumber
 -&\sum_{F_{oij}\subset \partial\Omega_i^*}  \int_{F_{oij}}\rho_i\nabla u_{i}^{*}\cdot n_{F_{oij}}\phi_{i,h}  -
  \frac{\eta\{\rho\}}{h}\big(u_{i}^{*}-u_{j}^{*}\big)\phi_{i,h}\, d\sigma\Big) \\
  -&  \sum_{i=1}^2 \int_{\partial \Omega_i^* \cap \partial \Omega}
  \rho_i\nabla u_i^*\cdot n_{\partial \Omega_i^*} \phi_{i,h}\,d\sigma,\quad  
  {\ } 1\leq i\neq j\leq 2,
  \end{align}
 \vskip -0.2cm
  \begin{alignat}{2}\label{Var_Form_Residuals}
\begin{split}
R_{\Omega_{o21}}(u^*,\phi_h) =&\int_{F_{o21}}  
                           -\frac{\{\rho\}}{h}\big(|r_{o12}|\nabla u_{2}^*(x_{o2})\cdot n_{F_{o21}} +
                            R^2u^*_{x_{o2}} \big )\phi_{2,h}\,d\sigma \\
                           -\int_{F_{o12}}   &
                        \frac{\{\rho\}}{h}\big(|r_{o21}|\nabla u_{1}^*(x_{o1})\cdot n_{F_{o12}} +
                          R^2u^*_{x_{o1}} \big ) \phi_{1,h}  \,d\sigma \\
 l_{f,\Omega_i^*}(\phi_h) = & \sum_{i=1}^2 \int_{\Omega_i^*}f\phi_{i,h}\,dx,
 \end{split}
\end{alignat}
\end{subequations}
where $\eta >0$ is a parameter that is going to be determined later. 
Based on  the forms defined  in (\ref{7_d3_b}), 
we introduce the discrete  bilinear form
$A_h(\cdot,\cdot):V_{\mathbb{B}}^*\times V_{\mathbb{B}}^* \rightarrow \mathbb{R}$
 and the linear form $F_h: V_{\mathbb{B}}^* \rightarrow \mathbb{R}$
as follows
\begin{equation}\label{7_d5}
	A_h(u_{h}^{*},\phi_h) = A_{\Omega_i^*}(u_{h}^{*},\phi_h) +
	\sum_{i=1}^2\frac{\eta\rho_i}{h}\int_{\partial\Omega_i^* \cap \partial \Omega}u_{i,h}^{*} \phi_{i,h}\,d\sigma,
\end{equation}
\begin{equation}\label{7_d6}
	F_h(\phi_h) = l_{f,\Omega_i^*}(\phi_h)+
	\sum_{i=1}^2\frac{\eta\rho_i}{h}\int_{\partial\Omega_i^*\cap \partial \Omega}u_D \phi_{i,h}\,d\sigma.
\end{equation} 
Finally, the DG-IGA scheme reads as follows:
find $u_{h}^{*}\in V_{\mathbb{B}}^* $ such that
\begin{equation}\label{7_d4}
	A_h(u_{h}^{*},\phi_h) = F_h(\phi_h), \quad \text{for all} {\ }\phi_h\in V_{\mathbb{B}}^*.
\end{equation}
\textcolor{black}{
\begin{remark}\label{Cons_dgIgA_for_u}
From the relations (\ref{PetrVF}), (\ref{A_VaraForm_u_Omega*}), the Remark \ref{modif_for_u_and_u^*} and 
the forms given in (\ref{7_d5}) and in (\ref{7_d6}),  we can derive that 
	 \begin{align}\label{Cons_dgIgA_for_u_1}
	 \begin{split}
	 a_{2,h}^{*}&(u_2^*,\phi_{2,h})+a_{1,h}^{*}(u_1^*,\phi_{1,h}) 
  \\
	  = & 
	  a_{o,2}(u_2^*,\phi_{2,h}) +a_{res}(u_2^*,\phi_{2,h}) +a_{1,h}^{*}(u_1^*,\phi_{1,h}) 
\\
 =&	 A_h(u^*,\phi_h)+	R_{\Omega_{o21}}(u^*,\phi_h)  \\
 =& a_{o,2}(u,\phi_{2,h}) + a_{1,h}^*(u,\phi_{1,h}) 
\\
=& a_{2,h}(u,\phi_{2,h}) -a_{res}(u,\phi_{2,h}) +a_{1,h}^{*}(u,\phi_{1,h}) \\ 
=&  A_h(u,\phi_h)+	R_{\Omega_{o21}}(u,\phi_h)  -a_{res}(u,\phi_{2,h}) \\
=&  F_h(\phi_h), \hskip 8 cm \qquad
	 {\ }\text{for}{\ }\phi_h:=(\phi_{1,h},\phi_{2,h})\in V_{\mathbb{B}}^*.
	 \end{split}
	 \end{align}	
\end{remark}
}
Below,  we quote  few  results  that  are useful for our  error analysis. 
For the proofs we refer to
\cite{HoferLangerToulopoulos_2016_SISC}, \cite{HoferToulopoulos_IGA_Gaps_2015a} and \cite{Report_Hofer_Langer_ToulopoulosDGIGA_GapOver2016}. 
\begin{lemma}\label{lemma4}
Under the assumption  (\ref{5b_1_o}),  there exist  positive constants 
$C_1$ and $C_2$ independent of $h$ such that the estimates
 \begin{flalign}\label{7_i0_1}
  | R_{\Omega_{o21}}(u,\phi_h)|\leq  C_{1} \mathcal{K}_{o}(u)
  \|\phi_h\|_{DG}\,  h^{\lambda-0.5}, &\quad
  | R_{\Omega_{o21}}(u^*,\phi_h)|\leq  C_{2} \mathcal{K}_{o}(u^*)
  \|\phi_h\|_{DG}\,  h^{\lambda-0.5}, \quad  
   \end{flalign}
   hold    for the solutions $u^*$ and $u$,
   and $\phi_h \in V_{\mathbb{B}}^*$, 
   where   $\mathcal{K}_{o}(v)= \|\nabla v\|_{L^2(\partial \Omega_{o21})}+
   \|\sum_{|\alpha|=2}|D^\alpha v|\|_{L^2(\Omega_{o21})}$.
  \end{lemma}
  
  \begin{lemma}\label{lemma_0}
	The bilinear form $A_h(\cdot,\cdot)$ in (\ref{7_d5}) is bounded and elliptic on $V_{\mathbb{B}}^*$, i.e., 
	 there are positive constants $C_M$ and $C_m$ such that the estimates
	\begin{align}\label{B_dG_bound}
		A_h(v_h,\phi_h) \leq C_M \|v_h\|_{DG}\|\phi_h\|_{DG}
		\quad \text{and}\quad  A_h(v_h,v_h) \geq C_m \|v_h\|^2_{DG},                             
	\end{align}
	hold for all $v_h,\,\phi_h\in V_{\mathbb{B}}^*$ 
	{ provided that $\eta$ is sufficiently large.}
\end{lemma}
  \begin{lemma}\label{lemma00} Let the  assumption  (\ref{5b_1_o}) and let $\beta=\lambda-\frac{1}{2}$. Then there is a constant 
{ $C_* > 0$} depending on the parametrization mappings but  independent of $h$ such that the inequality 
\begin{align}\label{7_d7a}
A_h(v,\phi_h)\leq & { C_*} \Big( \big(\|v\|^2_{DG} + \sum_{i=1}^2
   h\,\|\rho_i^{\frac{1}{2}} \nabla v\|^2_{L^2(\partial \Omega^*_{i})}\big)^{\frac{1}{2}} \Big)
   \|\phi_h\|_{DG},
\end{align}
 holds  for all   $(v,\phi_h) \in V_h^* \times V_{\mathbb{B}}^*$ and $(v,\phi_h)\in (V+ V_{\mathbb{B}}^*)\times V_{\mathbb{B}}^*$. 
\end{lemma}
\begin{proof} 
	Recall the definition of the pair function spaces in (\ref{spaces_Omega*}). 
	In view of the form of $A_h(\cdot,\cdot)$  and applying (\ref{HolderYoung}), we have
	\begin{align}\label{7_d7aA_1}
	\Big|\sum_{i=1}^2\Big(\int_{\Omega^*_i}\rho_i\nabla v_i\cdot \nabla \phi_{i,h}\,dx \Big| \leq 
	\Big(\sum_{i=1}^2 \|\rho_i^{\frac{1}{2}}\nabla v_i\|^2_{L^2(\Omega_i^*)}\Big)^{\frac{1}{2}}
	\Big(\sum_{i=1}^2 \|\rho_i^{\frac{1}{2}}\nabla \phi_{i,h}\|^2_{L^2(\Omega_i^*)}\Big)^{\frac{1}{2}}.
	\end{align}
Now, let us first show an estimate for the  normal fluxes on $F_{oij}$. 
Since $v \in V_h^*$ 	the normal traces on the interfaces are well defined. 
Using again (\ref{HolderYoung}), we obtain
      \begin{align}\label{7_d7aA_2}
      \begin{split}
	\Big|&   \int_{F_{oij}}\rho_i\nabla v_i\cdot n_{F_{oij}}\phi_{i,h}\,d\sigma \Big|  \leq C_i
	   \int_{F_{oij}}h^{\frac{1}{2}}\Big|\rho_i^{\frac{1}{2}}\nabla v_i\Big| \, 
	   \Big(\frac{ \{\rho\}}{h}\Big)^{\frac{1}{2}} \Big|\phi_{i,h} \Big|\,d\sigma 
	   	\\
	\leq &
	C_i \Big( h^{\frac{1}{2}}\|\rho_i^{\frac{1}{2}}\nabla v_{i}\|_{L^2(F_{oij})}\Big)
	\Big(\frac{\eta\{\rho\}}{h}\|\phi_{i,h}\|^2_{L^2(F_{oij})}\Big)^{\frac{1}{2}} 
	\leq 
	C_i \Big( h^{\frac{1}{2}}\|\rho_i^{\frac{1}{2}}\nabla v_i\|_{L^2(F_{oij})}\Big)
	\|\phi_h\|_{DG},
	\end{split}
	\end{align}
	for $1\leq i\neq j \leq 2. $
	\text{Also, we have } 
	\begin{align*}
	\begin{split}
	\hskip -0.5cm 
	\Big|\frac{\eta\{\rho\}}{h}  \int_{F_{o12}}\big(v_1-v_2(\mathbf{\Phi}_{o12})\big)\phi_{1,h}\, d\sigma\Big| 
	 \leq &
	 2  \Big(   \frac{\eta\{\rho\}}{h}
	 \int_{F_{o12}} v^2_1+ v^2_2(\mathbf{\Phi}_{o12} )\frac{|J_{\mathbf{\Phi}_{o12}}|}{|J_{\mathbf{\Phi}_{o12}}|}\,d\sigma \Big)^{\frac{1}{2}} \Big(   \frac{\eta\{\rho\}}{h}\|\phi_{1,h}  \|^2_{L^2(F_{o12})}\Big)^{\frac{1}{2}}\hskip 2cm\\
	 \leq C_{J_{\mathbf{\Phi}_{o12}}}
	  \Big(   \frac{\eta\{\rho\}}{h}
	 \|v_1  \|^2_{L^2(F_{o12})} + &
	 \frac{\eta\{\rho\}}{h}\|v_2 \|^2_{L^2(F_{o21})})\Big)^{\frac{1}{2}} \Big(   \frac{\eta\{\rho\}}{h}\|\phi_{1,h}  \|^2_{L^2(F_{o12})}\Big)^{\frac{1}{2}}
	 \leq  C_{J_{\mathbf{\Phi}_{o12}}} \|v\|_{DG}\,\|\phi_h\|_{DG},\\
	 \intertext{where $|J_{\mathbf{\Phi}_{o12}}|$ 
	 is the measure of the Jacobian of $\mathbf{\Phi}_{o12}$. In the same way, we show}
	\Big|\frac{\eta\{\rho\}}{h} & \int_{F_{o21}}\big(v_2-v_1(\mathbf{\Phi}_{o21})\big)\phi_h\, d\sigma\Big| 
	 \leq C_{J_{\mathbf{\Phi}_{o21}}} \|v\|_{DG}\,\|\phi_h\|_{DG}. 
	\end{split}
	\end{align*}
Gathering together the above bounds, we  show (\ref{7_d7a}). For the case where $(v,\phi_h)\in (V+ V_{\mathbb{B}}^*) \times V_{\mathbb{B}}^*$ 
we  work similarly. 
\BLACKBOX
\end{proof}
\subsection{Discretization error analysis}
\label{DiscreizErrosection}
Next, we discuss interpolation estimates that we will use to bound the discretization error. 
We recall the definition of the pair function spaces in (\ref{spaces_Omega*}). Let 
 $v \in H^{\ell}(\cal{T}^*_{H}(\Omega))$  with $\ell\geq 2$.
Under Assumptions \ref{smooth_Phi_i}, and using the results of
\cite{LT:Bazilevs_IGA_ERR_ESti2006a} and \cite{VeigaBuffaSangalli_2014}, we can construct a quasi-interpolant 
$\Pi^{*}_h v :=(\Pi^{*}_{1,h} v_1,(\Pi^{*}_{2,h} v_2)  \in V_{\mathbb{B}}^*$ such that
the 
estimates 
\begin{align}\label{Intep_Est_1}
\begin{split}
	\sum_{i=1,2}|v-\Pi^{*}_h v|_{H^1(\Omega_i^*)} \leq &   h^{s} \sum_{i=1,2}C_{1,i}\|v\|_{H^{\ell}(\Omega_i^*)}, \\
	\sum_{i=1,2}|v-\Pi^{*}_h v|_{L^2(\partial\Omega_i^*)} \leq &  h^{s-\frac{1}{2}} \sum_{i=1,2}C_{2,i}\|v\|_{H^{\ell}(\Omega_i^*)},
	\end{split}
\end{align}
hold, where $s=\min(\ell-1,p)$ and  the $C_{1,i}$, $C_{2,i}$ depend on  $p, \mathbf{\Phi}^*_i, \theta$   but not on $h$.  
 \begin{lemma}\label{lemma7}
 Let $v \in H^{\ell}(\cal{T}^*_{H}(\Omega))$ with $\ell\geq 2$  
 and let $\Pi^{*}_h v$ be  as in (\ref{Intep_Est_1}).   
 Then  there exist  constants $C_i>0$, $i=1,2$, depending on $p,\,\mathbf{\Phi}^*_i,\,i=1,2$ and 
 the quasi-uniformity of the meshes  but not on   $h$  such that 
  \begin{align}\label{4.5.d_1}
  	\Big(\|v-\Pi^{*}_h v\|^2_{DG} +\sum_{i=1}^2h\| \rho_i^{\frac{1}{2}} \nabla (v-\Pi^{*}_h v)\|^2_{L^2(\partial \Omega_i^*)}\Big)^{\frac{1}{2}}
  	  	& \leq   \sum_{i=1}^2 C_i h^{s}\|v\|_{H^{{\color{black}\ell}}(\Omega_i^*)},
  \end{align}
  where $s=\min(\ell-1,p)$.
 \end{lemma}
  \begin{proof}
  The   estimate (\ref{4.5.d_1}) can be shown using trace inequality and the  estimates (\ref{Intep_Est_1}), 
  see  details in   Lemma 10 in \cite{LT:LangerToulopoulos:2014a}.
   See also  \cite{HoferToulopoulos_IGA_Gaps_2015a} and \cite{HoferLangerToulopoulos_2016_SISC}.
  $\BLACKBOX$  
\end{proof}

\begin{theorem}\label{Theorem_1_estimates}
Let $\beta=\lambda -\frac{1}{2}$ and   $d_o = h^\lambda$ with $\lambda \geq 1$. Let $u^*\in H^{\ell}(\cal{T}^*_{H}(\Omega))$ with $\ell\geq 2$ be the solution of the  problems in  (\ref{Artf_Vart_Problm}), and
let $u_{h}^{*}\in V_{\mathbb{B}}^*$ be the corresponding DG-IGA solution of (\ref{7_d4}). 
Then the  error estimate
\begin{align}\label{4.5_e}
\|u^*-u_{h}^{*}\|_{DG} \lesssim  h^{r} \big( \sum_{i=1}^2\|u^*\|_{H^{{\ell}}(\Omega_i^*)}\big),
\end{align}
holds,  where $r=\min(s,\beta)$ with  $s=\min(\ell-1,p)$.
\end{theorem}
\begin{proof}
Let $z_h\in V_{\mathbb{B}}^*$. We set $u^{*}_{h}-z_h=\phi_h$. The properties 
(\ref{B_dG_bound}),  (\ref{7_d7a}) 
of $A_h(\cdot,\cdot)$ and (\ref{Cons_dgIgA_for_u_1}) imply 
\begin{align}\label{4.5_eA}
	\begin{split}
	c_m&\|u^{*}_{h}-z_h\|^2_{DG} \leq 	 A_h(u^{*}_{h}-z_h,\phi_h)=
           A_h(u^*,\phi_h)+	R_{\Omega_{o21}}(u^*,\phi_h)
	  - A_h(z_h,\phi_h)\\         
=& A_h(u^*-z_h,\phi_h)+R_{\Omega_{o21}}(u^*,\phi_h) \\
\leq & { C_*} \Big( \big(\|u^*-z_h\|^2_{DG} + \sum_{i=1}^N
   h\,\|\rho_i^{\frac{1}{2}} \nabla \big(u^*-z_h\big)\|^2_{L^2(\partial \Omega^*_{i})}\big)^{\frac{1}{2}} \Big)
   \|\phi_h\|_{DG} \\
   +&C_{2} \mathcal{K}_{o}(u^*)
  \|\phi_h\|_{DG}\,  h^{\lambda-0.5},
	\end{split}
\end{align}
where the bound (\ref{7_i0_1}) has been used previously. 
Setting in (\ref{4.5_eA}) $z_h=\Pi^{*}_h u^*$, and then  
using the triangle inequality $c_m\|u^{*}_{h}-u^*\|_{DG} \leq c_m\|u^{*}_{h}-\Pi^{*}_h u^*\|_{DG} +
                               c_m\|u^{*}-\Pi^{*}_h u^*\|_{DG}$ together with the estimate in  (\ref{4.5.d_1}), we 
                               derive (\ref{4.5_e}).
  $\BLACKBOX$  
\end{proof}   
\subsubsection{Main error estimate}
	The estimate given in (\ref{4.5_e}) concerns the distance between the DG-IGA solution 
	$u_{h}^{*}\in V_{\mathbb{B}}$ and the solution
	$u^*\in H^{\ell}(\cal{T}^*_{H}(\Omega))$  of
	the  problems  in (\ref{Artf_Vart_Problm}).  Below we give an estimate between 
	the solution $u$ of (\ref{4a}) and the DG-IGA solution $u_h^*$. 
	In the proof of this result we need the following interpolation estimate for $v\in V$
		\begin{align}\label{Apend_Eq1_A}
			\Big(\|v-\Pi^{*}_h v\|^2_{DG} +\sum_{i=1}^2h\| \rho_i^{\frac{1}{2}} 
			\nabla (v-\Pi^{*}_{i,h} v)\|^2_{L^2(\partial \Omega_i^*)}\Big)^{\frac{1}{2}}
  	  	& \leq   \sum_{i=1}^2 C_i h^{s}\|v\|_{H^{{\color{black}\ell}}(\Omega_i^*)}, 
  \end{align}
  where the quasi-interpolant $\Pi_h^*v=(\Pi_{1,h}v, \Pi_{2,h}v)$ is defined in (\ref{Intep_Est_1}) and  $s=\min(\ell-1,p)$.
The proof of (\ref{Apend_Eq1_A}) is provided in the Appendix. 
\textcolor{black}{
\begin{theorem}[main error estimate]\label{mainTheormErrEstm}
	 Let  $u$ be the solution of (\ref{4a}) and let  Assumption \ref{Assumption1} with $\ell \geq 2$.  We suppose further that $d_o=h^\lambda,\,\lambda \geq 1$ is the width of $\Omega_{o21}$. 
	 The following error estimate holds
	\begin{align}\label{mainErrEstm_0}
	\|u-u^*_{h}\|_{DG} \leq \widetilde{C}\big(h^{s} \sum_{i=1}^2
	\big(\|u\|_{{H}^{\ell}(\Omega_i^*)}+\|u^*_i\|_{{H}^{\ell}(\Omega_i^*)})
	       +d_o\|f\|_{L^2(\Omega)} +h^{\beta}\big(\mathcal{K}_o(u)+\mathcal{K}_o(u^*)\big)\Big),	
	\end{align}
where $\beta=\lambda -\frac{1}{2}$,  $s=\min(\ell-1,p)$, the constant $\widetilde{C}$ 
depends on the constants  in (\ref{4.5.d_1}),  (\ref{7_d7a}) and (\ref{B_dG_bound}), and 
 $\mathcal{K}_o$ has the form given in  Lemma  \ref{lemma4}. 
\end{theorem}
\begin{proof} 
          Let $z_h\in V^*_{\mathbb{B}}$  and let $\phi_h=u_h^{*}-z_h$.
          By the definition of the discrete DG-IGA scheme in (\ref{7_d4}),  the properties of $A_h(\cdot,\cdot)$ and the Remark \ref{Cons_dgIgA_for_u}  we have
         \begin{align}\label{mainErrEstm_1}
          \begin{split}
          c_m&\|u_h^*-z_h\|^2_{DG} \leq           A_h(u^{*}_{h} - z_h,\phi_h) - A_h(u^*,\phi_h) - R_{\Omega_{o21}}(u^*,\phi_h) + F_h(\phi_h)
          -A_h(\Pi_h^*u^*,\phi_h)+A_h(\Pi_h^*u^*,\phi_h)\\ 
          =\,&
          A_h(u_h^*-\Pi_h^* u^*,\phi_h) +  A_h(u^*-\Pi_h^* u^*,\phi_h) \\
          +\,&  A_h(-z_h,\phi_h ) + 
          A_h(u,\phi_h) -a_{res}(u,\phi_{2,h})+R_{\Omega_{o21}}(u,\phi_h)-R_{\Omega_{o21}}(u^*,\phi_h)\\
          =\, &  A_h(u_h^*-\Pi_h^* u^*,\phi_h) +  A_h(u^*-\Pi_h^* u^*,\phi_h)+ A_h(u-z_h,\phi_h)\\
          -\,& a_{res}(u,\phi_{2,h})+R_{\Omega_{o21}}(u,\phi_h)-R_{\Omega_{o21}}(u^*,\phi_h) \\
\leq   \, & { C_M}  \|u^*_h-\Pi_h^*u^*\|_{DG}\|\phi_h\|_{DG}  
\hskip 6cm \qquad \qquad   
by\, {(\ref{7_i0_1}),(\ref{B_dG_bound}), (\ref{7_d7a}), (\ref{fphi_hOmega_o1})} \\
   +\,& { C_*} \Big( \big(\|u^*-\Pi_h^*u^*\|^2_{DG} + \sum_{i=1}^N
   h\,\|\rho_i^{\frac{1}{2}} \nabla \big(u^*-\Pi_h^*u^*\big)\|^2_{L^2(\partial \Omega^*_{i})}\big)^{\frac{1}{2}} \Big)
   \|\phi_h\|_{DG} \\
 + \,& { C_*} \Big( \big(\|u-z_h\|^2_{DG} + \sum_{i=1}^N
   h\,\|\rho_i^{\frac{1}{2}} \nabla \big(u-z_h\big)\|^2_{L^2(\partial \Omega^*_{i})}\big)^{\frac{1}{2}} \Big)
   \|\phi_h\|_{DG} \\
 +\,& c_2d_o\|f\|_{L^2(\Omega)}\|\phi_h\|_{DG} +C_{2} \big(\mathcal{K}_{o}(u^*) +\mathcal{K}_{o}(u)\big)
  \|\phi_h\|_{DG}\,  h^{\beta}
                      \end{split}
  \end{align}
   Setting  $z_h=\Pi^{*}_h u$ into (\ref{mainErrEstm_1}),  using  (\ref{4.5_eA}), (\ref{4.5.d_1}), 
   and (\ref{Apend_Eq1_A}) and gathering together the similar terms
   we deduce that
   \begin{align}\label{mainErrEstm_2}
          \begin{split}
          c_m\|u_h^*-\Pi^{*}_h u\|_{DG} \leq & 
   \sum_{i=1}^2 C_i h^{s}\|u\|_{H^{{\color{black}\ell}}(\Omega_i^*)} +
   \sum_{i=1}^2 C_i h^{s}\|u^*_i\|_{H^{{\color{black}\ell}}(\Omega_i^*)} \\
   + &
   c_2d_o\|f\|_{L^2(\Omega)} +C_{2} \big(\mathcal{K}_{o}(u^*) +\mathcal{K}_{o}(u)\big)\,  h^{\beta}
   \end{split}
  \end{align}
 Applying the triangle inequality 
\begin{align}\label{mainErrEstm_3}
\|u-u^*_{h}\|_{DG} \leq \|u-\Pi^{*}_h u\|_{DG}+\|\Pi^{*}_h u-u^*_{h}\|_{DG},
\end{align}
the desired estimate follows.
\BLACKBOX
\end{proof}
}
\section{Implementation and Numerical tests}
\label{Section_numerics}
\subsection{Implementation remarks}
\label{Implementation}
\textcolor{black}{
In this paragraph  we focus on the implementation of the proposed scheme for both two and 
three dimensional problems. For simplicity of the presentation we first discuss the case of having two patches.
 Afterwards, we explain how the same ideas  can be generalized to the multipatch case.
 \\
 \begin{figure}
 \begin{subfigmatrix}{2}
  \subfigure[]{\includegraphics[clip,trim={0.5cm 3cm 2.5cm 3cm},width=.45\textwidth]{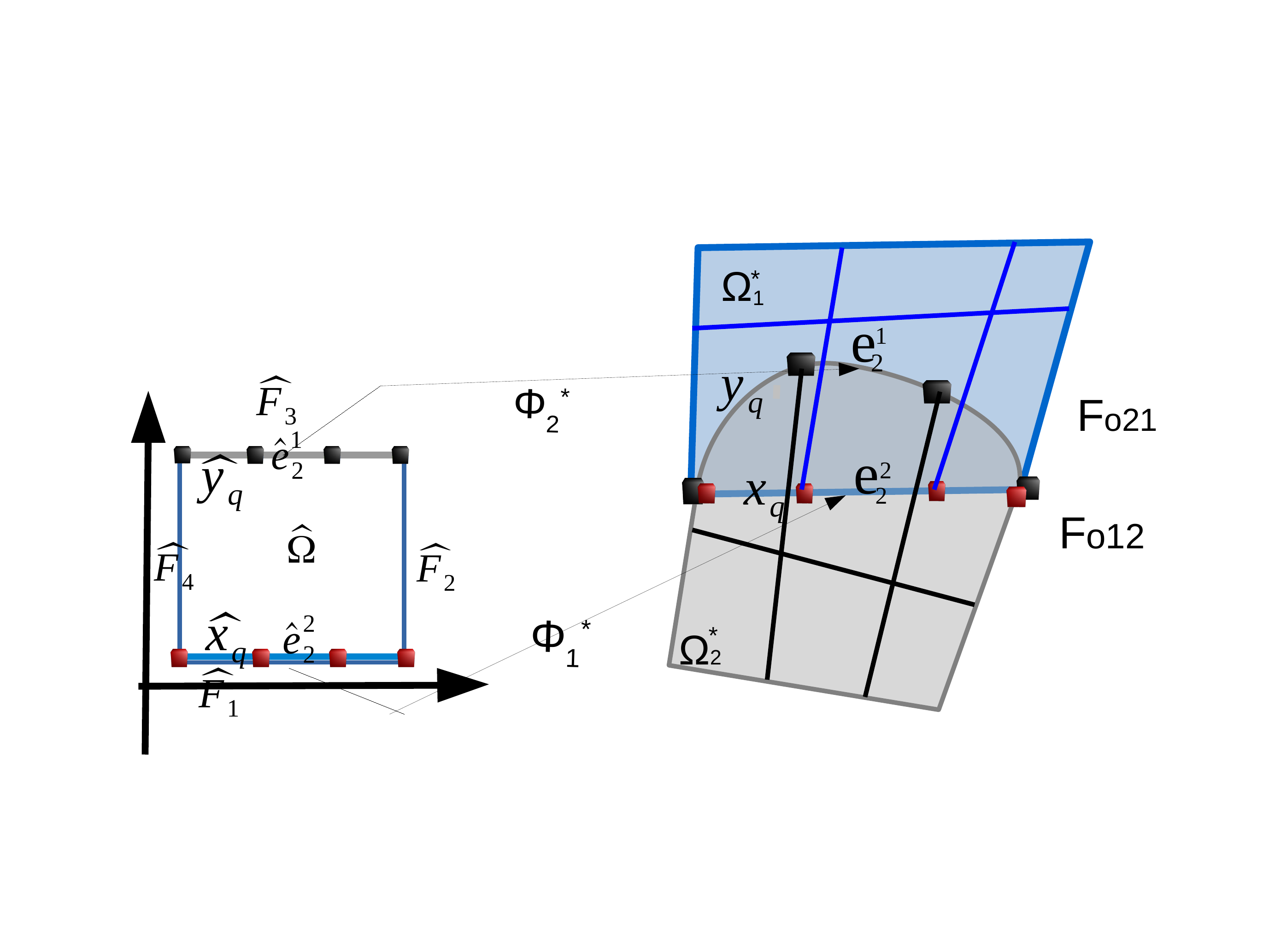}}
  \subfigure[]{\includegraphics[trim={0cm 3cm 2.5cm 3cm},width=.4\textwidth]{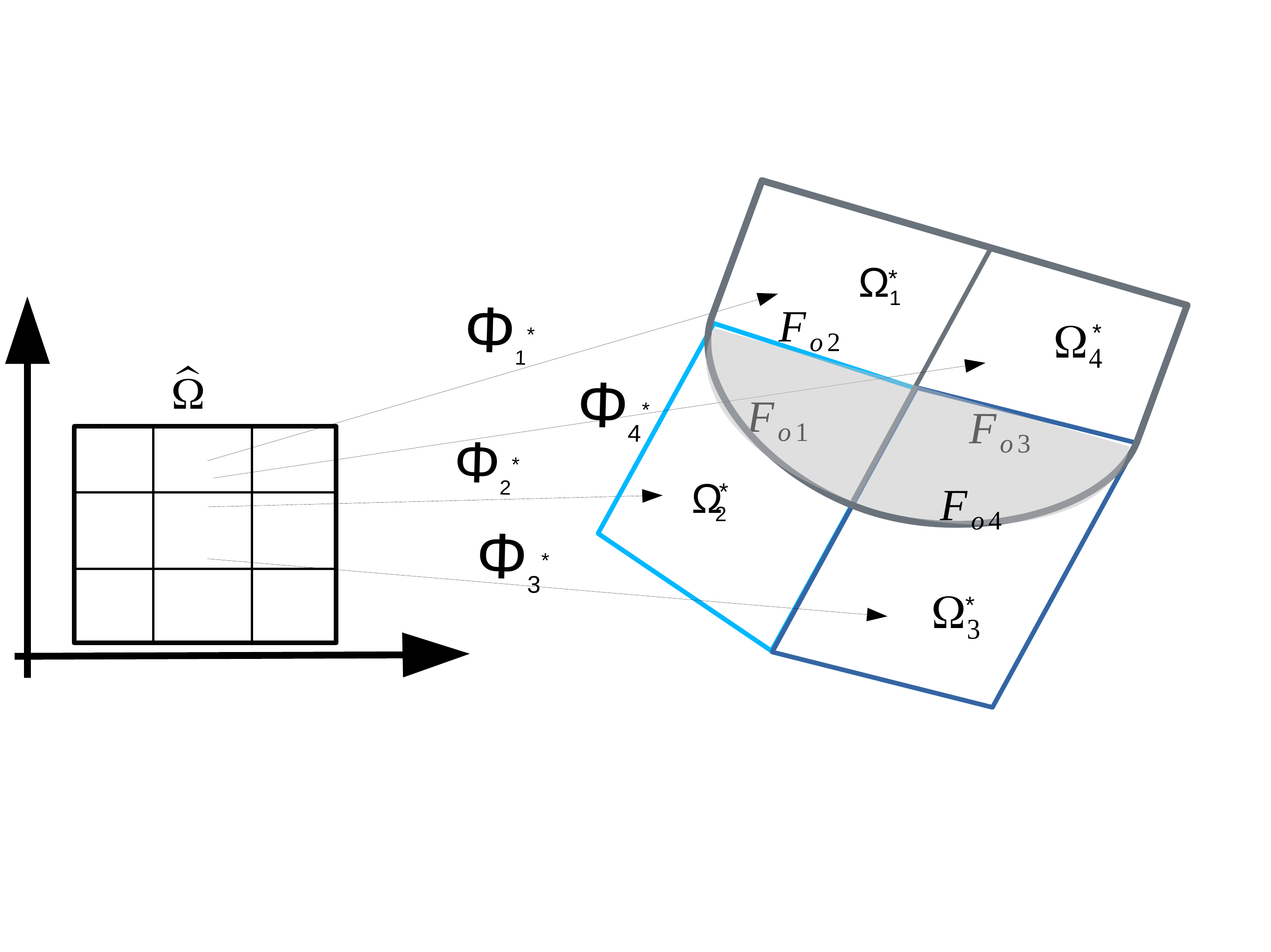}}
\end{subfigmatrix}
 \caption{ (a) Configuration of the faces and the edges on $\partial \Omega_{o12}$ and their
 corresponding edges on $\partial \widehat{\Omega}$ which are used to compute the interface integrals,  (b) an example of an overlapping region with more than two faces. The relative edges on the opposite faces must again much.   
         }
 \label{Fig_imag_NonMatcInterf_2}
\end{figure}
 Initially, we  consider interfaces with matching meshes, i.e., the number of edge elements 
 on $F_{o21}$ is the same as the number on $F_{o12}$, as shown in Fig. \ref{Fig_imag_NonMatcInterf_2}. 
 \\
 For the computation of the numerical flux terms of the 
  DG-IGA scheme given  in (\ref{Var_Form_Gap_Overl}), a Gauss quadrature rule
  is applied on every edge.
 The first term of the numerical flux can be directly computed by using the Gauss rule 
  and the related Jacobian term. For the computation of the jump terms we must know the 
  diametrically opposite  edge  and the associated quadrature point that are located on the
  other interface. We could proceed to this direction by constructing and using the mappings 
  $\mathbf{\Phi}_{o21}$ and $\mathbf{\Phi}_{o12}$ given in (\ref{parmatric_lro}) and (\ref{parmatric_rlo}) respectively. 
  For the practical implementation, it would be preferable to proceed without the construction of these mappings.
 \\
 We first  assign the edges belonging to $F_{o21}$ to the edges belonging to $F_{o12}$, for the example given 
 in Fig. \ref{Fig_imag_NonMatcInterf_2}(a), the edge $e_2^1$ of $F_{o21}$ is assigned to $e_2^2$ of $F_{o12}$.  
  In Fig. \ref{Fig_imag_NonMatcInterf_2}(a) the Gauss point are denoted by 
  $x_q$ and $y_q$ correspondingly. 
   The edge  $e_2^1$ is the image of the 
  edge $\hat{e}_2^1$ under the parametrization $\mathbf{\Phi}_2^*$, and also
  the edge $e_2^2$ is the image of the 
  edge $\hat{e}_2^2$ under the parametrization $\mathbf{\Phi}_1^*$. Hence, the Gauss rule 
   is transformed back to boundary edges of the parametric domain, and for every Gauss point $\hat{y}_q$ there is always 
   a corresponding Gauss point  on the other associated edge to perform the numerical integration. 
For the configuration given in  Fig.  \ref{Fig_imag_NonMatcInterf_2}(a), the other associated edge is located on face $\hat{F}_1$ and the corresponding Gauss point is denoted by $\hat{x}_q$. 
   Thus, having defined the quadrature points on the boundary
  edges of $\widehat{\Omega}$, we can compute the interface terms of the   numerical flux
  of the DG-IGA scheme. 
  \\
    Note that the above approach is quite simple
    and it follows the same ideas that we use for computing the numerical fluxes 
  in the case   of matching parametrized interfaces.
  It can be also applied for the case of having
  gap regions between the patches. The advantage of implementing  this approach is that we can
  develop a flexible DG-IGA  code which can treat patch unions  with matching and nonmatching interfaces in a similar way. Note also that the previous approach can be easily combined with the 
   adaptive  numerical  quadrature methods presented in \cite{Agns_Juttler_CurvSurf2016}, 
   in order to discretize  the problem using  non-matching structured meshes on the overlapping faces. 
  \\
  Overlapping regions with boundary consisting of more than two faces are shown in Fig. \ref{Fig_imag_NonMatcInterf_2}(b).   We consider again the case where the maximum number of the overlapping patches is two.  For the  example shown in Fig. \ref{Fig_imag_NonMatcInterf_2}(b) the domain has four  patches and the boundary of the overlapping region  is compromised of the four faces $F_{oi},\,i=1,\ldots,4$.   
  Anyway, the evaluation
  of the interface numerical fluxes in this case  needs   more work. We first find the faces that form the 
  boundary of the overlapping regions. Then between these faces, we determine those that are diametrically opposite, and 
  we continue  following the procedure described in the previous paragraph. This type of overlapping regions are discussed in the numerical Example 3.
   \\
   It is clear that through a segmentation and parametrization procedure, overlapping regions with 
   more  complicated shapes  than the shapes in the examples shown here can exist,  e.g., more than two overlapping patches, T-joint faces on the  boundary, see, e.g., \cite{HLT:PauleyNguyenMayerSpehWeegerJuettler:2015a}.  In an ongoing work we are extending the present methodology to treat these cases. We also  are  constructing  domain-decomposition methods, \cite{HLT:HoferLanger:2016a},  on these type of multipatch representations and we are discussing the influence of the size    of the overlapping region on the performance of the proposed methods. The first results of this work are included in  \cite{Report_Hofer_Langer_ToulopoulosDGIGA_GapOver2016}.  \\
   Finally, we mention that  during the investigation  of the proposed methodology in Section 3, we considered simple interior penalty fluxes on $\partial \Omega_{o21}$.  
   For the  performance of the numerical examples below, we have implemented the corresponding symmetric numerical fluxes, i.e., 
   $\textstyle{-\int_{F_{o12}} \frac{1}{2}\big(\rho_1 \nabla u_{1,h} +\rho_2 \nabla u_{2,h}(\mathbf{\Phi}_{o12}) \big)
   \cdot n_{F_{o12}} \phi_{1,h} +\frac{\eta\{\rho\}}{h} \Big(u_{1,h} -u_{2,h}(\mathbf{\Phi}_{o12}) \big) \phi_{1,h}\,ds.} $, see \cite{LT:LangerToulopoulos:2014a}, \cite{HoferLangerToulopoulos_2016_SISC}.
   }
\subsection{Numerical Examples}
In this section, we perform several numerical tests  with different shapes of  overlapping regions 
as well as combinations with non-homogeneous diffusion coefficients for two- and three- dimensional problems. 
We investigate  the order of accuracy of the DG-IGA scheme proposed in (\ref{7_d5}).  All examples have been performed using  second degree ($p=2$)  B-spline spaces. 
We present the asymptotic behavior of  the error convergence rates  for  widths $d_o=h^\lambda$  
 with  $\lambda \in \{1,2,2.5,3\}$. Every example has been solved
 applying several mesh refinement steps with $\ldots,h_i,h_{i+1},\ldots,$ satisfying Assumption \ref{Assumption2}.
 The numerical convergence rates $r$ have been  computed by the ratio 
 $r =\textstyle{\frac{\ln(e_i/e_{i+1})}{\ln (h_i /h_{i+1} )}}, \,i=1,2,\ldots$, where the error $e_i:=\|u-u^*_h\|_{DG}$ is always computed  
on the  meshes $\cup_{i=1}^2 T^{(i)}_{h_i,\Omega_i^*}$.
We mention that, in the test cases, we use highly smooth solutions in each patch, i.e., $p+1 \leq \ell$, 
and therefore the order  $s$ in  (\ref{4.5_e}) and (\ref{mainErrEstm_0}) becomes $s = p$.
    The predicted values of  power  $\beta$, the order $s$ and the expected convergence rate  $r$, for 
    several values of $\lambda$,  are displayed in Table \ref{table_value_r}. 
    In any test case, the   overlap regions  are artificially created by moving 
the control points, which are related to the interfaces $F_{ij}$, 
in the direction of $n_{F_{ij}}$ or of $-n_{F_{ij}}$.  
\par
All tests have been performed in G+SMO \cite{gismoweb},
which is a generic object-oriented C++ library for IGA computations,
\cite{HLT:JuettlerLangerMantzaflarisMooreZulehner:2014a,LangerMantzaflarisMooreToulopoulos_IGAA_2014a}. 
In Section 3, we developed and provided a rigorous analysis for the DG-IGA method (\ref{7_d4}) which includes a 
 non-symmetric numerical flux. 
In the materialization of the method, we utilized the associated symmetrized version the numerical flux, \cite{Rivierebook}.
For solving the resulting linear system, we use the DG-IETI-DP method presented  in  \cite{HLT:HoferLanger:2016a}, see also \cite{HLT:Hofer:2018a} for an analysis of the method and \cite{HLT:Hofer:2017a} for results on parallel scalability. 
\par
{Although in the analysis, we consider  meshes with similar quasi-uniform patch-wise properties,
it is known that the introduction of DG techniques
on the subdomain interfaces makes the use of non-matching and non-uniform meshes  easier, see \cite{LT:LangerToulopoulos:2014a}.
Keeping a constant linear relation between the sizes of the different patch meshes, 
the approximation properties
of the method are not affected,  \cite{LT:LangerToulopoulos:2014a}. 
In the examples  below, we exploit this advantage of the DG methods and 
first solve two-dimensional problems considering
non-matching meshes.  The   convergence rates are expected  to be the same as those displayed in Table \ref{table_value_r}. 

\begin{table}
  \centering
  \begin{tabular}{|c||c|c|c|c|}
 \hline
  & \multicolumn{4}{|c|}{B-spline degree $p$   }\\ \hline  
  & \multicolumn{4}{|c|}{Smooth solutions, $u\in H^{\ell\geq p+1}$}  \\ \hline
   $d_o=h^\lambda$       &$\lambda=1$& $\lambda=2$ &$\lambda=2.5$ &$\lambda=3$\\\hline
  $\beta:=$              &0.5        &  1.5        & 2                    &2.5\\   \hline
  $s:=$                  &$p$        &  $p$        & $p$                   &$p$\\   \hline
  $r:=$                  &$0.5$      &  $1.5$      & $\min(p,\beta)$       &$\min(p,\beta)$\\ 
 \hline
\end{tabular}
\caption{The values of the expected rates $r$ as they result from  estimate (\ref{mainErrEstm_0}). }
\label{table_value_r}
\end{table}
\subsection{Two-dimensional numerical examples}
The control points with the corresponding knot vectors of the domains given in Example 1-3 are available under the names \verb+yeti_mp2+, \verb+12pSquare+ and \verb+bumper+ as \verb+.xml+ files in 
G+SMO\footnote{G+SMO: https://www.gs.jku.at/trac/gismo}.

\paragraph{Example  1: uniform diffusion coefficient $\rho_i=1,\, i=1,\ldots,N$.} 
The first numerical example is a simple test case demonstrating
the  applicability of the proposed technique for constructing the 
DG-IGA scheme on segmentations 
including   overlaps  with general shape. 
The domain  $\Omega$ with the $N=21$ subdomains $\Omega_i^*$ 
 and the initial mesh  are shown in Fig.~\ref{Fig1_Test_1Gaps}(a). We note that 
 we consider non-matching meshes across the interior  interfaces. 
The Dirichlet boundary condition and the right hand side $f$ are determined by the exact solution 
$u(x,y)=\sin(\pi(x+0.4)/6)\sin(\pi(y+0.3)/3)+x+y $. 
 In this example, we consider the homogeneous diffusion case, i.e., 
$\rho_i=1$ for all  $\Omega_i^*,\,i=1,\ldots,N$.
\par
 We performed {four} groups of  computations, where for every group 
 the   maximum size of  $d_o$ was defined to be $\mathcal{O}(h^\lambda)$, 
 with $\lambda \in \{1,2, 2.5,3\}$. 
 In Fig.~\ref{Fig1_Test_1Gaps}(b) we present the discrete solution for $d_0 = h$.
 Since we are using second-order ($p=2$) B-spline space, based on Table 
 \ref{table_value_r}, we expect optimal convergence rates for $\lambda =2.5$ and $\lambda = 3$.  
 The numerical convergence rates for several levels of mesh refinement are plotted in Fig.~\ref{Fig1_Test_1Gaps}(c). They are in very good	agreement with the theoretically predicted estimates given in Theorem \ref{mainTheormErrEstm}, 
 see also Table \ref{table_value_r}. We observe that we have optimal rates $r$ for the cases where $\lambda \geq 2.5$ and
 sub-optimal for the rest values of $\lambda$.
   \begin{figure}[h]
   \begin{subfigmatrix}{4}
 \subfigure[]{\includegraphics[width=4.0cm, height=5.25cm]{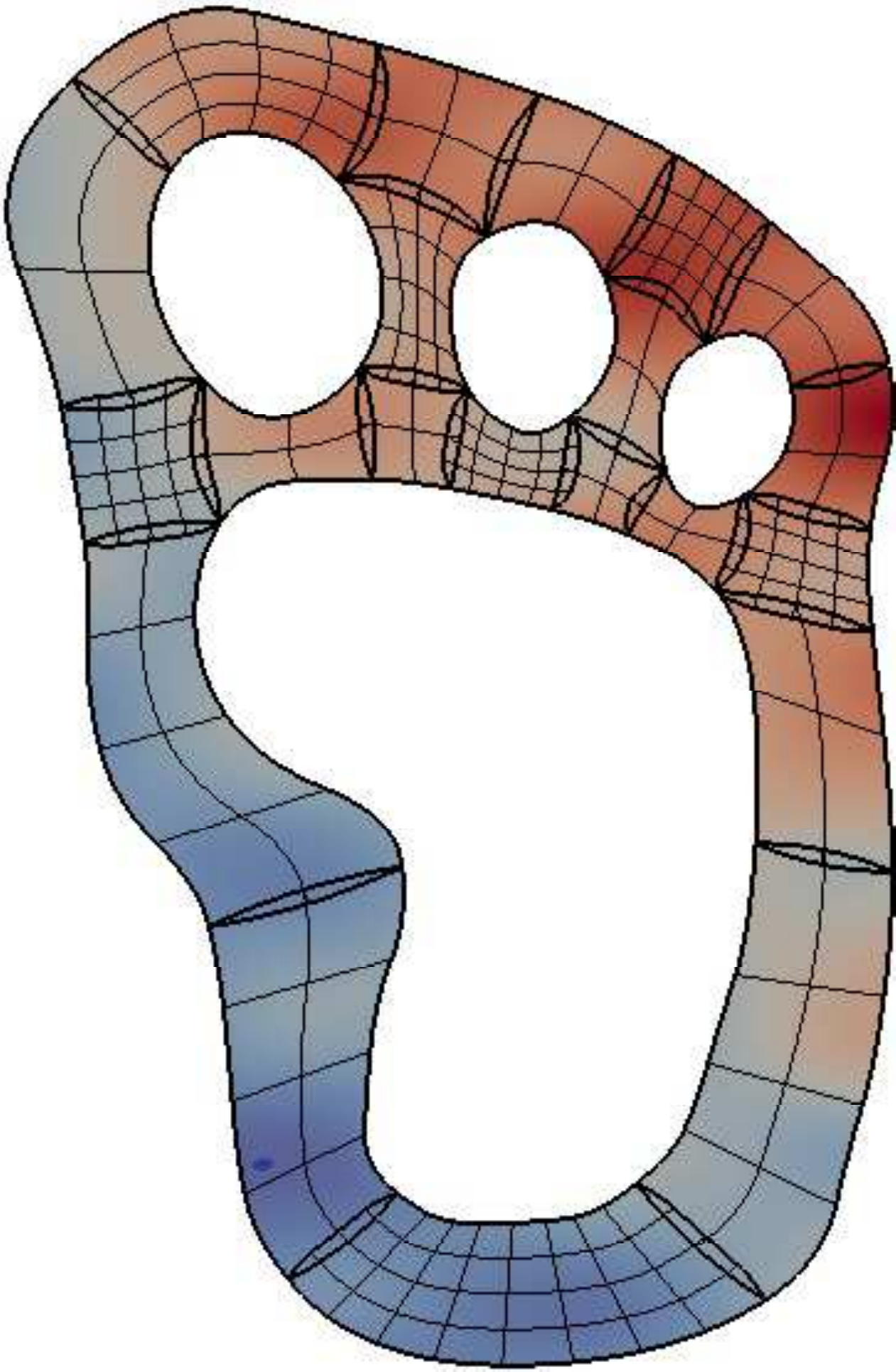}}
 \subfigure[]{\includegraphics[width=4.0cm, height=5.25cm]{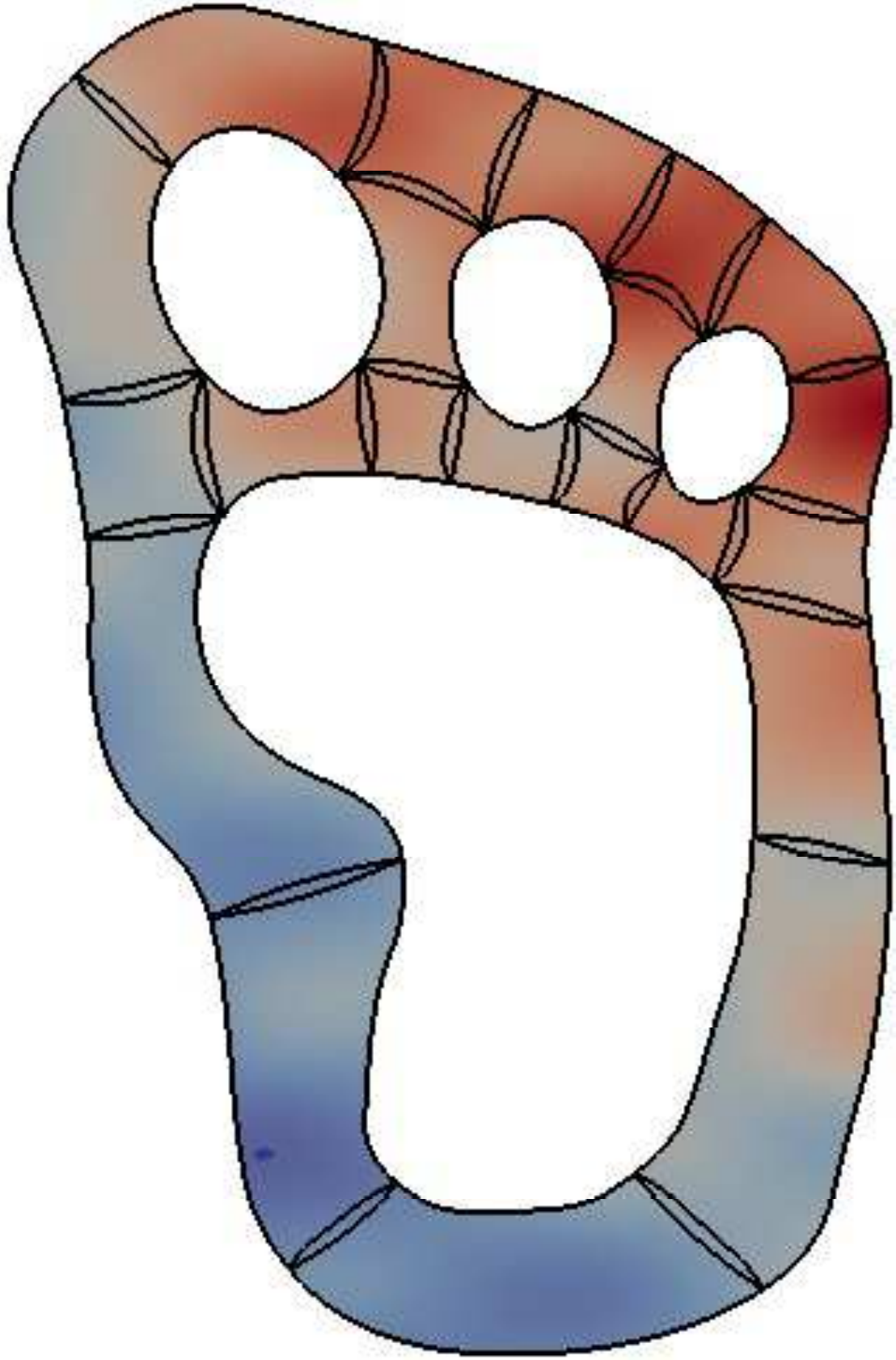}}
 \subfigure[]{{\includegraphics[width=4.25cm, height=4.5cm]{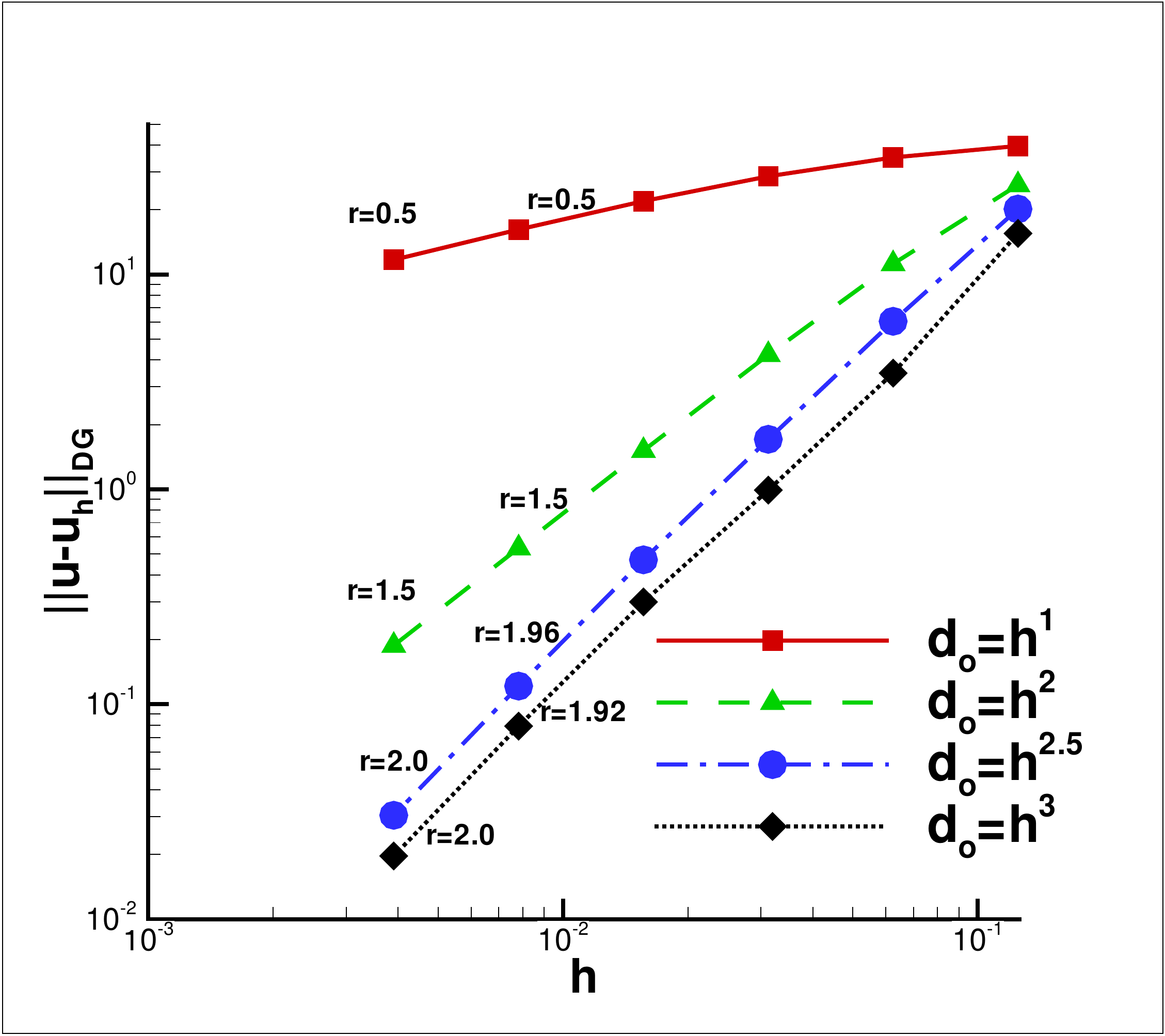}}}
  \end{subfigmatrix}
   \caption{Example 1: (a) The patches $\Omega_i^*$ with the initial non-matching  meshes and the contours of the exact solution. 
    (b) The contours of the $u^*_h$ solution for $d_o = h$. 
    (c)    The convergence rates for the different values of $\lambda$.
   }
   \label{Fig1_Test_1Gaps}
 \end{figure}
 \paragraph{Example  2: different diffusion coefficients $\rho_1\neq \rho_2$.}  
 In the second example,  we consider a rectangular domain
 $\Omega$, that is  described as a union of  $N=12$ patches, see Fig.~\ref{Fig2_Test_2Gaps}(a).
 Here,  we study the case of having smooth solutions in each $\Omega_i^*$
 but discontinuous coefficient,  i.e.,  we set $\rho_i = 3\pi/2$ for the patches belonging to
 half plane $x\leq 0$ and we set $\rho_i= 2$ for the rest patches  
 according to the pattern in Fig.~\ref{Fig2_Test_2Gaps}(a). 
 By this example,  we numerically validate the 
 predicted convergence rates on $\mathcal{T}_H^*$ with  overlaps,  
for the case of having  smooth solutions and discontinuous coefficient $\rho$. 
 The exact solution is given by the formula
\begin{align}\label{NE_1}
u(x,y)=
    \begin{cases}
	      \sin(\pi(2x+y)) &\text{ if } x<0 \\
	      \sin(\pi(\frac{3\pi}{2} x+y)) &\text{ otherwise}. 
  \end{cases} 
\end{align} 
The  boundary conditions and the source function $f$ are  determined by (\ref{NE_1}).  
Note that, we have  $\llbracket u \rrbracket |_{F_{ij}}=0$ as 
well $\llbracket \rho \nabla u \rrbracket |_{F_{ij}}\cdot n_{F_{ij}}=0$ for 
 all the interior physical  interfaces $F_{ij}$.  
\par
The problem has been solved   on a sequence of meshes  with $h_0,...,h_i,\,h_{i+1},...$,
   following a sequential refinement process, i.e., $h_{i+1}=\frac{h_i}{2}$,
   where we set  $d_o=h_{i}^\lambda$, with $\lambda \in\{1,2,2.5,3\}$. 
   For the numerical tests, we use    B-splines of the degree $p=2$.   
   Hence, we expect optimal rates for $\lambda \geq  2.5$.  
   In Fig.~\ref{Fig2_Test_2Gaps}(b) the approximate solution $u^*_h$ is presented on a relative coarse mesh  with $d_o = 0.06$.  
   The results of the computed rates are presented in  Fig.~\ref{Fig2_Test_2Gaps}(c). For all test cases, we can observe  that our theoretical results presented in Table \ref{table_value_r} are confirmed. 

  \begin{figure}
   \begin{subfigmatrix}{3}
 \subfigure[]{\includegraphics[width=5.5cm, height=5.75cm]{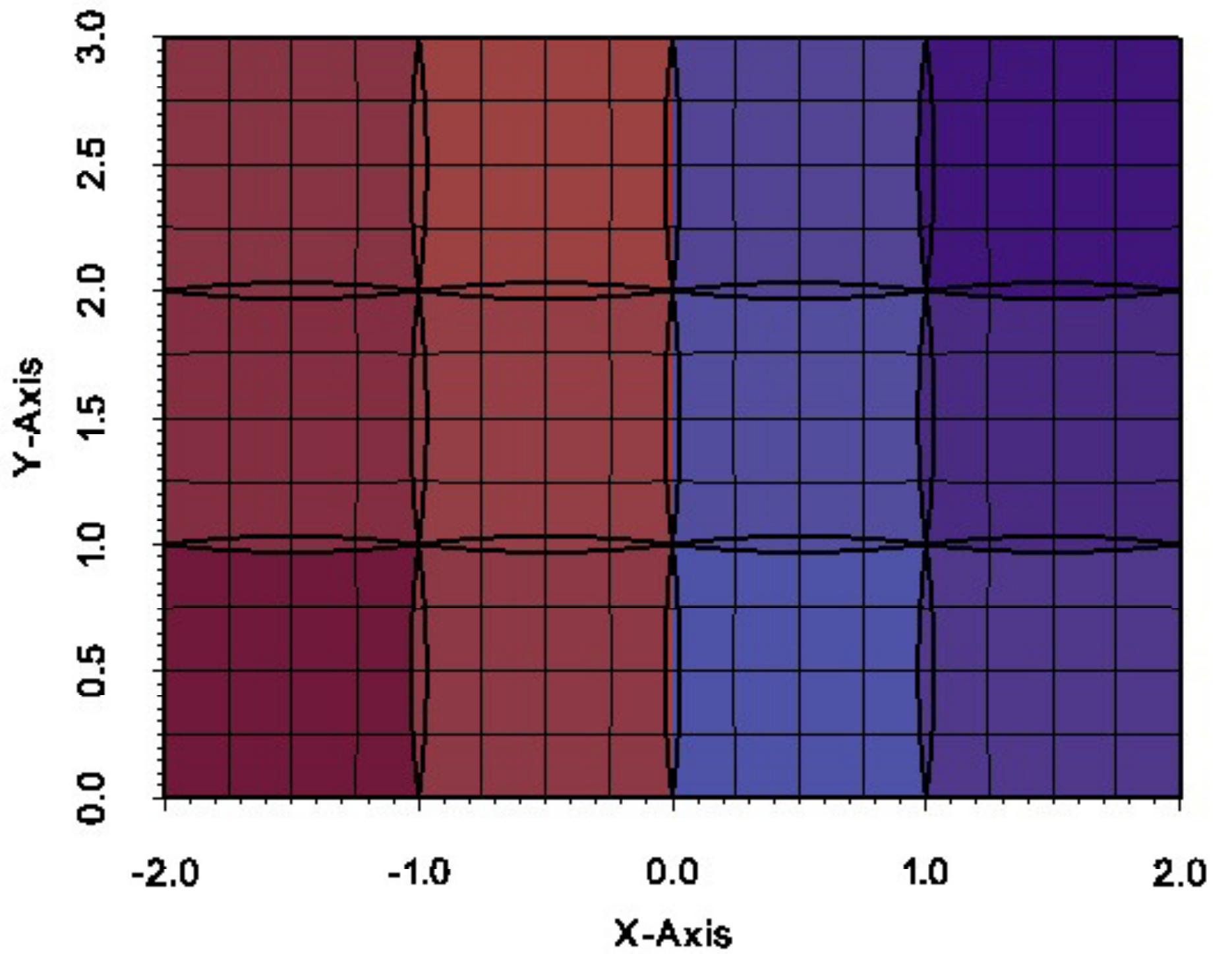}}
 \subfigure[]{\includegraphics[width=5.5cm, height=5.75cm]{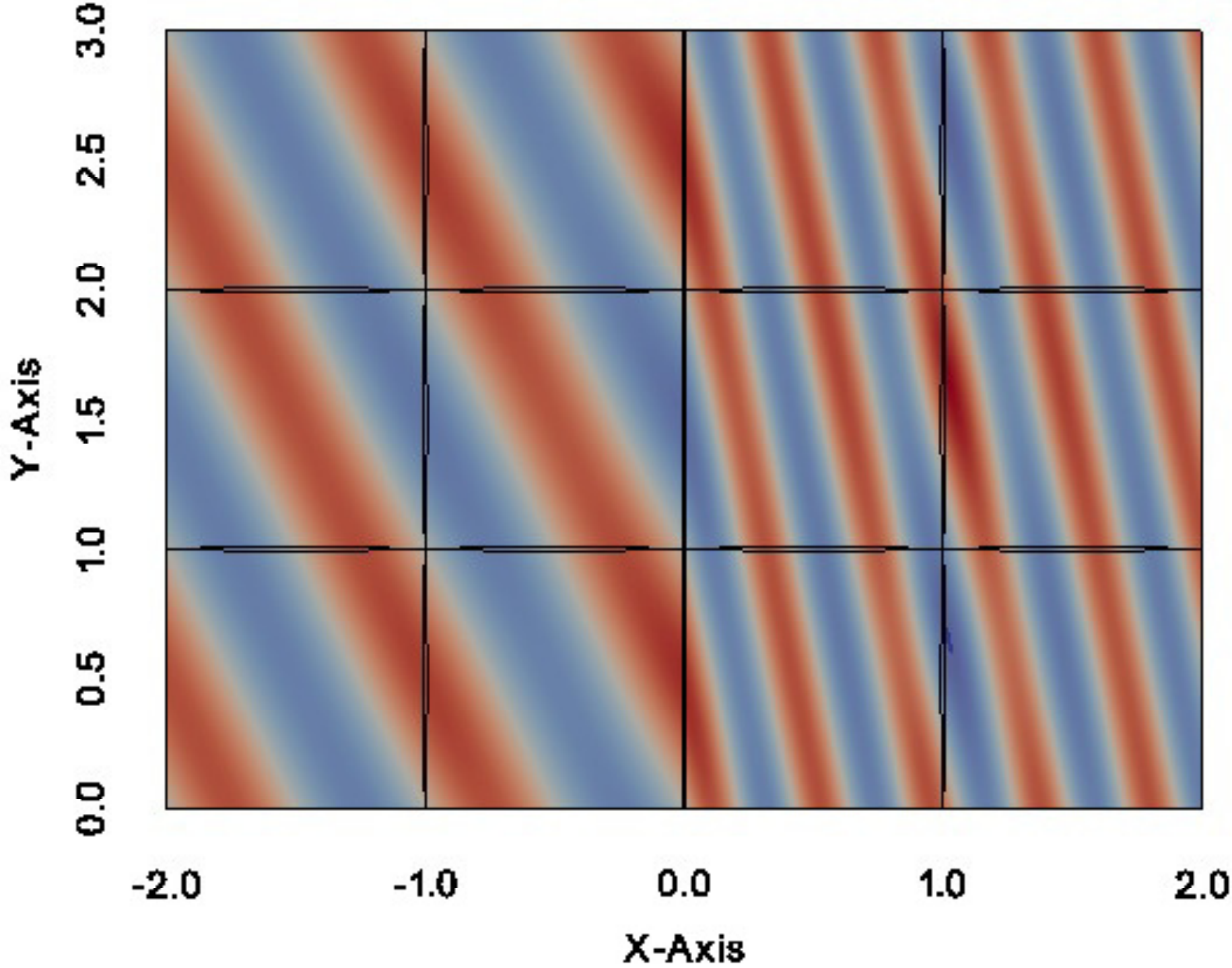}}
  \subfigure[]{\includegraphics[width=5.cm, height=5.75cm]{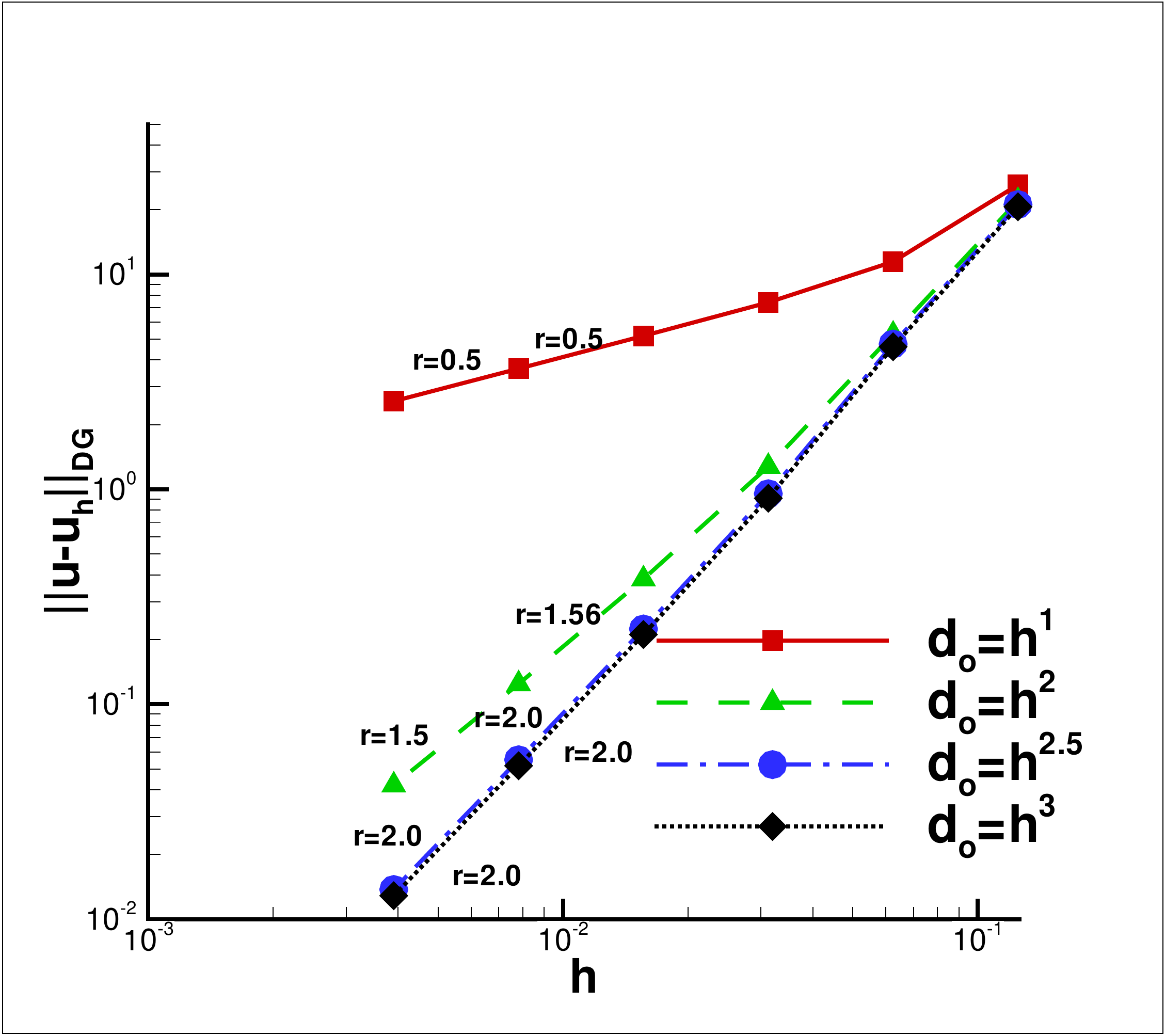}}
  \end{subfigmatrix}
   \caption{Example 2: (a) The overlapping patches $\Omega_i^*$  and the pattern of diffusion coefficients $\rho_i$, 
                       (b) The contours of $u^*_h$ on every  $\Omega_i$ computed  with $d_0=0.06$,
                       (c) The convergence rates for the 4 choices of $\lambda$. }
   \label{Fig2_Test_2Gaps}
 \end{figure}
\paragraph{Example  3: overlapping regions  with more than two faces. }  
\textcolor{black}{
The proposed method is now applied to a more complicated overlapping boundary with multiple  faces.
The geometric description  of the problem in shown in Fig. \ref{Fig_MultipatchTest}(a), the domain 
is decomposed into four patches and the overlapping region is defined by four interfaces. The exact solution is given by 
\begin{align}\label{NE_4patches}
u(x,y)= \sin(\pi(x+0.4))\sin(2\pi(y+0.3))+x+y
\end{align} 
The diffusion coefficient is globally constant, i.e., $\rho=1$,  the right-hand side $f$ and the Dirichlet boundary conditions are manufactured by the solution (\ref{NE_4patches}). 
We solved the problem using B-splines of degree $p=2$.  
In Fig. \ref{Fig_MultipatchTest}(b), we present the contours of the DG-IGA solution $u_h^*$  computed on the second mesh in a sequence. The corresponding error convergence results for the four values of $\lambda$, 
i.e., $\lambda \in \{1,2,2.5,3\}$, are given in Fig. \ref{Fig_MultipatchTest}(c). We can observe the 
suboptimal behavior of the rate for $\lambda=1$ and $\lambda=2$ as we move to the last mesh levels. 
On the other hand
we have optimal rates for the rest values  of $\lambda$. The numerical rates for all $\lambda$ cases are in agreement with the theoretical results. 
}
\begin{figure}
   \begin{subfigmatrix}{3}
 \subfigure[]{\includegraphics[width=6.0cm, height=5.75cm]{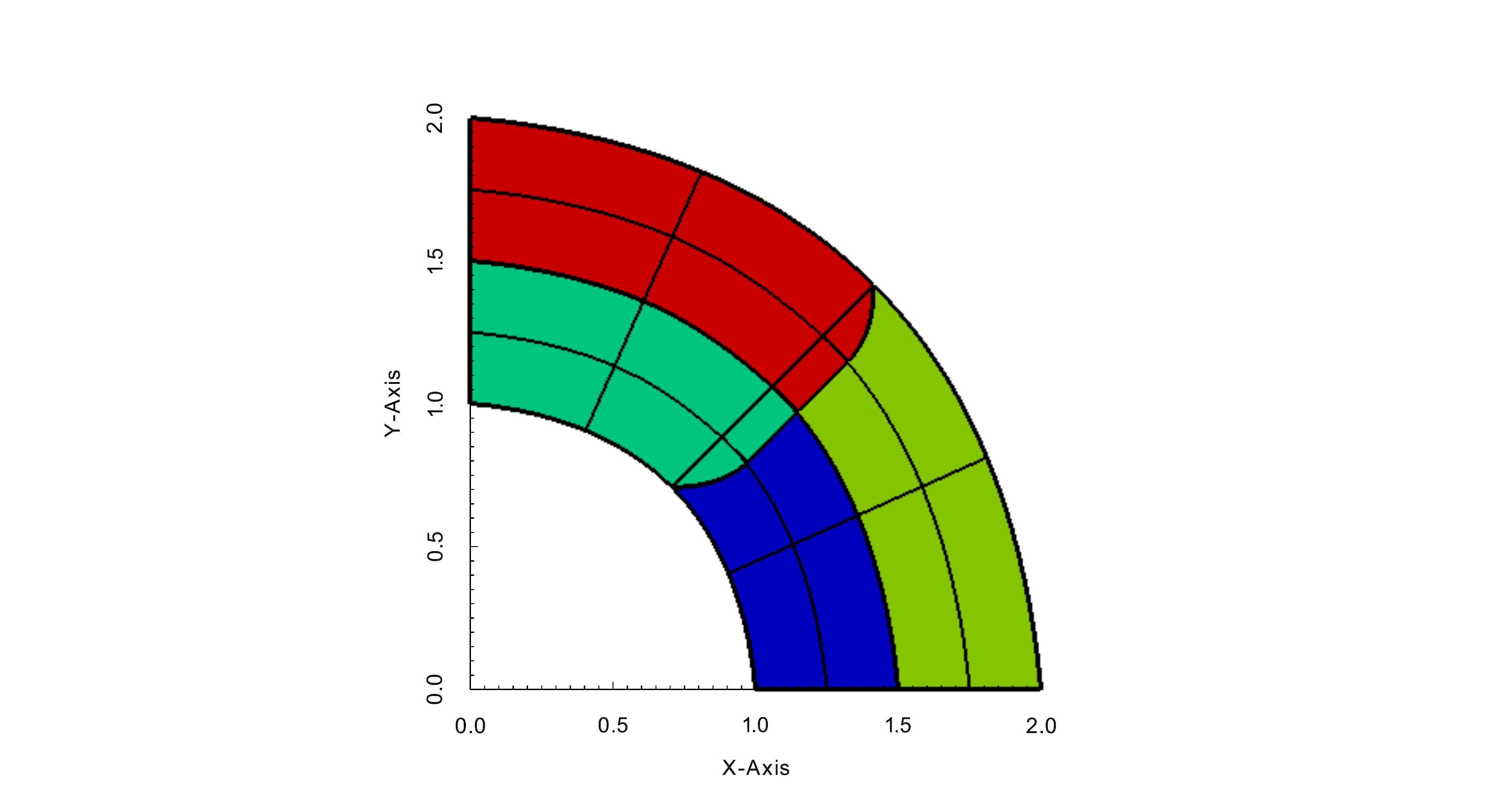}}
 \subfigure[]{\includegraphics[width=6.0cm, height=5.75cm]{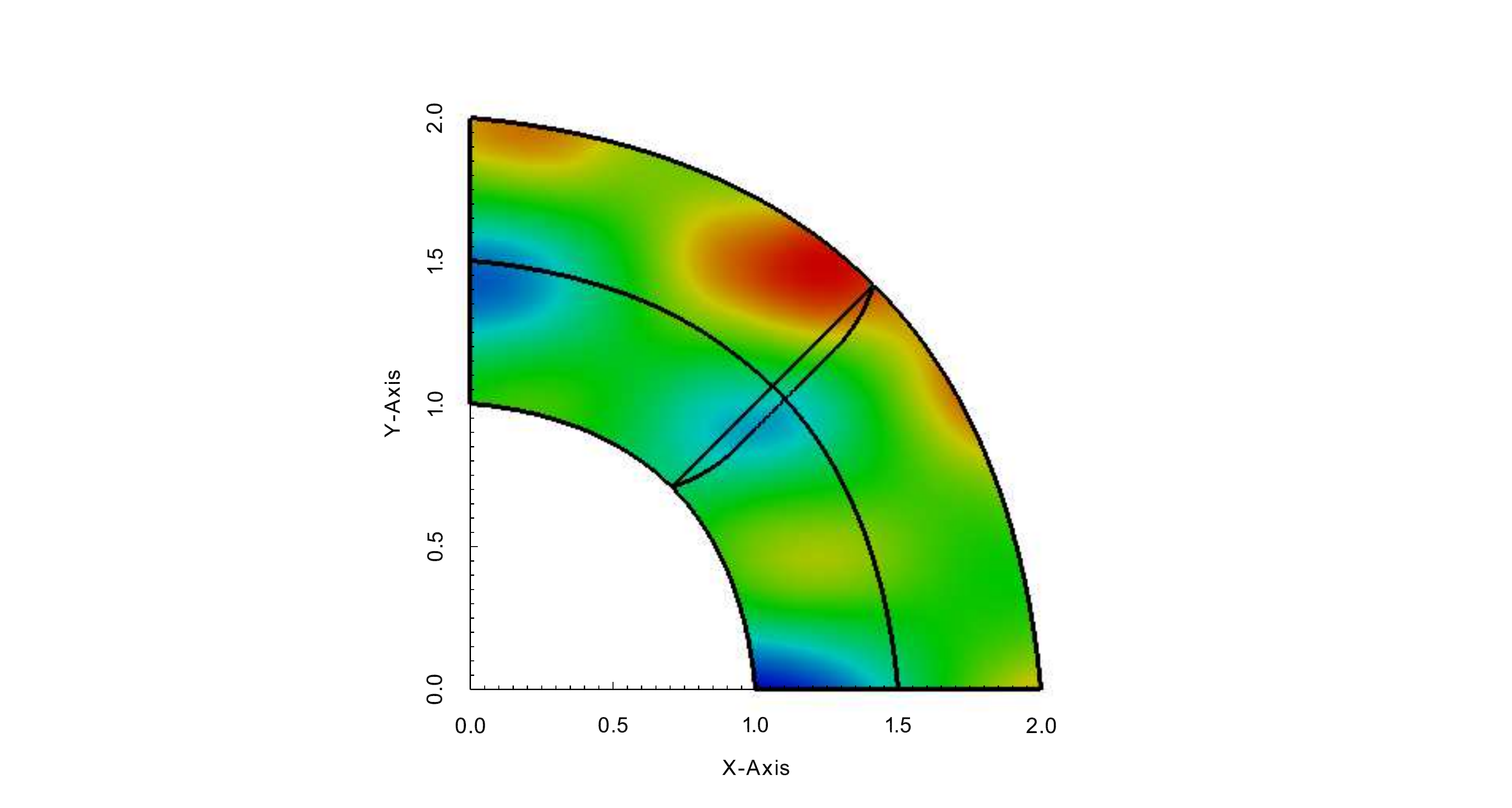}}
  \subfigure[]{\includegraphics[width=5.cm, height=5.75cm]{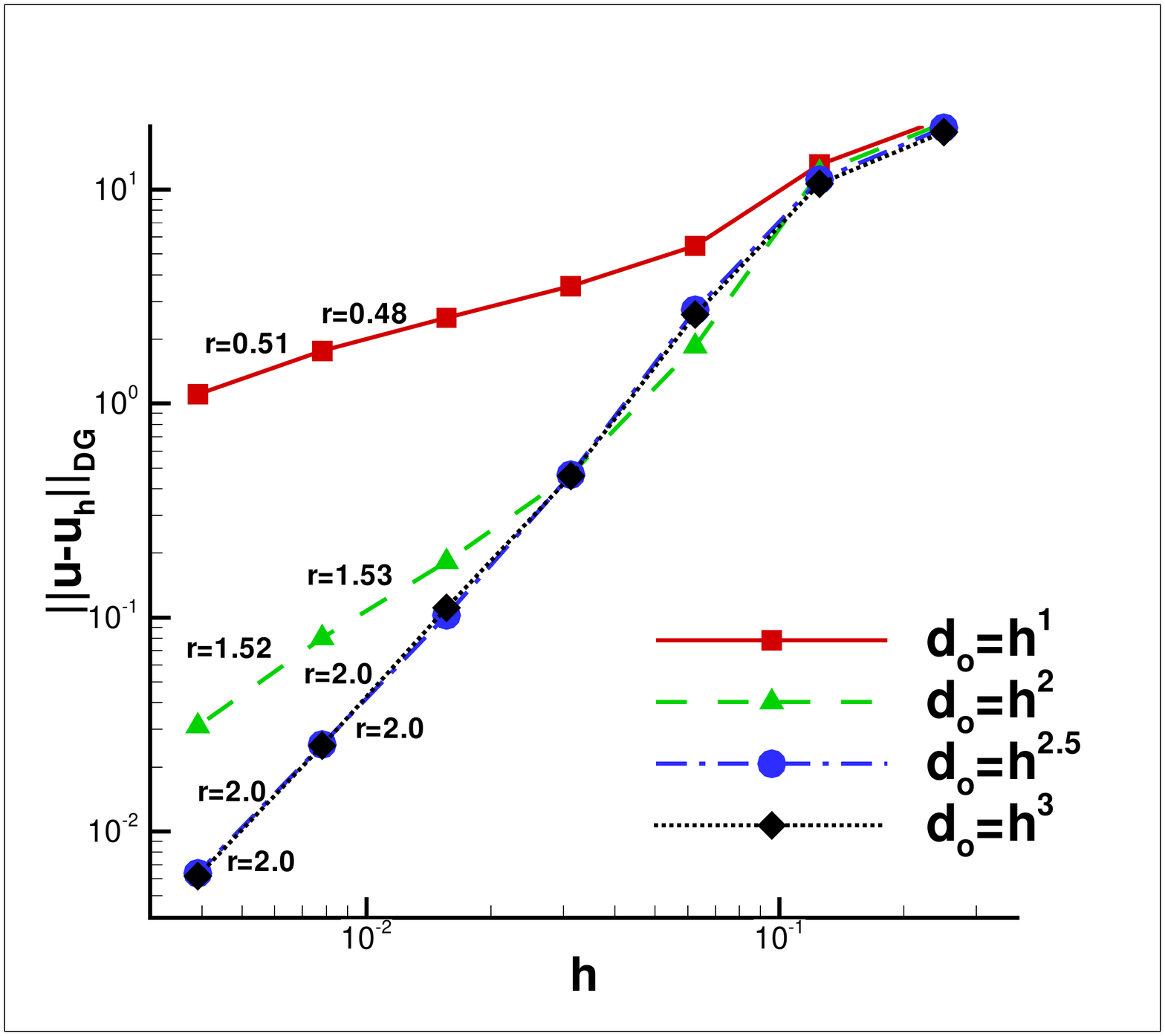}}
  \end{subfigmatrix}
   \caption{Example 3: (a) The overlapping patches $\Omega_i^*$  and the multiple curve boundary of the overlapping region, (b) The contours of $u^*_h$ on every  $\Omega_i$ computed on the second mesh level,     (c) The convergence rates for the 4 choices of $\lambda$. }
   \label{Fig_MultipatchTest}
 \end{figure}

 \subsection{Three-dimensional numerical examples}
 As a final example, we consider a  three-dimensional test.  
 The domain $\Omega$ has been constructed by  a straight  prolongation to the $z$-direction  of a two dimensional (curved) domain,
  see Fig. \ref{Fig3_3DTest_3}(a).
 The two physical domains $\Omega_1$ and $\Omega_2$ have the physical interface $F_{12}$ consisting of all points $(x,y,z)$ such that  $-1\leq x\leq 0,\,x+y=0$ and $0\leq z \leq 1$, see Fig. \ref{Fig3_3DTest_3}(a).
  The knot vector in $z$-direction  is simply $\Xi_i^3=\{0 , 0 ,0 , 0.5 ,1, 1, 1\}$ with $i=1,2$.
  We solve the problem  using  matching meshes, as depicted in  Fig. \ref{Fig3_3DTest_3}(a).  The B-spline parametrizations of these domains are constructed by adding a third component to the control points with the following values $\{0,0.5,1\}$. 
   The completed knot vectors $\mathbf{\Xi}_{i=1,2}^{k=1,2,3}$  together with the associated control nets can be found in
   G+SMO library in the file \verb+bumper.xml+. 
  The overlap region is artificially constructed by moving only the interior control points located at the interface into the normal direction $n_{F_{12}}$ of the related interface $F_{12}$. 
  Due to the fact that the  overlap has to be inside of the domain, we have to provide cuts though the domain in order to visualize them, cf.  Fig.~\ref{Fig3_3DTest_3}(b).
  The Dirichlet boundary conditions $u_D$ and the right hand side $f$, see (\ref{0}), are chosen such that the exact solution is 	
  	\begin{align}\label{NE_3}
  	 	                  u(x,y,z)=
 	                  \begin{cases}
                             \sin(\frac{\pi}{2}(x+y)) & \text{if } (x,y)\in \Omega_1, \\
                              e^{\sin(x+y)} & \text{if} {\ }(x,y)\in \Omega_2.
                          \end{cases}
                          \end{align} 
  	with diffusion coefficient $\rho=\{1,\pi/2\}$. Note that the interfaces conditions (\ref{Interf_Cond}) are satisfied. 
  	The two physical subdomains, the initial matching meshes and the 
  	exact solution are illustrated in Fig.~\ref{Fig3_3DTest_3}(a).   We construct an overlap region with $d_o=0.5$ and solve the
  	problem using $p=2$ B-spline functions. In Fig.~\ref{Fig3_3DTest_3}(b),  we show the domain meshes 
  	$T^{(i)}_{h_i,\Omega_i^*}\,,i=1,2$, the overlapped  meshes in $\Omega_{o12}$ and 
  	we plot the contours of the produced solution $u^*_h$ for the interior plane $z=0.5$.  We can see that, 
  	both faces of $\partial \Omega_{o12}$ are not parallel to the Cartesian axes. 
  	{Moreover, we point out that the problem has been solved using  non matching meshes on
  	the overlapping interfaces.}
  	We have computed the convergence rates for four different values  $\lambda\in\{1,2,2.5,3\}$ related
  	to the overlapping region width $d_o=h^\lambda$. The results of the computed rates are plotted in Fig.~\ref{Fig3_3DTest_3}(c). We observe from the plots that the  rates $r$ are  in agreement with the   rates  predicted by the theory, see estimate (\ref{mainErrEstm_0}) and Table \ref{table_value_r}. 
  \begin{figure}
   \begin{subfigmatrix}{3}
\subfigure[]{\includegraphics[width=5.450cm, height=5.475cm]{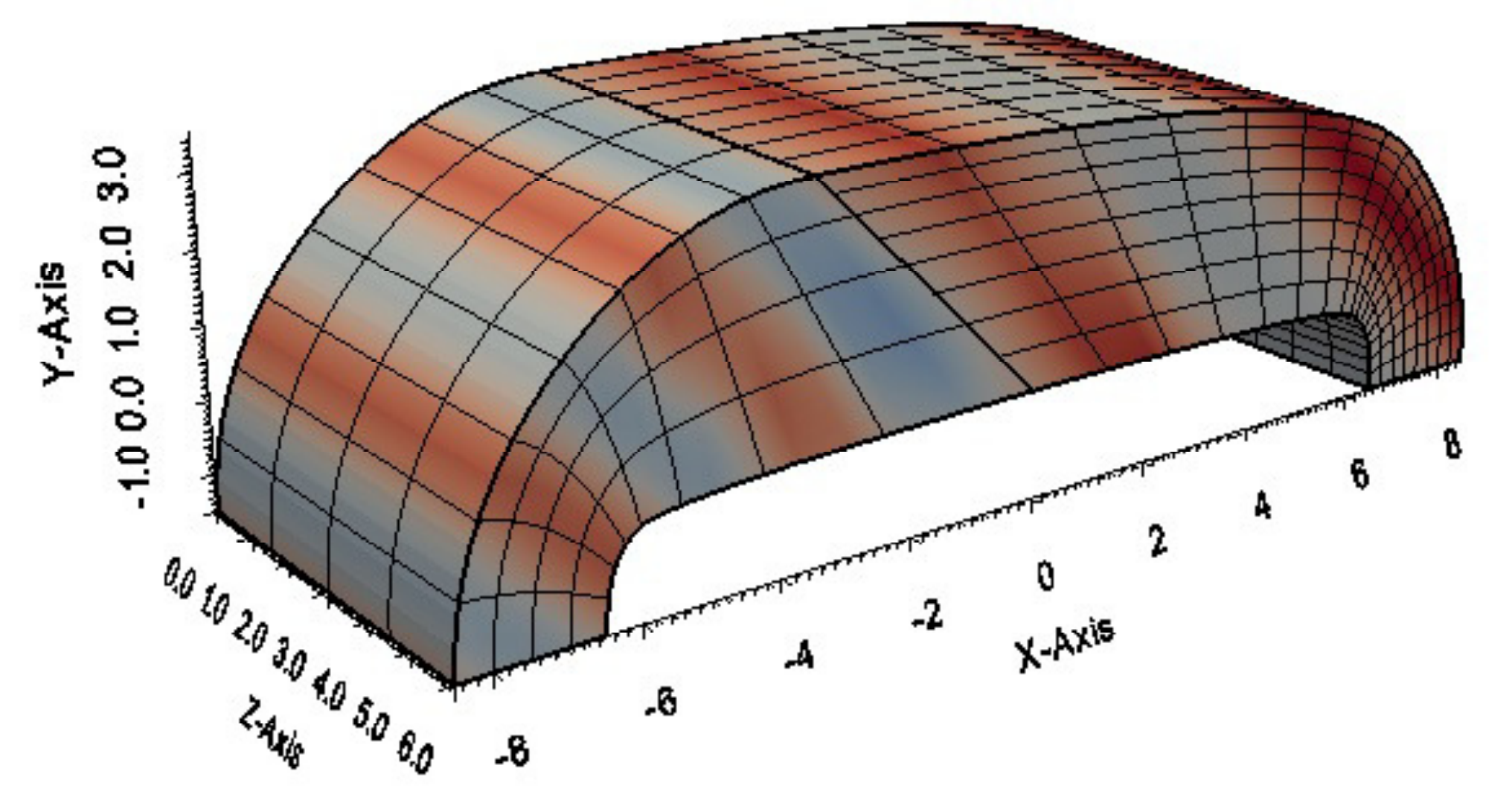}}
 \subfigure[]{\includegraphics[width=5.450cm, height=5.550cm]{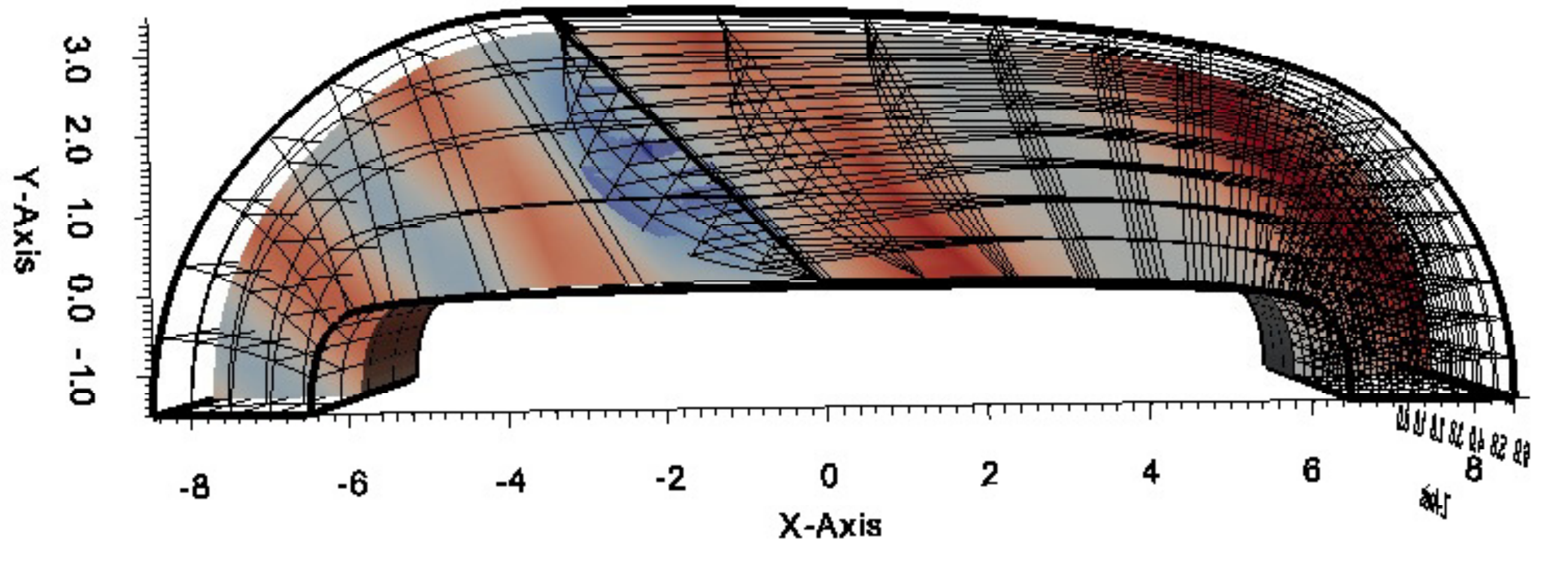}}
\subfigure[]{\includegraphics[width=5.450cm, height=5.550cm]{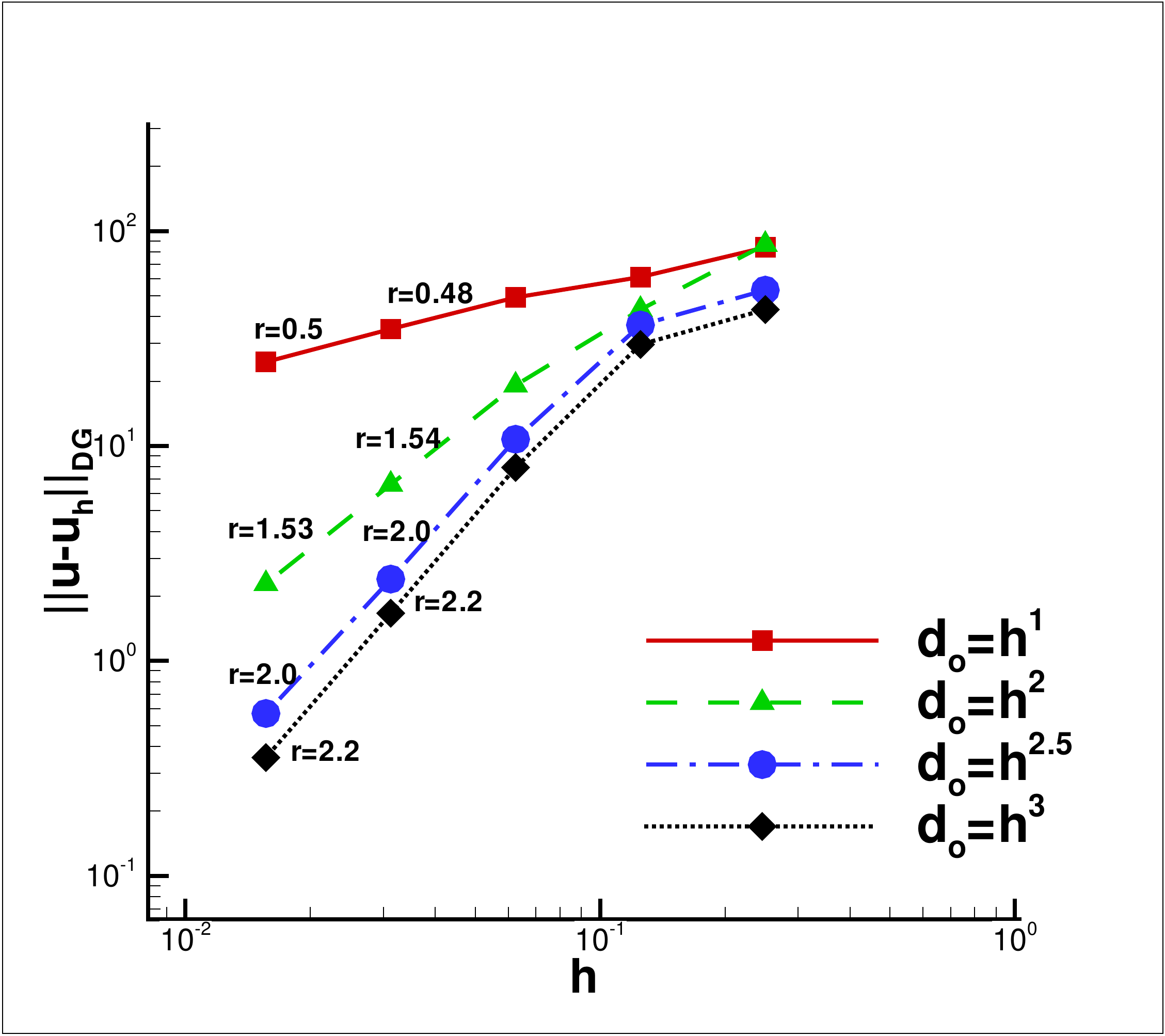}}
  \end{subfigmatrix}
   \caption{Example 4, $\Omega\subset \mathbb{R}^3$: (a) The physical patches with an initial coarse mesh and the contours of 
                          the exact solution, 
                       (b) The contours of $u^*_h$ computed on $\Omega_1^*\cup\Omega_2^*$ with $d_o=1.5$, 
                       (c) Convergence rates $r$ for the four values of $\lambda$.}
   \label{Fig3_3DTest_3}
 \end{figure}

\section{Conclusions}
In this article,  we have proposed  and analyzed a DG-IGA scheme for discretizing linear,
second-order, diffusion problems 
on IGA  multipatch representations with small   overlapping regions. 
This type of dmultipatch representations   lead   to the use of  different
diffusion coefficients on the  overlapping patches.
Auxiliary problems were introduced in every  patch
 and DG-IGA methodology applied for discretizing
 these problems.  
The    normal fluxes on the overlapped interior faces  were appropriately modified using Taylor expansions, 
and these fluxes were further used to construct numerical fluxes 
in order to couple  the associated discrete DG-IGA  problems.  
The method were successfully applied  
to the discretization of the diffusion
problem in cases with     complex   overlaps. 
A priori error estimates in the  DG-norm 
were shown in terms of 
the mesh-size $h$ and the maximum width $d_o$ of the overlapping regions.  
The estimates  were confirmed
by solving several two- and three- dimensional test problems with known exact solutions. 
The theoretical estimates 
 were also confirmed by  performing numerical tests   using non-matching grids on the overlapping faces.

 \section*{Acknowledgments}
The authors wish to thank Prof. Ulrich Langer, Prof. Bert J\"uttler and  Prof.  Dirk Pauly for many interesting discussions. 
This work was supported by the Austrian Science Fund (FWF) under the grant NFN S117-03 and  W1214-N15, project DK4.

\section{Appendix.}
\subsection{A bound for the extra non-consistent term. }
\label{Apend_u_minus_u*}
\textcolor{black}{
Comparing the relations given in  (\ref{0_PetrVF}) and (\ref{A_VaraForm_u_Omega*})
we can see  that there is  an extra  term  $-(\rho_2-\rho_1)\nabla u_2^*$ in $\Omega_{o21}$, 
which is a non consistent  term.  We derive   below a bound  for this term. \\
Let $\phi\in H^{1}_{0}(\Omega_2^*)$. By a simple computations on the forms in (\ref{Artf_Vart_Problm}), we have that
\begin{alignat}{2}\label{ConsiErr_1}
\nonumber
a_{2}^{*}(u_2^*,\phi_h)=&\int_{\Omega_{o21}}\rho_1\nabla u_2^*\cdot \nabla \phi\,dx 
                   + \int_{\Omega_{2}}\rho_2\nabla u_2^*\cdot \nabla \phi\,dx  
    -\int_{\partial \Omega_2\cap\partial \Omega}\rho_2\nabla u_2^*\cdot n_{\partial \Omega_2} \phi\,d\sigma\\
 -& \int_{F_{o21}} {\rho_2}\nabla u_2^*\cdot n_{F_{o21}}\phi\,d\sigma 
 =\int_{\Omega_{o21}}(\rho_1-\rho_2)\nabla u_2^*\cdot \nabla \phi\,dx + l_{2,f}^{*}(\phi). 
  \end{alignat}
 On the other hand, under the Assumption \ref{Assumption1}, we immediately have that
  \begin{alignat}{2}\label{ConsiErr_2}
  a_{o,2}(u,\phi_2) = & \int_{\Omega_{o21}}\rho_1\nabla u\cdot\nabla \phi\,dx +
                       \int_{\Omega_2}\rho_2\nabla u\cdot\nabla \phi\,dx \\
                       \nonumber
        -&\int_{F_{o21}}  \rho_1 \nabla u  \cdot n_{F_{o21}}\phi\,d\sigma-
        \int_{\partial \Omega_2^*\cap\partial\Omega}  \rho_2 \nabla u  \cdot n_{\partial{\Omega_2}}\phi\,d\sigma
                 = l_{2,f}^{*}(\phi).
 \end{alignat}
 Subtracting (\ref{ConsiErr_2}) from (\ref{ConsiErr_1}) and using $\phi|_{\partial\Omega_2^*}=0$
 we  obtain
  \begin{alignat}{2}\label{ConsiErr_2_a}
  	\begin{split}
  	\int_{\Omega_{o21}}\rho_1\nabla(u_2^*-u)\cdot\nabla \phi\,dx+ 
  	\int_{\Omega_{2}}\rho_2\nabla(u_2^*-u)\cdot\nabla \phi \,dx = 
  	\int_{\Omega_{o21}}(\rho_1-\rho_2)\nabla u_2^*\cdot\nabla \phi \,dx.
  	\end{split}
  \end{alignat}
  Applying integration by parts on the right 
  hand side in (\ref{ConsiErr_2_a}) and then setting $\phi=u_2^*-u$, we derive that
\begin{alignat}{2}\label{ConsiErr_2_b}
  	\begin{split}
  	\int_{\Omega_{2}^*}\rho&|\nabla(u_2^*-u)|^2\,dx = 
  	c_\rho \Big(-\int_{\Omega_{o21}}\rho_2\Delta u_2^*(u_2^*-u) \,dx +
  	\int_{F_{o12}}\rho_2 \nabla u_2^*\cdot n_{F_{o12}}(u_2^*-u)\,d\sigma\Big) \\
  	  	\leq &
  	c_\rho\Big(\int_{\Omega_{o21}}f\,(u_2^*-u) \,dx +
  	\int_{F_{o12}}\rho_2 \nabla u_2^*\cdot n_{F_{o12}}(u_2^*-u)\,d\sigma\Big) \\
  	  	 \overset{(\ref{HolderYoung})}{\leq} &
  	 	c_\rho \|f\|_{L^2(\Omega_{o21})}\|u_2^*-u\|_{L^2(\Omega_{o21})} +
  	 	 \|\rho_2 \nabla u_2^*\|_{L^2(F_{o12})}\|u_2^*-u\|_{L^2(F_{o12})} \\
          \overset{(\ref{Poincare_trace})}{\leq} &
  	 	c_{\rho} \|f\|_{L^2(\Omega_{o21})}\,\|u_2^*-u\|_{L^2(\Omega_{o21})} +
  	 	 \|\rho_2 \nabla u_2^*\|_{L^2(F_{o12})}
  	 	 \|u_2^*-u\|^{\frac{1}{2}}_{L^2(\Omega_{o21})}\|u_2^*-u\|^{\frac{1}{2}}_{H^1(\Omega_{o21})}\\
        \overset{(\ref{Poincare_trace})}{\leq} &
  	 	c_1\Big( \|f\|_{L^2(\Omega_{o21})}\,d_o\|\nabla (u_2^*-u)\|_{L^2(\Omega_{o21})} \\
  	 	&+ \|\rho_2 \nabla u_2^*\|_{L^2(F_{o12})}\,
  	 	d_o^{\frac{1}{2}}\|\nabla (u_2^*-u)\|^{\frac{1}{2}}_{L^2(\Omega_{o21})}
  	 	(d_o+1)\|\nabla (u_2^*-u)\|^{\frac{1}{2}}_{L^2(\Omega_{o21})}\\   
  	 \leq &c_2 \Big(\|f\|_{L^2(\Omega_{o21})}+	\|\rho_2 \nabla u_2^*\|_{L^2(F_{o12})}\Big)\,
  	 d_o^{\frac{1}{2}}\|\nabla (u_2^*-u)\|_{L^2(\Omega_{o21})},
  	\end{split}
  \end{alignat}
 where we have used that $0<d_o<1$. By (\ref{ConsiErr_2_b}), we can easily obtain that 
 \begin{align}
\|\rho\nabla (u_2^*-u)\|_{L^2(\Omega_2^*)} \leq c_2\,d_o^{\frac{1}{2}}
 \Big(\|f\|_{L^2(\Omega_{o21})}+	\|\rho_2 \nabla u_2^*\|_{L^2(F_{o12})}\Big),
 \end{align} 
 and this gives an estimate of the difference between the physical solution $u$ and the perturbed solution $u^*$.
 }
\subsection{Proof of the  interpolation estimate (\ref{Apend_Eq1_A}) }
\textcolor{black}{
Note that by Assumption \ref{Assu_Omega_*Shape} and the definition of (PV1) we can conclude that 
$\Omega_1=\Omega_1^*$ and $u|_{\Omega^*_1}=u^*_{\Omega^*_1}$. Hence we can construct an interpolant 
$\Pi_{1,h}^*u$ such that 
\begin{align}\label{Apendix_0}
\begin{split}
	\Big(\nabla(u_1-\Pi_{1,h}^*u_1)\|^2_{L^2(\Omega^*_1)} + 
	h\|\nabla(u_1-\Pi_{1,h}^*u_1)\|^2_{L^2(F_{o12})} \quad \\
	\quad+ 
	\frac{1}{h}\|(u_1-\Pi_{1,h}^*{u}_1)\|^2_{L^2(F_{o21})}\big)^{\frac{1}{2}} \leq 
	C_1 h^{\min(\ell-1,p)} \|u\|_{H^\ell(\Omega_1^*)}.
	\end{split}
\end{align}
Next we show an interpolation estimate for $u$ on $\Omega_2^*$. Let us  denote $D_1=\Omega_{o21}$ and $D_2=\Omega_2$. 
Let the extension operator $E_i:H^\ell(D_i)\rightarrow H^\ell(\Omega^*_i),\,i=1,2,$ 
such that for each $v\in H^\ell(D_i)$ it holds
(i) $(E_iv)|_{D_i} = v$ and (ii) $\|E_i v\|_{H^\ell(\Omega_i^*)} \leq C_{E_i} \|v\|_{H^\ell(D_i)}$, where the constant 
$C_{E_i}$ depending only on $D_i$ and $\Omega_i^*$, see \cite{Evans_PDEbook}. 
We recall  the B-spline interpolation operator$ \Pi_h^*$ given in (\ref{Intep_Est_1}) and define 
\begin{align}\label{Apendix_1}
	\Pi_{h}^* \tilde{v} :=(\Pi_{1,h}^*\tilde{v}_1, \Pi_{2,h}^* \tilde{v}_2),
	\quad \text{where}{\ }\Pi_{i,h}^*  \tilde{v}_i:=\Pi_{i,h}^* E_i v.    
\end{align}
Recalling $u_i=u|_{\Omega_i}$ and using  the properties of the extension operator and (\ref{Intep_Est_1}) we have
\begin{subequations}\label{Apendix_3}
\begin{align}\label{Apendix_3a}
	\|\nabla(u_1 - \Pi_{1,h}^* \tilde{u}_1)\|_{L^2(\Omega_{o21})} \leq 
	\|\nabla(E_1u - \Pi_{1,h}^* \tilde{u}_1)\|_{L^2(\Omega_1^*)} 
	\leq C_{intp} C_{E_1}\, h^s \|u_1\|_{H^{\ell}(\Omega^*_{1})},  
\end{align}
\text{and}
\begin{align}\label{Apendix_3b}
	\|\nabla(u_2 - \Pi_{2,h}^* \tilde{u}_2)\|_{L^2(\Omega_{2})} \leq 
	\|\nabla(E_2 u - \Pi_{2,h}^* \tilde{u}_2)\|_{L^2(\Omega_{2}^*)} \leq C_{intp} C_{E_2}\, h^s \|u_2\|_{\Omega_{2}},  
\end{align}
\end{subequations}
where $s=\min(p, \ell-1)$. \\
Using the trace inequality, \cite{LT:LangerToulopoulos:2014a},  
$\|v\|^2_{L^2(F_{21})} \leq C \big(h^{-1}\|v\|^2_{L^2(\Omega_o21)} +h|\nabla v\|^2_{L^2(\Omega_{o21})} \big)$
and proceeding as in (\ref{Apendix_3}) we can show
\begin{subequations}\label{Apendix_4}
\begin{align}\label{Apendix_4a}
\big(h\|\nabla(u_1-\Pi_{1,h}^*\tilde{u}_1)\|^2_{L^2(F_{o21})}\big)^{\frac{1}{2}} \leq C_{intp} C_{E_1}\, h^s \|u_1\|_{H^{\ell}(\Omega^*_{1})}, \\
\big(\frac{1}{h}\|(u_1-\Pi_{1,h}^*\tilde{u}_1)\|^2_{L^2(F_{o21})}\big)^{\frac{1}{2}} \leq
C_{intp} C_{E_1}\,  h^s \|u_1\|_{H^{\ell}(\Omega^*_{1})}.
\end{align}
\end{subequations}
Gathering the inequalities (\ref{Apendix_0}), (\ref{Apendix_3}) and (\ref{Apendix_4a})  we can derive  (\ref{Apend_Eq1_A}). 
}

\bibliographystyle{plain}
\bibliography{ShpClc_Overl_IgA.bib}

\end{document}